\documentclass[english]{article}
\usepackage[latin9]{inputenc}
\usepackage{babel}
\usepackage{comment}
\usepackage{varioref}
\usepackage{amsmath}
\usepackage{amsthm}
\usepackage{amssymb}
\usepackage{amscd}
\usepackage{stmaryrd}
\usepackage{esint}
\usepackage{xfrac}
\PassOptionsToPackage{normalem}{ulem}
\usepackage{ulem}
\usepackage{color}
\usepackage[all]{xy}
\usepackage[nottoc,numbib]{tocbibind}
\usepackage[CJKbookmarks=true]{hyperref}
\usepackage[dvipsnames,svgnames,x11names]{xcolor}

\usepackage{todonotes}

\makeatletter

\newcommand\specialsectioning{\setcounter{secnumdepth}{-2}}
  
\DeclareFontEncoding{LGR}{}{}

\ProvideTextCommand{\~}{LGR}[1]{\char126#1}

\DeclareTextSymbolDefault{\textquotedbl}{T1}

\theoremstyle{plain}
\newtheorem{thm}{\protect\theoremname}[section]
\theoremstyle{plain}
\newtheorem*{thm*}{\protect\theoremname}
\theoremstyle{plain}
\newtheorem{lem}[thm]{\protect\lemmaname}
\theoremstyle{definition}
\newtheorem{defn}[thm]{\protect\definitionname}
\theoremstyle{definition}
\newtheorem{rem}[thm]{\protect\remarkname}
\theoremstyle{plain}
\newtheorem{prop}[thm]{\protect\propositionname}
\theoremstyle{definition}
\newtheorem*{rem*}{\protect\remarkname}
\theoremstyle{plain}
\newtheorem{cor}[thm]{\protect\corollaryname}
\theoremstyle{plain}
\newtheorem*{fact*}{\protect\factname}
\theoremstyle{definition}
\newtheorem{fact}[thm]{\protect\factname}
\newenvironment{lyxlist}[1]
	{\begin{list}{}
		{\settowidth{\labelwidth}{#1}
		 \setlength{\leftmargin}{\labelwidth}
		 \addtolength{\leftmargin}{\labelsep}
		 }}
	{\end{list}}
\providecommand*{\strong}[1]{\textbf{#1}}

\usepackage{babel}
\providecommand{\corollaryname}{Corollary}
  \providecommand{\definitionname}{Definition}
  \providecommand{\factname}{Fact}
  \providecommand{\lemmaname}{Lemma}
  \providecommand{\propositionname}{Proposition}
  \providecommand{\remarkname}{Remark}
\providecommand{\theoremname}{Theorem}

\makeatletter
\def\blfootnote{\xdef\@thefnmark{}\@footnotetext}
\makeatother

\newcommand{\Z}{\mathbb{Z}}
\newcommand{\N}{\mathbb{N}}
\newcommand{\h}{\hat}

\DeclareMathOperator{\nrp}{NRP}
\DeclareMathOperator{\sgn}{sgn}
\DeclareMathOperator{\supp}{supp}
\newcommand{\id}{\mathrm{Id}}
\newcommand{\cfnil}{CF-Nil}
\newcommand{\haar}{\mathrm{Haar}}
\usepackage{babel}
\providecommand{\corollaryname}{Corollary}
  \providecommand{\definitionname}{Definition}
  \providecommand{\factname}{Fact}
  \providecommand{\lemmaname}{Lemma}
  \providecommand{\propositionname}{Proposition}
  \providecommand{\remarkname}{Remark}
\providecommand{\theoremname}{Theorem}

\usepackage{babel}
\providecommand{\corollaryname}{Corollary}
  \providecommand{\definitionname}{Definition}
  \providecommand{\factname}{Fact}
  \providecommand{\lemmaname}{Lemma}
  \providecommand{\propositionname}{Proposition}
  \providecommand{\remarkname}{Remark}
\providecommand{\theoremname}{Theorem}

\usepackage{babel}
\providecommand{\corollaryname}{Corollary}
  \providecommand{\definitionname}{Definition}
  \providecommand{\factname}{Fact}
  \providecommand{\lemmaname}{Lemma}
  \providecommand{\propositionname}{Proposition}
  \providecommand{\remarkname}{Remark}
\providecommand{\theoremname}{Theorem}

\usepackage{babel}
\providecommand{\corollaryname}{Corollary}
  \providecommand{\definitionname}{Definition}
  \providecommand{\factname}{Fact}
  \providecommand{\lemmaname}{Lemma}
  \providecommand{\propositionname}{Proposition}
  \providecommand{\remarkname}{Remark}
\providecommand{\theoremname}{Theorem}

\usepackage{babel}
\providecommand{\corollaryname}{Corollary}
  \providecommand{\definitionname}{Definition}
  \providecommand{\factname}{Fact}
  \providecommand{\lemmaname}{Lemma}
  \providecommand{\propositionname}{Proposition}
  \providecommand{\remarkname}{Remark}
\providecommand{\theoremname}{Theorem}

\usepackage{babel}
\providecommand{\corollaryname}{Corollary}
  \providecommand{\definitionname}{Definition}
  \providecommand{\factname}{Fact}
  \providecommand{\lemmaname}{Lemma}
  \providecommand{\propositionname}{Proposition}
  \providecommand{\remarkname}{Remark}
\providecommand{\theoremname}{Theorem}

\usepackage{babel}
\providecommand{\corollaryname}{Corollary}
  \providecommand{\definitionname}{Definition}
  \providecommand{\factname}{Fact}
  \providecommand{\lemmaname}{Lemma}
  \providecommand{\propositionname}{Proposition}
  \providecommand{\remarkname}{Remark}
\providecommand{\theoremname}{Theorem}

\usepackage{babel}
\providecommand{\corollaryname}{Corollary}
  \providecommand{\definitionname}{Definition}
  \providecommand{\factname}{Fact}
  \providecommand{\lemmaname}{Lemma}
  \providecommand{\propositionname}{Proposition}
  \providecommand{\remarkname}{Remark}
\providecommand{\theoremname}{Theorem}

\usepackage{babel}
\providecommand{\corollaryname}{Corollary}
  \providecommand{\definitionname}{Definition}
  \providecommand{\factname}{Fact}
  \providecommand{\lemmaname}{Lemma}
  \providecommand{\propositionname}{Proposition}
  \providecommand{\remarkname}{Remark}
\providecommand{\theoremname}{Theorem}

\usepackage{babel}
\providecommand{\corollaryname}{Corollary}
  \providecommand{\definitionname}{Definition}
  \providecommand{\factname}{Fact}
  \providecommand{\lemmaname}{Lemma}
  \providecommand{\propositionname}{Proposition}
  \providecommand{\remarkname}{Remark}
\providecommand{\theoremname}{Theorem}

\usepackage{babel}
\providecommand{\corollaryname}{Corollary}
  \providecommand{\definitionname}{Definition}
  \providecommand{\factname}{Fact}
  \providecommand{\lemmaname}{Lemma}
  \providecommand{\propositionname}{Proposition}
  \providecommand{\remarkname}{Remark}
\providecommand{\theoremname}{Theorem}

\usepackage{babel}
\providecommand{\corollaryname}{Corollary}
  \providecommand{\definitionname}{Definition}
  \providecommand{\factname}{Fact}
  \providecommand{\lemmaname}{Lemma}
  \providecommand{\propositionname}{Proposition}
  \providecommand{\remarkname}{Remark}
\providecommand{\theoremname}{Theorem}

\makeatother

\providecommand{\corollaryname}{Corollary}
\providecommand{\definitionname}{Definition}
\providecommand{\factname}{Fact}
\providecommand{\lemmaname}{Lemma}
\providecommand{\propositionname}{Proposition}
\providecommand{\remarkname}{Remark}
\providecommand{\theoremname}{Theorem}
\title{Strictly ergodic distal models and a new approach to the Host-Kra factors}
\newcommand{\Addresses}{{
  \bigskip
  \footnotesize
 
   \textsc{Institute of Mathematics, Polish Academy of Sciences, ul. \'{S}niadeckich 8, 00-656 Warszawa, Poland.}\par\nopagebreak
  
  Yonatan Gutman: \texttt{y.gutman@impan.pl}
   
  \medskip

  \textsc{ School of Mathematical Sciences, Xiamen University, Xiamen, Fujian 361005, P.R. China;}
  
 Zhengxing Lian: \texttt{lianzx@mail.ustc.edu.cn; lianzx@xmu.edu.cn}

}}

\author{
  Yonatan Gutman and Zhengxing Lian
}

\begin{document}
\maketitle
\blfootnote{Y. G. and Z.L. were partially supported by the National Science Centre (Poland) grant 2016/22/E/ST1/00448. Y.G. was partially supported by the National Science Centre (Poland) grant
2020/39/B/ST1/02329. Z. L. was partially supported by the Xiamen Youth Innovation Foundation No. 3502Z20206037, the Presidential Research Foundation of Xiamen University No. 20720210034 and the NNSF of China (Award No. 12101517).}
\begin{abstract}
\textit{Cocycles} are a key object in Antol\'{i}n Camarena and Szegedy's theory of nilspaces. We introduce variants of these cocycles, named \textit{nilcycles}, enabling us to give conditions which guarantee that an ergodic group extension of a strictly ergodic distal system admits a strictly ergodic distal topological model. In particular we show that if the base space is a dynamical nilspace then a dynamical nilspace topological model may be chosen for the extension.  This approach combined with a structure theorem of Gutman, Manners and Varj\'{u}  applied to the ergodic group extensions between successive Host-Kra characteristic factors gives a new proof that these factors are inverse limit of nilsystems for $\Z^k$-actions. 
\end{abstract}

\tableofcontents{}\label{content}

\section{Introduction.}
The \textit{nonconventional average}
\begin{equation}\label{nonconventional averages}
\lim_{N\rightarrow\infty}\frac{1}{N}\sum_{n=0}^{N-1} f_1(T^nx)\cdots f_k(T^{kn}x),
\end{equation}
for a measure-preserving system $(X,\mathcal{X},\mu,T)$ and $f_1,\ldots,f_k\in L^\infty(X,\mu)$ became prominent due to Furstenberg's proof \cite{F77} of Szemer\'{e}di's theorem \cite{Szemeredi1975sets}. Furstenberg proved that for $f=1_A$ with $\mu(A)>0$,
$$\liminf_{N\rightarrow\infty}\frac{1}{N}\sum_{n=0}^{N-1}\int_X 1_A(x)1_A(T^nx)\cdots 1_A(T^{kn}x) d\mu(x)>0,$$
which implies Szemer\'{e}di's theorem by what is known today as the Furstenberg correspondence principle \cite{F77}. About 30 years after Furstenberg's original proof, in what can only be described as a tour de force, Host and Kra proved that Equation (\ref{nonconventional averages}) converges in norm \cite{HK05} for all $k\geq 1$.\footnote{This was also proved independently by Ziegler \cite{Z07} and was known for $k=3$ for totally ergodic systems  by earlier work of Conze and Lesigne \cite{CL87}. Almost sure convergence for general $k$ is still unknown.}
Host and Kra's proof uses the \textit{characteristic factors} approach pioneered by Furstenberg and Weiss \cite{FW96}. A factor $(X,\mathcal{X},\mu,T)\rightarrow (Y,\mathcal{Y},\nu,T)$ is called characteristic for (\ref{nonconventional averages}) if
$$\frac{1}{N}\sum_{n=0}^{N-1}T^{n} f_1\cdots T^{kn}f_k- \frac{1}{N}\sum_{n=0}^{N-1} T^{n}\mathbb{E}(f_1|\mathcal{Y})\cdots T^{kn}\mathbb{E}(f_k|\mathcal{Y})$$ converges to zero in $L^2$-norm.
The striking discovery of Host and Kra was that (\ref{nonconventional averages}) possesses a characteristic factor $Z_{k-1}$ which is an inverse limit of  nilsystems. Thus a problem in measure theory gives rise to smooth structure.

The construction of a characteristic factor $Z_k$ which is an inverse limit of  nilsystems  is complicated and is the technical heart of both of \cite{HK05} and \cite{HK2018}. The main goal of this article is to put forward a conceptual  simplification of the main step of this construction. This approach was announced in \cite{G15(2)}. 

The $L^2$-convergence problem of the nonconventional average can be reduced to the ergodic case by using ergodic decomposition (\cite[p. 337]{HK2018}). Let us thus assume $(X,\mathcal{X},\mu,T)$  is ergodic. In \cite{HK05} the Host and
Kra measures $\mu^{[k]}$ and the Host-Kra-Gowers seminorms $
|||\cdot|||_{k}$ are introduced. These
are defined in the following way.  Denote $\mu^{[0]}=\mu$.  Let
$\mathcal{I}_{T}^{[k]}$ be the $\underbrace{T\times\cdots\times  T}_{2^k}$  invariant
$\sigma$-algebra of the product $\sigma$-algebra of
$(X^{[k]}:= X^{\{0,1\}^{k}},\mu^{[k]})$.
Define $\mu^{[k+1]}$ to be the relative independent joining of two
copies of $\mu^{[k]}$ over $\mathcal{I}_{T}^{[k]}$. Now introduce for
real-valued  $f\in L^{\infty}(\mu)$ the seminorms\footnote{The fact that these expressions are seminorms is proven in  \cite[Section 3.5]{HK05}. }:
$$
|||f|||_{k}=(\int_{X^{[k]}}\prod_{v\in\{0,1\}^{k}}f(x_v)d\mu^{[k]})^{1/2^{k}},\,\,
k=1,2,\ldots
$$
The factor $(X,\mu,T) \rightarrow (Z_{k-1},\mu_{k-1},T)$ is identified as the
maximal  factor for which $|||f|||_{k}$ is a norm (\cite[Lemma 4.3]{HK05}). In \cite{HK05}, Sections 6.3, 7, 9 and 10 (see also \cite[Chapters 18-20]{HK2018})  are devoted to an involved analysis of the factor map $Z_{k} \rightarrow Z_{k-1}$ which eventually establishes the structure of $Z_{k}$.

The main idea of our approach is to find a topological model for
$(Z_k,\mu_k,T)$ which is a dynamical nilspace. The concept of
\emph{nilspaces}, originating in Host and Kra's
\emph{parallelepiped structures} of \cite{HK08}, was introduced by  Antol\'{i}n
Camarena and Szegedy in \cite{CS12} (see also \cite{candela2016cpt_notes,candela2016alg_notes}). A (compact) nilspace is a compact metric space $X$ together
with  closed collections of \textit{cubes} $C^n(X)\subseteq X^{\{0,1\}^n}$,
$n=1,2,\ldots$, satisfying certain topological axioms. There is a canonical way of associating to a minimal action by an arbitrary group, certain nilspaces, known as \textit{dynamical nilspaces} (see Subsection \ref{subsec:Dynamical cubespaces}). According to a
structure theorem of Gutman,  Manners, and Varj\'{u} (\cite{GMVIII}) a dynamical nilspace is  an inverse limit of
nilsystems (for an alternative proof see \cite{candela2020nilspace}). Using this structure theorem one is
left with verifying topological properties of the topological model.
This task is much easier than proving directly the nilsystems inverse
limit representation and  results with a major simplification. The structure theorem for $\mathbb{Z}$-dynamical nilspaces was first proven by Host, Kra and  Maass (\cite[Theorem 1.2]{HKM10}), however the proof uses the full force of the structure theorem of \cite{HK05}. As our goal is to give a new proof of this theorem, we cannot use \cite{HKM10}. The proof in \cite{GMVIII} is independent of both \cite{HKM10} and \cite{HK05} and is essentially topological. Moreover it holds in broader generality which enables us to prove the  structure theorem in the generality of finitely generated abelian group actions, in particular $\Z^k$-actions. 

By an elegant and relatively short proof it is possible to show that  $Z_{k} \rightarrow
Z_{k-1}$ is an abelian  group extension (\cite[Section 6]{HK05}, \cite[Chapter 18]{HK2018}). In particular there is a natural measurable identification
$Z_k=Z_{k-1}\times A$ where $A$ is a compact abelian  group. However it is
important to note that in general $Z_{k-1}\times A$ equipped with the
product topology is not a topological model  as $T$ fails to be
continuous w.r.t this representation. It is thus  natural  to attempt
to  identify the  appropriate topological model and as mentioned above
our approach is to find a nilspace topological model.

To show that $Z_k$ admits a nilspace (topological) model we assume by induction that
$Z_{k-1}$ is a nilspace and as such has well-defined cubes $C^n(Z_{k-1})\subset Z_{k-1}^{\{0,1\}^n}$,
$n=1,2,\ldots$. Inspired by the nilspace-theoretic notion of \textit{cocycle} introduced in \cite{CS12}, we construct a special function $\rho:C^{k+1}(Z_{k-1})\rightarrow A$ which we name a \textit{nilcycle}. The key geometric-algebraic property of a nilcycle $\rho$ is that for \textit{almost all pairs} (w.r.t.\  a natural measure defined in Definition \ref{def:meaure on glueing pair}) of cubes $c_1,c_2\in C^{k+1}(Z_{k-1})$ which have a common face (e.g.\  $(c_1)_{\{1\}\times \{0,1\}^k}=(c_2)_{\{0\}\times \{0,1\}^k}$), the value of $\rho$ on the cube resulting from glueing on the common face
$$c_1\|c_2:=((c_1)_{\{0\}\times \{0,1\}^k},(c_2)_{\{1\}\times \{0,1\}^k})$$
satisfies $\rho(c_1\|c_2)=\rho(c_1)+\rho(c_2)$. This step is motivated by the fact that if we were to know that $Z_k$ admits a nilspace model then it would be possible (and in fact not hard) to construct a nilcycle which furthermore satisfies the nilcycle properties \textit{everywhere} instead of \textit{almost everywhere}.

Our approach depends on the surprising fact that a nilcycle determines a nilspace model $\hat{Z}_k$ for $Z_k$. The construction is not entirely straightforward and uses the concept of
\textit{function bundles} induced by \cite[Section 3.3]{CS12} (see also \cite[Section 2.3]{candela2016cpt_notes}) modeled on the concept of \textit{Banach bundles} (see \cite[Chapter 10]{Fell1977}). Specifically, we consider the 2-parameter family of functions $\rho_x+a:C^{k+1}_x(Z_{k-1})\rightarrow A$ for $a\in A,x\in X'\subset  Z_{k-1}$, where $X'$ is some (carefully chosen) set of full measure and
$$C^{k+1}_x(Z_{k-1}):=\{c\in C^{k+1}(Z_{k-1}):c_{\vec{0}}=x\}, \rho_x:=\rho|_{C^{k+1}_x(Z_{k-1})},$$
and take the closure $\hat{Z}_k$ w.r.t.\ a topology related to the natural continuous projection map $\mathfrak{p}_0:C^{k+1}(Z_{k-1})\rightarrow Z_{k-1}:c\rightarrow c_{\vec{0}}$. Crucially, $\hat{Z}_k$ is shown to be compact and using the properties of $\rho$ one may define a homeomorphism $\hat{T}$ such that $(\hat{Z}_k,\hat{T})$ is a (uniquely
 ergodic) topological nilspace model for $(Z_k,\mu_k,T)$. As remarked above this establishes the desired structure result. 
 
 We notice that recently a new and different approach to Host-Kra factors through nilspaces was put forth in \cite{CS2018}.

  The above considerations naturally lead us to consider the question when an ergodic (measurable) abelian group extension $Y$ of a strictly ergodic distal system $X$ (thus automatically measurable distal) admits a strictly ergodic topological model, that is, a uniquely ergodic and minimal model. Note that  $(\hat{Z}_k,T)$ is such a model, as being an inverse limit of nilsystems, it is distal which in turn, together with unique ergodicity, implies minimality.

    On one hand the Jewett-Krieger theorem \cite{jewett1970prevalence, krieger1972unique} implies that any ergodic system has a strictly ergodic model. On the other hand Zimmer's theorem implies that \textit{any} ergodic measure of a minimal distal system induces a measurable distal system (\cite{zimmer1976extensions}). Nevertheless the question if an ergodic measurable distal system has a strictly ergodic distal model\footnote{Salehi \cite{Salehi91} provides an abundance of examples of strictly ergodic distal systems.} has a negative answer. Indeed the classical Morse system $(M,S),M\subset\{0,1\}^\mathbb{Z}$
    which is known to be strictly ergodic is an abelian group extension of an odometer and thus measurable distal (\cite[Chapter III, 2.27]{dV93}); however by \cite[Proposition 13.5]{GW06} it admits \textit{no} strictly ergodic distal model. 
    In particular it follows that not every ergodic m.p.s.\ which is an abelian group extension of a strictly ergodic system admits a strictly ergodic distal model. Yet Lindenstrauss \cite[Claim 5.5]{L99} showed that such a model exists if the group extension is by a \textit{connected group}. We now state our main theorem which deals with a very particular setting in which one is able to show the existence of a strictly ergodic distal model of a measurable abelian group extension $(Y,\nu,G)\rightarrow (X,\mu,G)$, through a different method than the one employed by Lindenstrauss, in particular not assuming connectivity for the extension group (see also Remark \ref{rem:clarifications} 
    for clarification of terms in the theorem). It is by verifying this condition that we demonstrate the existence of nilspace models for the Host-Kra factors.

\begin{thm}
\label{thm:main theorem} (Main Theorem) Let $G$ be a  finitely generated abelian group (e.g., $G=\Z^k$, $k\in \N$). Suppose the measurable group
extension

$$
Y=(X\times A,\nu=\mu\times  m_{\haar(A)},G)\xrightarrow{\pi}(X,\mu,G)
$$
satisfies:
\begin{enumerate}
    \item $Y$ is an ergodic m.p.s.\ and $(X,G)$ is a distal $(2k+2)$-cube uniquely ergodic system;
    \item $A$ is a compact metrizable abelian group and there exists a nilcycle $\rho:C^{k+1}(X)\rightarrow A$ of degree $k$ 
 w.r.t.\  the extension.
\end{enumerate}
Then the extension $\pi$  has a topological model $\h{\pi}:(M,G)\rightarrow (X,G)$ such that  $(M,G)$ is a strictly ergodic distal t.d.s. and  such that $\h{\pi}$
is a topological group extension of $X$ by $A$, which is a fibration of order at most $k$. In particular, $\nrp_{G}^{[k]}(M\rightarrow X)=\triangle:=\{(x,x)|x\in X\}$.
\end{thm}
\begin{rem}\label{rem:clarifications}
We clarify several notions appearing in the statement of Theorem \ref{thm:main theorem}:
\begin{enumerate}
    \item The $m$-cube uniquely ergodic systems which constitute a natural subclass of strictly ergodic systems are defined in Definition \ref{k-cube uniquely ergodic} and are studied in the paper \cite{GL2019(2)}. The reason why $m=2k+2$ is required in the theorem is explained by Proposition \ref{uniquely ergodic for T(X)}.
    \item
Nilcycles are defined in Definition \ref{def:nilcycle}. Fibrations are defined in Section \ref{subsec:Cube-space,-nilspace}.
     \item
    The equivalence relation $\nrp_G^{[k]}(M\rightarrow X)$ is defined in Definition \ref{NRP(M-G)} and is extensively studied in the article \cite{GL2019}.
\end{enumerate}
\end{rem}
We note that Theorem \ref{thm:main theorem} is applied in Section \ref{new proof of the Host-Kra structure theorem} in order to show the existence of nilspace models for Host-Kra factors in the generality of finitely generated abelian group actions, in particular $\Z^k$-actions.

\subsection*{Acknowledgements.}
The present work relies essentially on
the groundbreaking papers \cite{HK05,CS12}.
The first author is grateful to Bernard
Host and Bal\'{a}zs Szegedy who introduced
him to these articles and helped him to
take his first steps in the related
theory. We are grateful to Tomasz Downarowicz for helpful discussions. We are grateful to Bingbing Liang for a careful reading of a previous version. We are thankful to the two reviewers for a detailed reading and many useful comments. 

\section{Preliminaries.}

In this section, we introduce some basic definitions and properties.
For extensive background we recommend \cite{G03}, \cite{HK05},  \cite{CS12}, \cite{candela2016cpt_notes}, \cite{candela2016alg_notes}, \cite{GMVI}, \cite{GMVII}, \cite{GMVIII} and \cite{HK2018}.

\subsection{Dynamical background.}\label{subsec:Dynamical-background.}

Throughout this article we assume every topological space to be metrizable. A \strong{topological dynamical system (t.d.s.)}\ is a pair $(X,G)$,
where $X$ is a compact metric space and $G$ is finitely generated abelian group (e.g., $G=\Z^k$, $k\in \N$), equipped with the discrete topology acting on $X$ continuously (i.e., the pairing $G\times X\rightarrow X$ is continuous). The action of $g\in G$ on $x\in X$ is denoted by $gx$. Denote by $\mathcal{M}(X)$ ($\mathcal{M}_G(X)$) the set of ($G$-invariant) Borel probability measures of $X$. The \strong{orbit $\mathcal{O}(x)$ }of $x\in X$ is the set $\mathcal{O}(x)=\{tx:t\in G\}$. A t.d.s.\ is \strong{minimal} if $\overline{\mathcal{O}(x)}=X$ for all $x\in X$. A t.d.s.\ $(X,G)$ is \strong{distal} if for a compatible metric $d_{X}$
of $X$, for any $x\neq y\in X$, $\inf_{g\in G}d_{X}(gx,gy)>0$. Let $H$ be a finitely generated abelian group. We say $\pi:(Y,G)\rightarrow(X,H)$ is a \strong{(topological) factor map} w.r.t.\ 
a group epimorphism $\phi:G\rightarrow H$ if $\pi$ is a continuous and
surjective map such that for any $g\in G$ and any $x\in X$, $\pi(gx)=\phi(g)\pi(x)$. We sometimes denote it by $(\pi,\phi):(Y,G)\rightarrow(X,H)$. If  $G=H$ and $\phi=\id$
, we just write $\pi:(Y,G)\rightarrow(X,G)$. A map $\pi:(Y,G)\rightarrow(X,H)$ is called a \strong{distal factor map}
if there is a compatible metric $d_{Y}$ of $Y$ such that for any $y_{1}\neq y_{2}\in Y$
with $\pi(y_{1})=\pi(y_{2})$, $\inf_{g\in G}d_{Y}(gy_{1},gy_{2})>0$.
\begin{prop}\label{prop:distality}(\cite[Chapter IV, Proposition 2.27]{dV93}, \cite[Chapter 5, Theorem 6]{A}; see also \cite[Lemma 7.11]{GGY2018})
\begin{enumerate}
    \item If $\pi:(X,G)\rightarrow(Y,H)$ is a factor map and $(X,G)$ is distal, then $(Y,H)$ is distal.
    \item
     If $(X,G)$ is distal and $k\in \N$, then the  product action $(X^k,G^k)$ is distal. If $H\subset G^k$ is a subgroup and $Y\subset X^k$ is a closed $H$-invariant subset  then the induced action $(Y,H)$ is distal.
\end{enumerate}
\end{prop}
 A t.d.s.\ $(Y,G)$ is called a \strong{topological group extension} of $(X,H)$ by a compact group $K$ if there exists a continuous action $\alpha:K\times Y\rightarrow Y$ such that the actions $G$ and $K$ commute and  for all $x,y\in X$, $\pi(x)=\pi(y)$ iff there
exists a unique $k\in K$ such that $kx = y$. 

Suppose a sequence of t.d.s.\ $\{(X_{m},G)\}_{m\in\mathbb{N}}$ and factor maps  $\pi_m:(X_m,G)\rightarrow(X_{m-1},G)$, $m\geq 2$, are given. Define the \strong{inverse limit} of $\{(X_{m},G)\}_{m\in\mathbb{N}}$ by

\[
\underleftarrow{\lim}(X_{m},G):=\{(x_{m})\in\prod_{m\in\mathbb{N}}X_{m}:\ \pi_{m+1}(x_{m+1})=x_{m}\text{ for }m\geq1\},
\]
equipped with the product topology and $G$-action given by $t(x_{m})_{m\in\mathbb{N}}=(tx_{m})_{m\in\mathbb{N}}$, $t\in G$.

Throughout this article we assume every probability space $(X,\mathcal{X},\mu)$ to be a standard Borel space. We denote by $\textrm{Aut}(X,\mathcal{X},\mu)$ the group of invertible measurable measure-preserving maps $(X,\mathcal{X},\mu)\rightarrow (X,\mathcal{X},\mu)$. A \strong{measure-preserving probability system (m.p.s.)}\ is a quadruple $(X,\mathcal{X},\mu,G)$, where $(X,\mathcal{X},\mu)$
is a probability space and $G$ is a finitely generated abelian subgroup of $\textrm{Aut}(X,\mathcal{X},\mu)$. A m.p.s.\ $(X,\mathcal{X},\mu,G)$ is \strong{ergodic}
if for every set $C\in\mathcal{X}$ such that $tC=C$ for all $t\in G$ (up to measure zero),
one has $\mu(C)=0$ or $1$. The $G$-invariant sub-$\sigma$-algebras of $\mathcal{X}$ induce \strong{(measurable) factor maps}  $(X,\mathcal{X},\mu,G)\rightarrow (Y,\mathcal{Y},\nu,G)$  and vice versa (\cite[Chapter 2.2]{G03}). Let $(Y,\mathcal{Y},\nu,G)$ be an ergodic m.p.s. A \textbf{skew-product} $(Y\times K,\mathcal{Y}\otimes\mathcal{B}, \nu\times m_{\haar},G)$ of $Y$ with a compact metrizable group $K$, equipped with its Borel $\sigma$-algebra $\mathcal{B}$, is given by the action $t(y,u)=(ty,\beta(t,y)u)$, $t\in G$, where the measurable map $\beta:G\times Y\rightarrow K$ is a \strong{cocycle}, that is, it has the property that for any $t,t'\in G$ and a.e. $y\in Y$, $\beta(tt',y)=\beta(t,t'y)\beta(t',y)$. The projection $(Y\times K,\mathcal{Y}\otimes \mathcal{B}, \nu\times m_{\haar},G)\rightarrow (Y,\mathcal{Y},\nu,G)$ given by $(y,a)\mapsto y$ is called a \textbf{(measurable) group extension} (cf. \cite[Theorem 3.29]{G03}).

When no confusion arises, we sometimes omit the $\sigma$-algebra from the notation of an m.p.s., writing $(X,\mu,G)$.

A t.d.s.\ $(X,G)$ is called \textbf{uniquely ergodic} if $|\mathcal{M}_G(X)|=1$. If in addition it is minimal then it is called \textbf{strictly ergodic}.

\begin{rem}\label{rem:strictly ergodic distal}
We note that a uniquely ergodic distal t.d.s.\ $(X,G)$ is automatically strictly ergodic. This is an easy consequence of the fact that a distal system is \textit{pointwise minimal}, i.e.\ it decomposes as a disjoint union of minimal subsystems (see \cite[Proposition 2.5-8]{glasner2007}). Minimality in turn implies that the unique ergodic measure has full support. 
\end{rem}

\subsection{Conditional expectation.}\label{Sec: Con exp}

Let  $(X,\mathcal{X},\mu)$ be a  probability space and let $\mathcal{B}$ be a sub-$\sigma$-algebra of $\mathcal{X}$. For $f\in L^1(\mu)$, the \strong{conditional expectation} of $f$ w.r.t.\  $\mathcal{B}$ is the function $\mathbb{E}(f|\mathcal{B})\in L^1(X,\mathcal{B},\mu)$ satisfying
\begin{equation}\label{conditional expection}\int_Bfd\mu=\int_B\mathbb{E}(f|\mathcal{B})d\mu,
\end{equation}
for every $B\in \mathcal{B} $. For $f\in L^1(\mu)$ and $g\in L^{\infty}(X,\mathcal{B},\mu)$, it holds (see \cite[Chapter 2, Section 2.4]{HK2018}):
\begin{equation}\label{con exp commute}
\int_Xfgd\mu=\int_X\mathbb{E}(f|\mathcal{B})gd\mu.
\end{equation}
Let $(X,\mathcal{X},\mu)$ and $(Y,\mathcal{Y},\nu)$ be probability spaces and let $\pi:X\rightarrow Y$ be a measurable map such that the pushforward $\pi_{*}\mu$ is $\nu$. Denote by $\mathbb{E}(f|Y)\in L^1(Y,\nu)$ w.r.t.\  $\pi$ the function such that $\mathbb{E}(f|Y)\circ\pi=\mathbb{E}(f|\pi^{-1}(\mathcal{Y}))$. Note this is well-defined. Thus the difference between $\mathbb{E}(f|Y)$ and $\mathbb{E}(f|\pi^{-1}(\mathcal{Y}))$ is that the first function is considered as a function on $Y$ and the second as a function on $X$.

\subsection{Measure disintegration.}
\begin{thm}
\label{thm:(measure-disintegration-theorem)}(Measure disintegration
theorem) \cite[Theorem A.7]{G03} Let $Y$ and $X$ be two standard
Borel spaces, and $\pi:X\rightarrow Y$ a Borel map. Let $\mu\in \mathcal{M}(X)$
be a Borel probability measure of $X$ and $\nu=\pi_*(\mu)$ its image in $\mathcal{M}(Y)$, then there is a Borel map $y\rightarrow\mu_{y}$,
from $Y$ into $\mathcal{M}(X)$ such that:
\end{thm}

\begin{itemize}
\item For $\nu$ almost every $y\in Y$, $\mu_{y}(\pi^{-1}(y))=1$;
\item $\mu=\int_{Y}\mu_{y}d\nu(y)$.
\end{itemize}
Moreover such a map is unique in the following sense: If $y\rightarrow\mu_{y}'$
is another such map, then $\mu_{y}=\mu_{y}'$, $\nu$-a.e. The collection
$\{\mu_{y}\}_{y\in Y}$ is called a \strong{measure disintegration of $\mu$}
w.r.t $\pi$.

\subsection{Nilsystems.} \label{subsec:Nilsystem-and-a system at most d}

A (real) \strong{Lie group} is a group that is also a finite-dimensional
real smooth manifold such that the group operations of multiplication
and inversion are smooth. Let $L$ be a Lie group. Let $L_1=L$ and $L_{k}=[L_{k-1},L]$
for $k\geq 2$, where $[L,H]=\{[g,h]:g\in L,h\in H\}$ and $[g,h]=g^{-1}h^{-1}gh$.
If there exists some $d\geq1$ such that $L_{d}\neq\{e\}$ and $L_{d+1}=\{e\}$,
$L$ is called a \strong{$\mathbf{d}$-step nilpotent} Lie group. We say that a discrete subgroup $\Gamma$ of a Lie group $L$ is 
\strong{cocompact} if $L/\Gamma$, endowed with the quotient topology,
is compact. We say that the quotient $X=L/\Gamma$ is a \strong{$d$-step nilmanifold}
if $L$ is a $d$-step nilpotent Lie group and
$\Gamma$ is a discrete, cocompact subgroup. The group $L$ acts naturally on $X=L/\Gamma$ by \strong{left translations}, $g(a\Gamma)=ga\Gamma$, for $g, a\in L$. A \strong{nilsystem} $(X,G)$ of degree at most $d$ is given by an $\ell$-step nilmanifold
$X=L/\Gamma$ where $\ell\leq d$  and a subgroup $G\subset L$ acting on $X$ by left translations.   Note that $X=G/\Gamma$ has a unique Borel probability measure $\mu_{\haar}$ which is invariant under all left translations. The measure $\mu_{\haar}$ is called the \strong{Haar measure} of $X=L/\Gamma$ (see \cite[Chapter 10, Section 2.1]{HK2018}).

\subsection{Fibrant cubespaces and nilspaces.\label{subsec:Cube-space,-nilspace}}
The theory of nilspaces originating in Host and Kra's theory of
parallelepiped structures of \cite{HK08}, was introduced by  Antol\'{i}n Camarena and Szegedy in \cite{CS12}. Oftentimes in this paper we follow the exposition of  \cite{candela2016cpt_notes, candela2016alg_notes}.
\begin{defn}\label{def:cubespace}
A map $p=(p_1,\ldots p_{\ell}):\{0,1\}^{k}\rightarrow\{0,1\}^{\ell}$ is called a \strong{morphism}  (of discrete cubes)
 if each coordinate function
$p_{j}(\omega_{1},\ldots,\omega_{k})$ is either identically  $0$,
identically $1$, or it equals either $\omega_{i}$
or $\overline{\omega_{i}}=1-\omega_{i}$ for some $1\leq i=i(j)\le k$. A morphism $\sigma:\{0,1\}^{k}\rightarrow\{0,1\}^{k}$ is called an \strong{isomorphism}
if it is bijective. Each isomorphism
can be written as 
$$\sigma(c)(j)=\begin{cases}
c(\delta(j)) & j\in I_{0}\\
1-c(\delta(j)) & j\in I_{1}
\end{cases}$$ 
for some $I_{0}\sqcup I_{1}=\{1,\ldots,k\}$ and bijection $\delta:\{1,\ldots,k\}\rightarrow\{1,\ldots,k\}$. We denote\footnote{This function is denoted $r(\cdot)$ in \cite[\S 3.3.3]{candela2016alg_notes}.} $\sgn(\sigma)=(-1)^{|I_{1}|}$. Let $[k]=\{0,1\}^{k}$. A \strong{cubespace} $(X,C^{\bullet}(X))$ is a topological space $X$, together with closed
sets $C^{k}(X)\subset X^{[k]}$ of \emph{$k$-cubes} for $k\geq0$ with $C^{0}(X)=X$,
satisfying the following condition. Suppose $p:\{0,1\}^{k}\rightarrow\{0,1\}^{\ell}$
is a morphism of discrete cubes and $c:\{0,1\}^{\ell}\rightarrow X$
is in $C^{\ell}(X)$. Then $c\circ p:\{0,1\}^{k}\rightarrow X$ is in
$C^{k}(X)$. A cubespace $(X,C^{\bullet}(X))$ is called \strong{$k$-ergodic} if $C^k(X)=X^{[k]}$. In particular, $(X,C^{\bullet}(X))$ is called \strong{ergodic} if $C^1(X)=X^{2}$. Suppose $(X,C^{\bullet}(X))$ and $(Y,C^{\bullet}(Y))$ are cubespaces. Let $f:X\rightarrow Y$ be a continuous map. We say  $f$ is a \strong{cubespace morphism} or just \strong{morphism} if $f^{[k]}(C^k(X))\subset C^k(Y)$ for all $k\geq 0$. An $n$-\strong{face} of the discrete cube $\{0,1\}^{k}$  is a subcube obtained by fixing $k-n$ of the
coordinates.  
\end{defn}

\begin{defn}\label{k ergodic fibration} Denote $\vec{1}=(1,\ldots,1)$. We call a map $\lambda:  \{0,1\}^\ell \setminus \{\vec 1\}\to X$ an $\ell$-{\bf corner} if $\lambda|_{\omega_i=0}$ is an $(\ell-1)$-cube for all $1\leq i\leq \ell$. We say that a morphism $f:(X,C^{\bullet}(X))\rightarrow (Y,C^{\bullet}(Y))$ is a \strong{fibration} if the following holds for all $\ell\geq 1$. Let $\lambda:\{0,1\}^\ell \setminus \{\vec 1\}\to X$ be an $\ell$-corner and $c\in C^k(Y)$ a compatible cube, in the sense that
$f\circ\lambda=c|_{\{0,1\}^\ell\setminus\{\vec 1\}}$.
Then there is a completion $c'\in C^\ell(X)$ of $\lambda$ compatible with $c$, i.e.\  it holds $c'|_{\{0,1\}^\ell \setminus \{\vec1\}} = \lambda$ and $f\circ c'=c$. We say a fibration $f$ is $k$-\strong{ergodic} if $C^{k}(X)=(f^{[k]})^{-1}(C^k(Y))$. In particular, if $f:X\rightarrow \bullet$ is a fibration, then $(X,C^{\bullet}(X))$ is called a \strong{fibrant} cubespace.
\end{defn}
\begin{defn}\label{def:nilspace}
We say a fibration $f:X\rightarrow Y$ is a fibration of order at most $k$ if the following
holds: whenever $c,c'\in C^{k}(X)$, $f^{[k]}(c)=f^{[k]}(c')$ and $c(w)=c'(w)$ for all $w\in\{0,1\}^{k}\setminus\{\vec{1}\}$
then $c=c'$. In particular, if $f:X\rightarrow \bullet$ is a fibration of order at most $k$, then $(X,C^{\bullet}(X))$ is called a \strong{nilspace} of order at most $k$. 
\end{defn}

\begin{defn}
A cubespace $(X,C^{\bullet}(X))$ is called \strong{compact} if (the topological space) $X$ is compact and metric. A cubespace $(X,C^{\bullet}(X))$ is called \strong{abstract} if $X$ is discrete. Note that no axiom of countability is assumed in the latter definition making it essentially a "topology-free" category. In this paper we mostly use compact cubespaces. Therefore, in order to simplify notation, \strong{in this paper we refer to compact cubespaces (nilspaces) simply as cubespaces (nilspaces)}. In contrast abstract cubespaces are always referred to as \textit{abstract}.   
\end{defn}

We identify $\{0,1\}^{n}$ with the collection
of all subsets of $\{1,\ldots,n\}$ and write $v'\subset v$ for $v',v\in\{0,1\}^{n}$
if $v'_i\leq v_i$ for all $i$. Define $|v|=\#\{j:v_j=1\}$. Let $V\subset\{0,1\}^{n}$ be a
\strong{downward-closed} subset, i.e.\  if $v\in V$ and $v'\subset v$
then $v'\in V$. Denote by $\text{Hom}(V,X)$ the set of maps $\alpha:V\rightarrow X$
such that for all $v\in V$, $\alpha|_{\{v':\ v'\subset v\}}$ is
a cube of $X$. We say that
$\text{Hom}(V,X)$ has the \strong{extension property} if for every
$\alpha\in\text{Hom}(V,X)$, there exists $c\in C^{n}(X)$ such that
$c|_{V}=\alpha$.
\begin{lem} (cf.
\cite[Lemma 7.12]{GGY2018}\cite[Lemma 3.1.5]{candela2016alg_notes}\label{lem: Extension property}) Let $(X,C^{\bullet}(X))$ be a fibrant cubespace and $V\subset\{0,1\}^{d}$ be a downward-closed
subset, then $\text{Hom}(V,X)$ has the extension property.
\end{lem}
\begin{proof} We prove by induction on $i$, where $|V|=2^{d}-i$. For $i=1$, by the definition of fibrant cubespace, $\text{Hom}(V,X)$ has 
the extension property. Assume the claim holds for $i$.  For any downward-closed
subset $V$ such that $|V|=2^{d}-i-1$, let $w\in \{0,1\}^{d}$ be an element such that $|w|=\min_{v\in\{0,1\}^d\setminus V}|v|$. This implies that $V\cup\{w\}$ is a downward-closed subset. Let $c\in \text{Hom}(V,X)$. As $X$ is a fibrant cubespace, there exists $\hat{c}_w\in X$ such that $b=(c_v,\hat{c}_w)_{v\subset w,v\neq w}\in C^{|w|}(X)$. Note that it may hold $\{v|\,v\subset w\}\subsetneq V$, however it is easy to see that defining $\hat{b}$ by $\hat{b}\in X^{V\cup\{w\}}$ by $\hat{b}_v=c_v$ for $v\in V$ and $\hat{b}_w=\hat{c}_w$, one has $\hat{b}\in\text{Hom}(V\cup\{w\},X)$.  Using the inductive assumption, there exists $b'\in C^{d}(X)$ such that $b'|_{V\cup\{w\}}=b$ which implies $b'|_V=c$. Q.E.D.
\end{proof}

\subsection{\label{subsec:Definition-of-the mu k}The Host-Kra cube group and the Host-Kra measures.}

Host and Kra introduced \textit{cube measures} defined on $X^{[k]}$ for ergodic $\Z$-actions.
These definitions easily generalize to the context of ergodic finitely generated abelian actions, in particular ergodic $\Z^k$-actions.

\begin{defn}\label{host-kra cube group and face group}
Let $k\geq 1$. For $1\leq j\leq k$, let $\overline{\alpha}_{j}=\{v\in\{0,1\}^{k}:v_j=1\}$ be
the $j$-th upper face of $\{0,1\}^{k}$.   For any face $F\subset \{0,1\}^k$ and $g\in G$, define $$(g^{F})_v=\begin{cases} g & v\in F\\
\id& v\notin F.
\end{cases}$$
Define the \strong{face group}   $\mathcal{F}^k(G)$ to be the subgroup of $G^{[k]}$ generated by $\{g^{\overline{\alpha}_{j}}:g\in G,1\leq j\leq k\}$. Define the \strong{diagonal subgroup} by $\triangle_{G}^{[k]}={\{g^{[k]}|\,g\in G\}}$. Define the $k$-th \strong{Host-Kra cube group} $\mathcal{HK}^k(G)$ to be the subgroup of $G^{[k]}$ generated by $\mathcal{F}^k(G)$ and $\triangle_{G}^{[k]}$. Also set $\mathcal{HK}^0(G)=G$.
\end{defn}

\begin{lem}\cite[Corollary A.15]{GMVI}\label{fibrant cubespace for G}
$(G,\mathcal{HK}^{\bullet}(G))$ is a fibrant abstract cubespace. 
\end{lem}

\begin{defn} \label{def: muk definition} 
Let $(X,\mathcal{B},\mu,G)$ be an ergodic m.p.s. Let $\mu^{[0]}=\mu$.  Let $\mathcal{I}_{G}^{[k]}$ be the $\triangle_{G}^{[k]}$-invariant
$\sigma$-algebra of $(X^{[k]}:= X^{\{0,1\}^{k}},\mathcal{B}^{[k]},\mu^{[k]})$.
Define $\mu^{[k+1]}$ to be the relative independent
joining of two copies of $\mu^{[k]}$ over $\mathcal{I}^{k}=\mathcal{I}_{G}^{[k]}$,
that is, for $f_{v}\in L^{\infty}(\mu)$, $v\in\{0,1\}^{k+1}$:
\[
\int_{X^{[k+1]}}\prod_{v\in\{0,1\}^{k+1}}f_{v}(x_v)d\mu^{[k+1]}=\int_{X^{[k]}}\mathbb{{E}}(\prod_{v\in\{0,1\}^{k}}f_{v0}(x_{v})|\mathcal{I}_{G}^{[k]})\mathbb{{E}}(\prod_{v\in\{0,1\}^{k}}f_{v1}(x_{v})|\mathcal{I}_{G}^{[k]})d\mu^{[k]}(\mathbf{x})
.\]

The measures $\{\mu^{[k]}\}_{k\geq 0}$ are referred to as the \textbf{Host-Kra measures}.
\end{defn}

From the previous definition it is easy to see that for any measurable functions $H_1,H_2\in L^\infty(X^{[k]},\mu^{[k]})$,
\begin{equation}\label{def muk 2}\int_{X^{[k]}}\mathbb{{E}}(H_1|\mathcal{I}_{G}^{[k]})(c)\mathbb{{E}}(H_2|\mathcal{I}_{G}^{[k]})(c)d\mu^{[k]}(c)= \int_{X^{[k]}}\mathbb{{E}}(H_1|\mathcal{I}_{G}^{[k]})(c)H_2(c)d\mu^{[k]}(c).\end{equation}

Note $\mu^{[k]}$ is  $\mathcal{HK}^{k}(G)$-invariant (cf.\ \cite[Chapter 9, Proposition 2]{HK2018}).

\subsection{Dynamical cubespaces.} \label{subsec:Dynamical cubespaces}

\begin{defn}
\label{def: k- equivalent} Let $(X,G)$ be a minimal t.d.s. Define the \textbf{induced dynamical cubespace} $(X,C_G^{\bullet}(X))$ by ($k\geq 0$): $$C_G^{k}(X)=\overline{\{\mathbf{g}\mathbf{x}:\mathbf{x}=(x,\ldots,x)\in X^{[k]},\mathbf{g}\in\mathcal{HK}^{k}(G)\}}.$$

\end{defn}
\begin{thm}\label{thm:distal->fibrant} (\cite[Theorem 7.10]{GGY2018}\label{thm: distality implies extension property})
Let $(G,X)$ be a minimal distal t.d.s.,\ then the cubespace $(X,C_G^\bullet(X))$ is ergodic and fibrant.
\end{thm}

\begin{defn}\label{NRP(M-G)}
Let $c^{k+1}_y(x):\{0,1\}^{k+1}\rightarrow X$ be the configuration given by $c^{k+1}_y(x)_{\vec{0}}=y$ and $c^{k+1}_y(x)_{\nu}=x$ for $\nu\neq \vec{0}$. Define the relation $\sim_{k}$ for $k\geq 0$ by $
x\sim_{k}y\text{ if }c^{k+1}_y(x)\in C_G^{k+1}(X).$ According to \cite[Section 5]{GGY2018}, $\sim_{k}$ is an closed $G$-invariant  equivalence relation. Suppose that $(M,G)\rightarrow (X,G)$ is a topological factor map.
Following \cite[Definition 1.33]{GL2019} define $\nrp_G^{[k]}(M\rightarrow X)$ for $k\geq 0$ by
$$
(x,y)\in \nrp_G^{[k]}(M\rightarrow X)\text{ if }\pi(x)=\pi(y)\text{ and }x\sim_{k}y.
$$
Clearly $\nrp_G^{[k]}(M\rightarrow X)$ is a closed $G$-invariant equivalence relation. The equivalence relation $\nrp_G^{[k]}(M\rightarrow X)$ is called the \strong{relative nilpotent regionally proximal relation} of order $k$ w.r.t.\  $M\rightarrow X$.
Define $\nrp_G^{[k]}(X)=\nrp_G^{[k]}(X\rightarrow \bullet)$. By \cite[Theorem 7.15]{GGY2018} $(X/\nrp_G^{[k]}(X), C_G^{\bullet}(X/\nrp_G^{[k]}(X)))$ is a nilspace of order at most $k$. It is referred to as the 
$k$-th \textbf{dynamical nilspace} associated with $(G,X)$.
\end{defn}

The following theorem was proven in the generality of an acting group 
which has a dense subgroup generated by a compact set. We specialize to the case of a finitely generated abelian acting group. 

\begin{thm}
\label{thm:dynamical sturcture strong ver}\cite[Theorem 1.29]{GMVIII} Let $G$ be a finitely generated abelian group (e.g., $G=\Z^k$, $k\in \N$) and $(X,G)$ a minimal
t.d.s. The following are equivalent for $d\geq 1$:
\begin{enumerate}
\item The t.d.s.\ $(X,G)$ is an inverse limit of nilsystems of degree at most $d$;

\item It holds that $\nrp_G^{[d]}(X)=\triangle$.
\end{enumerate}
\end{thm}

\subsection{The Host-Kra factors.}
\label{subsec:The Host-Kra factors}
The following lemmas and definitions are stated in \cite{HK2018} for ergodic $\Z$-actions but easily generalize to the context of finitely generated abelian actions, in particular $\Z^k$-actions.

\begin{lem}(\cite[Chapter 9, Proposition 2]{HK2018})
\label{lem:muk is ergodic} Let $k\in \mathbb{N}$. Let $G$ be a finitely generated abelian group and $(X,G)$ an ergodic m.p.s. Then the m.p.s.\ $(X^{[k]},\mu^{[k]},\mathcal{HK}^{k}(G))$ is ergodic.
\end{lem}

Denote by $\mathcal{S}_{k}$ the group of $k$-discrete cube isomorphisms.

\begin{lem}
\label{lem:muk is invariant under symmetries}(\cite[Chapter 8, Proposition 8]{HK2018}) Let $k\in \mathbb{N}$. Let $G$ be a finitely generated abelian group and $(X,G)$ an ergodic m.p.s. Then the m.p.s.\ $(X^{[k]},\mu^{[k]})$ is invariant under $S_{k}$.
\end{lem}
\begin{defn}
(\cite[Chapter 9, Section 1]{HK2018}\label{def:corner}) For $k\geq1$,
we write $[k]^{*}=[k]\setminus\{\vec{0}\}$. Points of $X_{*}^{[k]}$
are written $x^{*}=(\mathbf{x}_{v}:v\in[k]^{*})$. The natural projection
from $X^{[k]}$ to $X_{*}^{[k]}$ obtained by removing the coordinate
$\vec{0}$ is denoted by $x\rightarrow x^{*}$. The image of $\mu^{[k]}$
under the projection $x\rightarrow x^{*}$ is denoted by $\mu^{[k]*}$. Recall the definition of $\mathcal{F}^{k}(G)$ in Definition \ref{host-kra cube group and face group}. Let $\mathcal{J}_{*}^{[k]}$ be the $\sigma$-algebras of sets invariant
under $\mathcal{F}^{k}(G)$ on $X_{*}^{[k]}$.
\end{defn}
Now we can define the factor $Z_{k}(X)$.
\begin{defn}
\label{subsec:Z_k}   \cite[Subsection 4.2 on p. 411]{HK05}
\cite[Subsection 9.1]{HK2018} Let $\mathcal{Z}_{k}(X)$ be the $\sigma$-algebra consisting of measurable sets $B$ such that there exists a $\mathcal{J}_{*}^{[k+1]}$-measurable
set $A\subset X^{[k+1]^{*}}:= X^{\{0,1\}_{*}^{k}}$ so that
up to $\mu^{[k+1]}$- measure zero it holds:
\begin{equation}
X\times A=B\times X^{[k+1]^{*}}.\label{eq:measurable unique completeness}
\end{equation}
Let $Z_k(X)$ be the measurable factor of $X$ w.r.t.\  $\mathcal{Z}_{k}(X)$. Let $\mu_k$ be the projection of $\mu$ w.r.t.\  $X\rightarrow Z_{k}(X)$. The m.p.s.\
$(Z_{k}(X),\mu_k,G)$ is called the \strong{factor of order k of $(X,\mu,G)$}.
A system 
$(X,\mu,G)$ is called a \strong{system of order k} if  $(X,\mu,G)=(Z_{k}(X),\mu_k,G)$.
Note $Z_{0}(X)$ is the trivial factor and $Z_{1}(X)$ is the Kronecker
factor of $X$.
\end{defn}

\subsection{The tricube.\label{subsec:The-tri-cube(3-cube)}}

In this section we define special cubespaces, \textit{tricubes}, which will be useful
in many calculations. These will simply be $n$-cubes of side length two, divided into unit cubes. We will typically use them to form new cubes in fibrant cubespaces
by glueing together cubes into tricubes and considering the cube given by outer
vertices. One can learn more from \cite{candela2016cpt_notes,candela2016alg_notes}.

For every $v\in\{0,1\}^{n}$ , let $v_j$ be the $j$-th
coordinate of $v$. We define the injective maps $\Psi_{v}:\{0,1\}^{n}\rightarrow\{-1,0,1\}^{n}$
by
\[
\Psi_{v}(\varepsilon_1,\dots,\varepsilon_n)_j=(1-2v_j)(1-\varepsilon_j).
\]
The embedding $\Omega:\{0,1\}^{n}\rightarrow\{-1,0,1\}^{n}$ defined
by $\Omega(v)=\Psi_{v}(\vec{0})$ maps the cube $\{0,1\}^{n}$ to the
set $\{1,-1\}^{n}$ of ``outer vertices'' of $\{-1,0,1\}^{n}$.


\begin{defn}
Let $G$ be a finitely generated abelian group. The following subgroup $T^n(G)\subset G^{\{-1,0,1\}^{n}}$
is called the \strong{$n$-tricube group} of $G$:
\[
T^n(G)=\{\mathbf{t}\in G^{\{-1,0,1\}^{n}}:\ (\mathbf{t}_{\Psi_{v}(w)})_{w\in\{0,1\}^{n}}\in\mathcal{HK}^{n}(G)\text{ for all }v\in\{0,1\}^{n}\}.
\]

\end{defn}

\begin{defn}\label{def of tricube}
Let $(X,C^{\bullet}(X))$ be a fibrant cubespace. The following closed
subset $T^n(X)\subset X^{\{-1,0,1\}^{n}}$ is
called the \strong{$n$-tricube} of $X$:

\[
T^n(X)=\{\mathbf{t}\in X^{\{-1,0,1\}^{n}}:\ (\mathbf{t}_{\Psi_{v}(\tilde{v})})_{\tilde{v}\in\{0,1\}^{n}}\in C^{n}(X)\text{ for all }v\in\{0,1\}^{n}\}.
\]
\end{defn}
\begin{defn}
Let $\psi_{v}:T^{n}(X)\rightarrow C^{n}(X)$ (or $T^{n}(G)\rightarrow \mathcal{HK}^{n}(G)$) be
given by $\psi_{v}(\mathbf{t})=(\mathbf{t}_{\Psi_{v}(\tilde{v})})_{\tilde{v}\in\{0,1\}^{n}}$. Clearly by the definition of $T^{n}(X)$ (or $T^{n}(G)$), $\psi_{v}(\mathbf{t})$ is an $n$-cube for all $\mathbf{t}$.
\end{defn}
\begin{defn}
The \textbf{outer cube map} $\omega:T^{n}(X)\rightarrow C^{n}(X)$ (or $T^{n}(G)\rightarrow \mathcal{HK}^{n}(G)$) is given by $\omega(\mathbf{t})=(\mathbf{t}_{\Omega(\tilde{v})})_{\tilde{v}\in\{0,1\}^{n}}$.
\end{defn}
The fact that the outer cube map is well-defined is given by the following proposition:
\begin{prop}\label{prop:omega}
Let $(Y,C^{\bullet}(Y))$ be a fibrant abstract cubespace (e.g., $(G,\mathcal{HK}^{\bullet}(G))$) or a fibrant (compact) cubespace. Then for all $\mathbf{t}\in T^{n}(Y)$, $\omega(\mathbf{t})\in C^{n}(Y)$.
\end{prop}
\begin{proof}
 See \cite[Lemma 3.1.16]{candela2016alg_notes}.
\end{proof}

\subsection{Topological models.}
The terminology of \textit{models} in ergodic theory was introduced by Weiss in \cite{W85}, where a far-reaching generalization of the Jewett-Krieger theorem is given (\cite{jewett1970prevalence, krieger1972unique}).

\begin{defn}\label{topological model}
Let $(X,\mathcal{X},\mu,G)$ be a m.p.s. We say that a t.d.s.  $(\hat{X},G)$
is a \textbf{(topological) model} for $(X,\mathcal{X},\mu,G)$ w.r.t. to a $G-$invariant probability
measure $\hat{\mu}$ on $\hat{\mathcal{X}}$, the Borel $\sigma$-algebra of $X$,
if the system 
$(X,\mathcal{X},\mu,G)$ is isomorphic to $(\hat{X},\hat{\mathcal{X}},\hat{\mu},G)$ as m.p.s., that is, there exist a $G$-invariant Borel subset $C\subset X$ and  a $G$-invariant Borel subset $\hat{C}\subset \hat{X}$ of full measure and a (bi)measurable and equivariant measure-preserving bijective Borel map $p:C\rightarrow \hat{C}$. Notice that oftentimes in this article $(\hat{X},G)$ will be uniquely ergodic so that $\hat{\mu}$ will be the unique $G-$invariant probability
measure of $X$.
\end{defn}

\begin{defn}\label{def:top model for map} \sloppy
Let $(X,\mathcal{X},\mu,G)$, $(Y,\mathcal{Y}, \nu, G)$ be m.p.s. Let  $(\hat{X},G)$, $(\hat{Y},G)$ be t.d.s. which are topological models of $(X,\mathcal{X},\mu,G)$ and $(Y,\mathcal{Y}, \nu, G)$ w.r.t. measures $\h{\mu}$ and $\h{\nu}$ and maps $\phi$ and $\psi$ respectively. We say that $\h{\pi}: (\h{X},G)\rightarrow (\h{Y},G)$ is a \textbf{(topological)
model} for a factor map $\pi: (X,\mathcal{X},\mu, G)\rightarrow (Y,\mathcal{Y}, \nu, G)$ if
$\h{\pi}$ is a topological factor and the following diagram
\[
\begin{CD}
X @>{\phi}>> \h{X}\\
@V{\pi}VV      @VV{\h{\pi}}V\\
Y @>{\psi }>> \h{Y}
\end{CD}
\]
is commutative, i.e. $\h{\pi}\phi=\psi\pi$. If  $(\h{X},G)$ is strictly ergodic (and therefore $(\h{Y},G)$ is also strictly ergodic), then we say that  $\h{\pi}$ is a \textbf{strictly ergodic} (topological)
model.
\end{defn}

\section{Nilcycles and $k$-cube uniquely ergodic systems.\label{sec:Nil-distal-systems-of}}
\begin{defn}
\label{k-cube uniquely ergodic} A t.d.s.\ $(X,G)$ is called a $k$-\strong{cube uniquely ergodic system}
if $(C_G^{k}(X),\mathcal{HK}^{k}(G))$ is uniquely ergodic. Note that if $(X,G)$ is a $k$-cube uniquely ergodic system, then it is an $i$-cube uniquely ergodic system for
$0\leq i\leq k$. Given a $k$-cube uniquely ergodic system $(X,G)$, we define $\mu_{C_G^{k}(X)}$ to be the unique ergodic measure of $(C_G^{k}(X),\mathcal{HK}^{k}(G))$.
\end{defn}
\begin{rem}
There is an interesting equivalent characterization of $k$-cube uniquely ergodic $\Z$-systems developed in \cite{GL2019(2)}. A strictly ergodic t.d.s.\ $(X,\Z)$,  is said to be a \cfnil($k$) system if $(Z_{k}(X),\mu_{k},\Z)$ is isomorphic to (the uniquely ergodic)  $(X/\nrp_G^{[k]}(X),\Z)$ as m.p.s. According to \cite[Theorem C]{GL2019(2)}, a minimal t.d.s.\ $(X,\Z)$ is a $(k+1)$-cube uniquely ergodic system iff $(X,\Z)$ is a \cfnil($k$) system. This characterization will not be used in this article.
\end{rem}
\begin{defn}
 We say a cubespace $X$ has the \strong{glueing property}
if ``glueing'' two cubes along a common face yields another cube. Formally, let $k\in \mathbb{N}$. Suppose $c,c'\in C^{k}(X)$, and $c(v1)=c'(v0)$
for all $v\in\{0,1\}^{k-1}$ (here we use $v0$ to denote
$(v_{1},\dots,v_{k-1},0)$ and so on). Then the configuration
\[
c\bigparallel c'\colon v\mapsto\begin{cases}
c(v) & \colon v_{k}=0\\
c'(v) & \colon v_{k}=1
\end{cases}
\]
is in $C^{k}(X)$. In this case we say $c$ and $c'$ are \textbf{glueable}.
\end{defn}

\begin{defn}
\label{lem:RG is conditional product} $(X,C^{\bullet}(X))$ is a fibrant cubespace. Define the space of \textbf{glueable pairs} of cubes for cubespace $(X,C^{\bullet}(X))$.

\[
\mathcal{P}^{\ell}(X)=\{(c_{1},c_{2})\in C^{\ell}(X)\times C^{\ell}(X)|\,c_{1},\ c_{2}\ \mathrm{are}\ \mathrm{glueable}\}.
\]
Define the maps $\pi_{i}$ $i=L,U:$ $\pi_{L}(c)=(c{}_{w0})_{w\in\{0,1\}^{\ell-1}}$,
$\pi_{U}(c)=(c{}_{w1})_{w\in\{0,1\}^{\ell-1}}$.
\end{defn}
\begin{prop}
\label{prop:glueing pair uniquely ergodic} Suppose $(X,G)$ is a distal $(k+2)$-cube uniquely ergodic system. Let
$$ \mathcal{P}^{k+1}(G)=\{(\mathbf{g}_{1},\mathbf{g}_{2})\in\mathcal{HK}^{k+1}(G)\times\mathcal{HK}^{k+1}(G)|\,\mathbf{g}_{1},\mathbf{g}_{2}\ \mathrm{are}\ \mathrm{glueable}\}.
$$
Then $(\mathcal{P}^{k+1}(X),\mathcal{P}^{k+1}(G))$ is uniquely ergodic.
\end{prop}

\begin{proof}
Let 
\[
V=\{v\in\{0,1\}^{k+2}:v_{k+2}=0\text{ or }v_{k+1}=0\}.
\]
There is a projection $p:C_G^{k+2}(X)\rightarrow\mathcal{P}^{k+1}(X)$ given by:
\[
c\mapsto((c_{v})_{v_{k+2}=0,v\in\{0,1\}^{k+2}},(c_{v})_{v_{k+1}=0,v\in\{0,1\}^{k+2}}).
\]
Similarly we define $\tilde{p}:\mathcal{HK}^{k+2}(G)\rightarrow\mathcal{P}^{k+1}(G)$,
then it holds $p(\mathbf{g}c)=\tilde{p}(\mathbf{g})p(c)$. Note $V$ is a downward-closed subset. From Theorem \ref{thm: distality implies extension property} and Lemma \ref{lem: Extension property},
for any $(b_{1},b_{2})\in\mathcal{P}^{k+1}(X)$,
there is $c\in C_G^{k+2}(X)$ such that $p(c)=(b_{1},b_{2})$. Thus
$p$ is surjective and therefore $p$ is a factor map. As $(X,G)$
is a distal $(k+2)$-cube uniquely ergodic system, $(C_G^{k+2}(X),\mathcal{HK}^{k+2}(G))$
is uniquely ergodic, which implies that $(\mathcal{P}^{k+1}(X),\mathcal{P}^{k+1}(G))$
is uniquely ergodic.
\end{proof}
\begin{defn}
\label{def:meaure on glueing pair} Suppose $(X,G)$ is a distal $(k+2)$-cube uniquely ergodic system. Define $\mu_{\mathcal{P}^{k+1}(X)}$ as
 the unique invariant measure on $(\mathcal{P}^{k+1}(X),\mathcal{P}^{k+1}(G))$.
\end{defn}

\begin{defn}
\label{def:nilcycle}Let $(X,G,\mu)$ be a distal $(k+2)$-cube uniquely ergodic system. Suppose $Y=(X\times A,\nu=\mu\times m_{\haar(A)},G)$ is an ergodic m.p.s.\ which is an abelian group
extension of $X$. A \textbf{nilcycle of degree $k$} is a Borel
map $\rho\colon C_G^{k+1}(X)\to A$ satisfying

\begin{enumerate}
\item Cube invariance:
$\rho\circ\sigma(c)=\sgn(\sigma)\cdot\rho(c)$ for $\mu_{C_G^{k+1}(X)}$-a.e $c$ and all discrete cube isomorphisms $\sigma\colon\{0,1\}^{k+1}\to\{0,1\}^{k+1}$.
\item Glueing: $\rho(b\bigparallel c)=\rho(b)+\rho(c)$ for $\mu_{\mathcal{P}^{k+1}(X)}$-a.e.
$(b,c)$.
\item Measure decomposition: Let $\theta_{k+1}(\mathbf{a})=\underset{v\in\{0,1\}^{k+1}}{\sum}(-1)^{|v|}\mathbf{a}_{v}$ for $\mathbf{a}\in A^{[k+1]}$.\footnote{Sometimes we write $\theta(\mathbf{a})=\theta_{k+1}(\mathbf{a})$ if there no confusion arises .} There exists a closed subgroup $L$ of the group $\{\mathbf{a}\in A^{[k+1]}:\theta (\mathbf{a})=0\}$ and a Borel map $C_G^{k+1}(X)\rightarrow A^{[k+1]}:c\mapsto \mathbf{a}_c$ such that  $\rho(c)=\theta (\mathbf{a}_{c})$, $\{\mathbf{a}_c+m_{\haar(L)}\}_{c\in C_G^{k+1}(X)}$ is a measure disintegration for $\nu^{[k+1]}$.
\end{enumerate}
\end{defn}

\begin{rem}
The definition is modeled on the definition of \textit{cocycle} of
degree $k$ by Antol\'{i}n Camarena and Szegedy. The key difference
is that we allow properties (1) and (2) to hold almost everywhere (see \cite[Definition 3.3.14]{candela2016alg_notes}). We choose
the name "nilcycle" rather than "weak cocycle" as the term "cocycle" is traditionally reserved in the theory of dynamical systems for another purpose (see Section \ref{subsec:Dynamical-background.}).
\end{rem}

\begin{lem}\label{lem:cocycle_eq}
\label{lem:the action on M}Let $(X,G,\mu)$ be as above. Suppose $Y=(X\times A,\mu\times m_{\haar(A)},G)$ is an ergodic m.p.s.\ which is an abelian group extension of $X$ with cocycle $\beta$ and
$\rho\colon C_G^{k+1}(X)\to A$ is a map which has property (3) of nilcycles of degree $k$. Then for
$\mu^{[k+1]}$-a.e. $c\in C_G^{k+1}(X)$ and any $\mathbf{g}=(\mathbf{g}_{v})_{v\in\{0,1\}^{k+1}}\in\mathcal{HK}^{k+1}(G)$,

\[
\rho(\mathbf{g}c)=\rho(c)+\sum_{v\in\{0,1\}^{k+1}}(-1)^{|v|}\beta(\mathbf{g}_{v},c_{v}).
\]
\end{lem}
\begin{rem}
In particular, if $\mathbf{g}=h^{[k+1]}$ for some $h\in G$,
\begin{equation}
\rho(h^{[k+1]}c)=\rho(c)+\sum_{v\in\{0,1\}^{k+1}}(-1)^{|v|}\beta(h,c_{v}).\label{eq:the action on M}
\end{equation}
\end{rem}

\begin{proof}[Proof of Lemma \ref{lem:cocycle_eq}]
By a.s. uniqueness of measure disintegration and the fact that $(\mu\times m_{\haar(A)})^{[k+1]}$ is $\mathcal{HK}^{k+1}(G)$-invariant (see Subsection \ref{subsec:Definition-of-the mu k}),
for $\mu^{[k+1]}$-a.e. $c\in X^{[k+1]}$ and any $\mathbf{g}=(\mathbf{g}_{v})_{v\in\{0,1\}^{k+1}}\in\mathcal{HK}^{k+1}(G)$,
\[
\mathbf{a}_{\mathbf{g}c}+m_{\haar(L)}=\mathbf{g}_{*}(\mathbf{a}_{c}+m_{\haar(L)}).
\]
For any $\mathbf{b}\in L+\mathbf{a}_{c},$ identifying $Y^{[k+1]}=X^{[k+1]}\times A^{[k+1]}$,
\[
\mathbf{g}(c,\mathbf{b})=(\mathbf{g}c,\mathbf{b}+(\beta(\mathbf{g}_{v},c_{v}))_{v\in\{0,1\}^{k+1}}).
\]
Thus $\mathbf{a}_{\mathbf{g}c}=\mathbf{a}_{c}+(\beta(\mathbf{g}_{v},c_{v}))_{v\in\{0,1\}^{k+1}}$.
 By Definition \ref{def:nilcycle} (3), we have $$\nu^{[k+1]}=\int_{C_G^{k+1}(X)} \mathbf{a}_c+m_{\haar(L)} d\mu_{C_G^{k+1}(X)}(c).$$ In particular, one has that $\nu^{[k+1]}((\pi^{[k+1]})^{-1}(C_G^{k+1}(X)))=1$, where $\pi:Y\rightarrow X$ is the factor map. By the definition of the Host-Kra measures, one has that $(\pi^{[k+1]})_*\nu^{[k+1]}=\mu^{[k+1]}$ and therefore $\mu^{[k+1]}(C_G^{k+1}(X))=1$. Thus for
$\mu^{[k+1]}$-a.e. $c\in C_G^{k+1}(X)$ and any $\mathbf{g}=(\mathbf{g}_{v})_{v\in\{0,1\}^{k+1}}\in\mathcal{HK}^{k+1}(G)$,
\begin{equation}
\begin{array}{ll}
\rho(\mathbf{g}c) & =\sum_{v\in\{0,1\}^{k+1}}(-1)^{|v|}(\mathbf{a}_{\mathbf{g}c})_{v}\\
  &  =\sum_{v\in\{0,1\}^{k+1}}(-1)^{|v|}\left((\mathbf{a}_{c})_{v}+\beta(\mathbf{g}_{v},c_{v})\right)\\
  &  =\rho(c)+\sum_{v\in\{0,1\}^{k+1}}(-1)^{|v|}\beta(\mathbf{g}_{v},c_{v}).
\end{array}\label{eq: the action of HK}
\end{equation}

\end{proof}

\subsection{Continuous  measure disintegration (CMD) maps}\label{subsec:Relative-invariant-measures(RIM)}\label{subsec:Continuous-systems-of}
\begin{defn}
\label{def:CSM and CSM factor map} Let $\pi:B\rightarrow C$ be a
continuous and surjective map between two compact metric spaces. A
collection (system) of Borel probability measures on $B$, $\mu_{\pi}^{\bullet}=\{\mu_{\pi}^{c}\}_{c\in C}$,
is called a \textbf{continuous system of measures for $\pi$
}if
\begin{enumerate}
\item For every $c\in C$, $\mu_{\pi}^{c}$ is supported in $\pi^{-1}(c)$.
\item For every continuous function $f:B\rightarrow\mathbb{C}$ the function
$\mu_{\pi}^{(\cdot)}(f):C\rightarrow\mathbb{C}$ given by $\mu_{\pi}^{c}(f):=\int_{B}f(b)d\mu_{\pi}^{c}(b)$
is continuous; i.e.\  the map $C\rightarrow \mathcal{M}(B)$: $c\rightarrow\mu_{\pi}^{c}$
is continuous.
\end{enumerate}
Suppose in addition that $\{\mu_\pi^c\}_{c\in C}$ is a measure disintegration of some measure $\mu$ on $B$ (see Theorem \ref{thm:(measure-disintegration-theorem)}), then we say that $\pi:(B,\mu)\rightarrow (C,\pi_*\mu)$ is a \strong{continuous  measure disintegration (CMD) map}. 
\end{defn}

\begin{rem}\label{uniqueness of CSM}
Note that $\pi:(B,\mu)\rightarrow (C,\pi_*\mu)$ may have more than one continuous system of measures which is a measure disintegration. However if $\mu_\pi^{(\cdot)}$ and $\hat{\mu}_\pi^{(\cdot)}$ are continuous systems of measures such that $\mu_\pi^{x}=\hat{\mu}_\pi^{x}$ for $x\in C'$ for a dense subset $C'\subset C$, then by continuity  $\mu_\pi^{(\cdot)}=\hat{\mu}_\pi^{(\cdot)}$. In particular, if $\pi_*\mu$ has full support, then $\pi$ has only one measure disintegration which is a continuous system of measures by the a.s. uniqueness of measure disintegration.
\end{rem}

\begin{defn}
\cite[Section 3]{Shmuel1975}\label{def:continuous section} Let $(\pi,\tilde{\pi}):(X,G)\rightarrow(Y,H)$ be a factor map. A collection (system) of Borel probability measures $\{\lambda_y\}_{y\in Y}$ on $X$ is called a \strong{continuous equivariant measure section} if 
\begin{enumerate}
\item
The map $\lambda:Y\rightarrow\mathcal{M}(X)$ given by $y\rightarrow\lambda_{y}$ is continuous.
\item
For each $y\in Y$, $\lambda_{y}(\pi^{-1}(y))=1$.
\item
For each $y\in Y$ and $g\in G$, $\lambda_{\tilde{\pi}(g)y}=g_{*}\lambda_{y}$.   
\end{enumerate}
Note that in particular it follows that $\{\lambda_y\}_{y\in Y}$ is a continuous system of measures for $\pi$.

\end{defn}

\begin{lem}
\cite[Proposition 3.1]{Shmuel1975}\label{lem:invariant measure of section}
Let $(\pi,\tilde{\pi}):(X,G)\rightarrow(Y,H)$
be a factor map. Let $\{\lambda_y\}_{y\in Y}$ be a continuous equivariant measure section. If $\nu$ is a $H$-invariant
measure on $Y$, then $\int_{Y}\lambda_{y}d\nu$ is a $G$-invariant
measure on $X$.
\end{lem}

\begin{prop}
\label{RIM for distal}\cite[Proposition 3.8]{Shmuel1975} Let $\pi:(X,G)\rightarrow(Y,H)$
be a distal factor map. Then there exists a continuous equivariant measure section w.r.t.\  $\pi$.
\end{prop}

The following theorem is proven for amenable group actions without the distality assumption in \cite[Proof of Proposition 8.1]{AKL2014}. When the factor map is distal, one can prove the theorem without the amenability assumption.
\begin{thm}\label{thm:factor of u.e.}
Let $\pi:(X,G)\rightarrow(Y,H)$
be a distal factor map and assume $(X,G)$ is uniquely ergodic. Then $(Y,H)$ is uniquely ergodic.
\end{thm}
\begin{proof}
From Proposition \ref{RIM for distal}, there exists a continuous equivariant measure section $\{\lambda_y\}_{y\in Y}$ w.r.t.\  $\pi$. Assume that $(Y,H)$ has two possibly distinct ergodic measures $\nu_1,\nu_2$. Let $\mu$ be the unique ergodic measure of $(X,G)$. From Lemma \ref{lem:invariant measure of section}, as $(X,G)$ is uniquely ergodic,
$\mu=\int_Y \lambda_y d\nu_1(y)=\int_Y \lambda_y d\nu_2(y).$
 For any measurable function $F:Y\rightarrow\mathbb{C},$  notice that $$\int_X F( \pi(x))d\mu(x)=\int_Y \int_{\pi^{-1}(y)}F(\pi(x))d\lambda_y(x) d\nu_1(y)=\int_Y \int _{\pi^{-1}(y)}F(\pi(x))d\lambda_y(x) d\nu_2(y),$$ and $\int _{\pi^{-1}(y)}F(\pi(x))d\lambda_y(x)=F(y)$. Therefore one has that $\int_Y F(y)d\nu_1(y)=\int_Y F(y)d\nu_2(y)$ for any measurable function $F:X\rightarrow \mathbb{C}$, thus $\nu_1=\nu_2$. 
\end{proof}

\begin{prop}\label{CMD for distal}
Let $(X,G)$ be a uniquely ergodic distal system. Let $\pi:(X,G)\rightarrow(Y,H)$ be a factor map. Then $\pi$ is a CMD map w.r.t.\  the unique ergodic measure $\mu_X$ of $(X,G)$  and there exists a unique continuous system of measures for $\pi$ which is a continuous equivariant measure section.
\end{prop}
\begin{proof}
From Proposition \ref{RIM for distal}, there exists a continuous equivariant measure section $\{\lambda_y\}_{y\in Y}$ w.r.t.\  $\pi$. Let  $\nu_Y=\pi_*\mu_X$. From Lemma \ref{lem:invariant measure of section}, $\int_{Y}\lambda_{y}d\nu_Y=\mu_X$. As a continuous equivariant measure section is a continuous system of measures, $\pi:(X,G,\mu_X)\rightarrow (Y,H,\nu_Y)$ is a CMD map. Moreover as $Y=\pi(X)$ is distal and uniquely ergodic (see Subsection \ref{subsec:Dynamical-background.} and Theorem \ref{thm:factor of u.e.}), by Remark \ref{rem:strictly ergodic distal}, $\nu_Y$ has full support. Therefore by Remark \ref{uniqueness of CSM}, $\{\lambda_y\}_{y\in Y}$ is a unique continuous system of measures for $\pi$. 
\end{proof}

\begin{cor}\label{cor:p_0 cont meas sec}
Let $(X,G,\mu)$ be a distal $(k+1)$-cube uniquely ergodic system and let $\mathfrak{p}_0:(C_G^{k+1}(X),\mathcal{HK}^{k+1}(G))\rightarrow (X,G)$ be given by $ c\rightarrow c_{\vec{0}}$. Then there exists a unique continuous system of measures for $\mathfrak{p}_0$ which is a continuous equivariant measure section.
 \end{cor}
\begin{proof}
As $(C_G^{k+1}(X),\mathcal{HK}^{k+1}(G))$ is distal by Proposition \ref{prop:distality}(2) and strictly ergodic by assumption and $\mathfrak{p}_0$ is a factor map, the result follows from Proposition \ref{CMD for distal}.
\end{proof}

\begin{defn}\label{CSM for p0}
Let $(X,G,\mu)$ be a distal $(k+1)$-cube uniquely ergodic system. Denote by $\{\mu_{C_G^{k+1}(X)}^x\}_{x\in X}$ the unique continuous system of measures for $\mathfrak{p}_0:(C_G^{k+1}(X),\mathcal{HK}^{k+1}(G))\rightarrow (X,G)$.  
\end{defn}
\begin{lem}\cite[Proposition 3.3]{CG14}
\label{lem:CSM circ} Suppose $\pi_{1}:(B,\mu_{B})\rightarrow(C,\mu_{C})$,
$\pi_{2}:(C,\mu_{C})\rightarrow(D,\mu_{D})$, $\pi_{3}:(B,\mu_{B})\rightarrow(D,\mu_{D})$
are continuous surjective measure-preserving maps between compact spaces. If $\pi_{3}=\pi_{2}\circ\pi_{1}$ and $\pi_{1}$, $\pi_{2}$
are CMD maps,
then $\pi_{3}$ is a CMD map w.r.t.\  the continuous system of measures  $\{\gamma^{d}=\int_{C}\mu_{B}^{c}d\mu_{C}^{d}(c)\}_{d\in D}$.
\end{lem}

\subsection{Relations between tricube measures}\label{subsec:measure}

\begin{defn}
\label{def:definition of the product space} Let $(X,G)$ be a t.d.s. For $v\in \{0,1\}^{n}$, define the spaces
\[
C_G^{n}(X)\times_{X}^v C_G^{n}(X):=\{(c_1,c_2)\in C_G^{n}(X)\times C_G^{n}(X):(c_1)_v=(c_2)_{\vec{0}}\},
\]
\[
\mathcal{HK}^{n}(G)\times_{G}^v \mathcal{HK}^{n}(G):=\{(\mathbf{g}_{1},\mathbf{g}_{2})\in\mathcal{HK}^{n}(G)\times\mathcal{HK}^{n}(G):\ (\mathbf{g}_{1})_{v}=(\mathbf{g}_{2})_{\vec{0}}\}.
\]
If $v=\vec{0}$, we sometimes write $C_G^{n}(X)\times_{X}C_G^{n}(X)$ and $\mathcal{HK}^{n}(G)\times_{G} \mathcal{HK}^{n}(G)$ instead of $C_G^{n}(X)\times_{X}^{\vec{0}}C_G^{n}(X)$ and $\mathcal{HK}^{n}(G)\times_{G}^{\vec{0}} \mathcal{HK}^{n}(G)$ respectively. We 
define the maps\footnote{The fact that the maps $\tilde{\pi}_v,\tilde{\phi}_v$ are well-defined is proven in Lemma \ref{factor map lemmma}.}:
$$\begin{array}{l}
\pi':C_G^{n}(X)\times_{X}C_G^{n}(X)\rightarrow X:(c_{1},c_{2})\mapsto(c_{1})_{\vec{0}},\\
\tilde{\pi}_v:T^{n}(X)\rightarrow C_G^{n}(X)\times_{X}^v C_G^{n}(X):\mathbf{t}\mapsto (\omega(\mathbf{t}),\psi_{v}(\mathbf{t})),\\
\tilde{\phi}_v:T^{n}(G)\rightarrow\mathcal{HK}^{n}(G)\times_{G}^v \mathcal{HK}^{n}(G): \mathbf{g}\mapsto(\omega(\mathbf{g}),\psi_{v}(\mathbf{g})),\\
\pi_{L,v}:C_G^{n}(X)\times_{X}^vC_G^{n}(X)\rightarrow C_G^{n}(X) : (c_1,c_2)\mapsto c_{1},\\
\phi_{L,v}:\mathcal{HK}^{n}(G)\times_{G}^v \mathcal{HK}^{n}(G)\rightarrow \mathcal{HK}^{n}(G): (\mathbf{g}_1,\mathbf{g}_2)\mapsto \mathbf{g}_1.
\end{array}$$
\end{defn}
\noindent
We denote by $(\tilde{\pi}_v,\tilde{\phi}_v)$ and $(\pi_{L,v},\phi_{L,v})$ the maps  $(\mathbf{t},\mathbf{g})\mapsto(\tilde{\pi}_v(\mathbf{t}),\tilde{\phi}_v(\mathbf{g}))$ and $((c_1,c_2),(\mathbf{g}_1,\mathbf{g}_2))\mapsto(\pi_{L,v}((c_1,c_2)),\phi_{L,v}((\mathbf{g}_1,\mathbf{g}_2))$ respectively. If $v=\vec{0}$, we sometimes write $\tilde{\pi}$, $\tilde{\phi}$ and $\pi_{L}$ instead of $\tilde{\pi}_{\vec{0}}$,  $\tilde{\phi}_{\vec{0}}$ and $\pi_{L,\vec{0}}$ respectively.
\begin{lem}\label{factor map lemmma}
Let $(X,G)$ be a minimal distal t.d.s.\ and $v\in \{0,1\}^n$.  Then  $\omega$ and $\psi_{v}$  are surjective maps. The map  $(\tilde{\pi}_v,\tilde{\phi}_v)$ is well-defined. The maps $(\tilde{\pi}_v,\tilde{\phi}_v)$ and $(\pi_{L,v},\phi_{L,v})$ are factor maps.
\end{lem}

\begin{proof}

Let $\sigma:\{-1,0,1\}\rightarrow\{0,1\}^{2}$ be the function $\sigma(1)=(0,0),\sigma(0)=(1,0),\sigma(-1)=(0,1).$
Then $q=\sigma^{n}:v\rightarrow(\sigma(v_1),\ldots,\sigma(v_n))$
is an injective map $\{-1,0,1\}^n\rightarrow\{0,1\}^{2n}$. Now we prove that for any fibrant  abstract or compact cubespace $(Y,C^\bullet(Y))$ (e.g., $(X,C^\bullet(X))$ or  $(G,\mathcal{HK}^\bullet(G))$; see Lemma \ref{fibrant cubespace for G} and Theorem \ref{thm:distal->fibrant}), $\omega$ is a surjective map. Let $U=q\left(\Omega(\{0,1\}^n)\right)$.  Notice that $U$ is a downward-closed subset. By Lemma \ref{lem: Extension property}, for any $b\in C^n(Y)$, there exists $c\in C^{2n}(Y)$ such that $(c_{q\circ\Omega(v)})_{v\in\{0,1\}^{n}}=b$. As $\mathbf{t}=(c_{q(w)})_{w\in\{-1,0,1\}^{n}}$ is
a tricube, one has that $\omega(\mathbf{t})=b$, therefore $\omega: T^n(Y)\rightarrow C^n(Y)$ is surjective. Similarly, one has that $\psi_{v}$ is surjective for any $v\in \{0,1\}^n$.

We now show that $(\tilde{\pi}_v,\tilde{\phi}_v)$ is well-defined. One has to verify that $\omega(\mathbf{t})_v=\psi_{v}(\mathbf{t})_{\vec{0}}$. Indeed, by Subsection \ref{subsec:The-tri-cube(3-cube)},
  $\omega(\mathbf{t})_v=\mathbf{t}_{\Omega(v)}=\mathbf{t}_{\Psi_{v}(\vec{0})}=\psi_{v}(\mathbf{t})_{\vec{0}}$. Next we show that $(\tilde{\pi}_v,\tilde{\phi}_v)$ is a factor map. 
 It is clear that $\tilde{\pi}_v,\tilde{\phi}_v$ are continuous and $\tilde{\pi}_v(\mathbf{gt})=\tilde{\phi}_v(\mathbf{g})\tilde{\pi}_v(\mathbf{t})$. Now we prove that the map $\tilde{\pi}_v$ is surjective. i.e.\  for any $(b_1,b_2)\in C_G^{n}(X)\times_{X}^v C_G^{n}(X)$, there exists $\mathbf{t}\in T^n(X)$ such that $\tilde{\pi}_v(\mathbf{t})=(b_{1},b_{2})$. First we treat the case $v=\vec{0}$. Let $V=q\left(\Omega(\{0,1\}^{n})\cup\Psi_{\vec{0}}(\{0,1\}^{n})\right)\subset\{0,1\}^{2n}$. Note that
$V$ is a downward-closed subset. Arguing as above, we see that for any $(b_{1},b_{2})\in C_G^{n}(X)\times_{X}C_G^{n}(X)$, there is $c\in C_G^{2n}(X)$
such that $(c_{q\circ\Omega(v)})_{v\in\{0,1\}^{n}}=b_{1}$, $(c_{q\circ\Psi_{\vec{0}}(v)})_{v\in\{0,1\}^{n}}=b_{2}$\footnote{From $\Omega(\vec{0})=\Psi_{\vec{0}}(\vec{0})$ (see Subsection \ref{subsec:The-tri-cube(3-cube)}), it follows that the condition $(b_1)_{\vec{0}}=(b_2)_{\vec{0}}$ poses no obstruction.}. As $\mathbf{t}=(c_{q(w)})_{w\in\{-1,0,1\}^{n}}$ is
a tricube, one has that $\tilde{\pi}(\mathbf{t})=(b_{1},b_{2})$,
i.e.\  $\tilde{\pi}$ is surjective. 

Exploiting the case $v=\vec{0}$ we will now prove the surjectivity of $\tilde{\pi}_v$, for $v\in \{0,1\}^{n}\setminus \{\vec{0}\}$, working directly with tricubes (making the proof hopefully more intuitive, see \cite[Remark 3.1.18]{candela2016alg_notes}).  Define the map $\tau=\tau(v):\{-1,0,1\}^{n}\rightarrow \{-1,0,1\}^{n}$ by
$$\tau(\epsilon_1,\ldots \epsilon_{n})_i=\begin{cases}
\epsilon_i & \,\, v_{i}=0\\

-\epsilon_i & \,\, v_{i}=1\\
\end{cases}$$
We claim that $\tau$ induces a map  $\tilde{\tau}:T^{n}(X)\rightarrow T^{n}(X)$ defined by $$\tilde{\tau}\big((t_w)_{w\in \{-1,0,1\}^{n}}\big)=(t_{\tau(w)})_{w\in \{-1,0,1\}^{n}}.$$
Indeed let us verify that  for any $\mathbf{t}\in T^{n}(X)$, $\tilde{\tau}(\mathbf{t})\in T^{n}(X)$. By definition one has to show that for every
$v'\in \{0,1\}^n$, it holds $(\tilde{\tau}(\mathbf{t})_{\Psi_{v'}(\delta)})_{\delta\in\{0,1\}^{n}}\in C^{n}(X)$. Define the map $\hat{\tau}=\hat{\tau}(v):\{0,1\}^n\rightarrow \{0,1\}^n$ by:
 $$\hat{\tau}(\delta_1,\ldots \delta_{n})_i=\begin{cases}
\delta_i & \,\, v_{i}=0,\\

1-\delta_i & \,\, v_{i}=1.\\
\end{cases}$$ 
Note $(\tau(\Psi_{v'}(\delta)))_i=(1-2v'_i)(1-\delta_i)$ if $v_i=0$ and $(\tau(\Psi_{v'}(\delta)))_i=(1-2(1-v'_i))(1-\delta_i)$ if $v_i=1$. Thus $\tau(\Psi_{v'}(\delta))=\Psi_{\hat{\tau}(v')}(\delta)$ which implies $(\tilde{\tau}(\mathbf{t})_{\Psi_{v'}(\delta)})_{\delta\in\{0,1\}^{n}}=(\mathbf{t}_{\Psi_{\hat{\tau}(v')}(\delta)})_{\delta\in\{0,1\}^{n}}\in C_G^{n}(X)$.

We will now investigate $\tilde{\pi}_v(\tilde{\tau}(\mathbf{t}))$ for $\mathbf{t}\in T^{n}(X)$. We will show that as $\tilde{\tau}$ sends the outer cube $\Omega(\{0,1\}^n)$ to itself and the cube $\Psi_{\vec{0}} (\{0,1\}^n)$ to $\Psi_{v}(\{0,1\}^n)$, the surjectivity of  $\tilde{\pi}_v$ follows. Indeed fix  $(b_1,b_2)\in C_G^{n}(X)\times_{X}C_G^{n}(X)$. We have established there exists $\mathbf{t}\in T^{n}(X)$ with $\tilde{\pi}_{\vec{0}}(\mathbf{t})=(b_1,b_2)$. As $\Omega(\{0,1\}^n)=\{-1,1\}^n$, it follows  $\tau (\Omega(\{0,1\}^n))=\Omega(\{0,1\}^n)$. An explicit calculation shows $\Omega^{-1}\circ\tau \circ\Omega=\hat{\tau}$. Thus $\omega(\tilde{\tau}(\mathbf{t}))=\hat{\tau}(b_1)$. As $\Psi_{\vec{0}} (\{0,1\}^n)=\{0,1\}^n$, it follows $\mathbf{p}\in \tau(\Psi_{\vec{0}} (\{0,1\}^n))$ if and only if $\mathbf{p}_i\in \{0,1\}$ for $v_i=0$ and 
$\mathbf{p}_i\in \{-1,0\}$ for  $v_i=1$. This implies  that $\Psi_{v}(\{0,1\}^n)= \tau(\Psi_{\vec{0}} (\{0,1\}^n))$. Moreover an explicit calculation shows $\Psi_{v}^{-1}\circ\tau \circ \Psi_{\vec{0}}=\id$. Thus $\psi_{v}(\tilde{\tau}(\mathbf{t}))=b_2$. It is easy to see that $(\hat{\tau},\id)$ induces a bijective map $C_G^{n}(X)\times_{X}C_G^{n}(X)\rightarrow C_G^{n}(X)\times_{X}^v C_G^{n}(X)$. Conclude $\tilde{\pi}_v$ is surjective.
Similarly one establishes surjectivity for $\tilde{\phi}_v$ using Lemma \ref{fibrant cubespace for G}.

Finally we show that $(\pi_{L,v},\phi_{L,v})$ is a factor map. 
 It is clear that $\pi_{L,v},\phi_{L,v}$ are continuous and $\pi_{L,v}((\mathbf{g}_1,\mathbf{g}_2)(c_1,c_2))=\phi_{L,v}((\mathbf{g}_1,\mathbf{g}_2))\pi_{L,v}((c_1,c_2))$. Let $\sigma'$ be an isomorphism of discrete cubes such that $\sigma' (\vec{0})=v$. Note that for all $c\in C_G^{n}(X)$, $(c,\sigma'(c))\in C_G^{n}(X)\times_{X}^v C_G^{n}(X)$. This shows that the map $\pi_{L,v}$ is surjective.
 Similarly one establishes surjectivity for  $\phi_{L,v}$.
\end{proof}

 Now we assume that $(X,G)$ is a distal $(2k+2)$-cube uniquely ergodic system (i.e.\  $(X,G)$ is distal and $(C_G^{2k+2}(X),\mathcal{HK}^{2k+2}(G))$ is uniquely ergodic). Then by the next lemma there is a canonical measure for $T^{k+1}(X)$.
\begin{lem}\label{uniquely ergodic for T(X)}
If $(X,G,\mu)$ is a distal $(2k+2)$-cube uniquely ergodic system, then $(T^{k+1}(X),T^{k+1}(G))$ and $\left(C_G^{k+1}(X)\times_{X}C_G^{k+1}(X),\mathcal{HK}^{k+1}(G)\times_{G}\mathcal{HK}^{k+1}(G)\right)$ are uniquely ergodic. 
\end{lem}
\begin{proof}
Following the notation of the proof of Lemma \ref{factor map lemmma} and using the extension property for the downward-closed set $q(\{-1,0,1\}^{k+1})$, one has that $(T^{k+1}(X),T^{k+1}(G))$ is a factor of $(C_G^{2k+2}(X),\mathcal{HK}^{2k+2}(G))$ w.r.t.\  the projection $C_G^{2k+2}(Y)\rightarrow T^{k+1}(Y):c\rightarrow (c_{q(w)})_{w\in\{-1,0,1\}^{k+1}}$, where $Y=X,G$. As $(C_G^{2k+2}(X),\mathcal{HK}^{2k+2}(G))$ is uniquely ergodic, $(T^{k+1}(X),T^{k+1}(G))$ is uniquely ergodic.  Similarly,  $\left(C_G^{k+1}(X)\times_{X}C_G^{k+1}(X),\mathcal{HK}^{k+1}(G)\times_{G}\mathcal{HK}^{k+1}(G)\right)$ is uniquely ergodic since it is a factor of $(T^{k+1}(X),T^{k+1}(G))$ by Lemma \ref{factor map lemmma}. 
\end{proof}
\begin{defn}\label{definition of the measures for T(X) and CXC}
Let $(X,G,\mu)$ be a distal $(2k+2)$-cube uniquely ergodic system.
Denote the unique invariant measures of $(T^{k+1}(X),T^{k+1}(G))$ and $\left(C_G^{k+1}(X)\times_{X}C_G^{k+1}(X),\mathcal{HK}^{k+1}(G)\times_{G}\mathcal{HK}^{k+1}(G)\right)$ by $\mu_{T^{k+1}(X)}$ and $\mu_{C_G^{k+1}(X)\times_{X}C_G^{k+1}(X)}$ respectively. 
\end{defn}
\begin{cor} Let $(X,G,\mu)$ be a distal $(2k+2)$-cube uniquely ergodic system. Then
$\mu_{C_G^{k+1}(X)\times_{X}C_G^{k+1}(X)}=\int_X\mu_{C_G^{k+1}(X)}^x\times \mu_{C_G^{k+1}(X)}^xd\mu(x)$.
\end{cor}
\begin{proof}
As $(X,G,\mu)$ is a distal $(k+1)$-cube uniquely ergodic system, by Corollary \ref{cor:p_0 cont meas sec}, $\{\mu_{C_G^{k+1}(X)}^x\}_{x\in X}$ (defined in Definition \ref{CSM for p0})  is a continuous equivariant measure section. Thus for any $\mathbf{g}=(\mathbf{g}_{v})_{v\in\{0,1\}^{k+1}}\in\mathcal{HK}^{k+1}(G)$ one has that $\mathbf{g}_*\mu_{C_G^{k+1}(X)}^x=\mu_{C_G^{k+1}(X)}^{\mathbf{g}_{\vec{0}}x}$ for all $x\in X$. Note that for any $(\mathbf{g},\mathbf{h})\in \mathcal{HK}^{k+1}(G)\times_{G}\mathcal{HK}^{k+1}(G)$, one has that $\mathbf{g}_{\vec{0}}=\mathbf{h}_{\vec{0}}$, and thus $(\mathbf{g},\mathbf{h})_*(\mu_{C_G^{k+1}(X)}^x\times \mu_{C_G^{k+1}(X)}^{x})=\mu_{C_G^{k+1}(X)}^{\mathbf{g}_{\vec{0}}x}\times \mu_{C_G^{k+1}(X)}^{\mathbf{g}_{\vec{0}}x}$ for all $x\in X$. Therefore, as $\mu$ is $G$-invariant, by Lemma \ref{lem:invariant measure of section}, $\int_X \mu_{C_G^{k+1}(X)}^x\times \mu_{C_G^{k+1}(X)}^xd\mu(x)$ is an $\mathcal{HK}^{k+1}(G)\times_{G}\mathcal{HK}^{k+1}(G)$-invariant measure for $\left(C_G^{k+1}(X)\times_{X}C_G^{k+1}(X),\mathcal{HK}^{k+1}(G)\times_{G}\mathcal{HK}^{k+1}(G)\right)$. By Lemma \ref{uniquely ergodic for T(X)}, $\left(C_G^{k+1}(X)\times_{X}C_G^{k+1}(X),\mathcal{HK}^{k+1}(G)\times_{G}\mathcal{HK}^{k+1}(G)\right)$ is uniquely ergodic and therefore $\mu_{C_G^{k+1}(X)\times_{X}C_G^{k+1}(X)}=\int_X\mu_{C_G^{k+1}(X)}^x\times \mu_{C_G^{k+1}(X)}^xd\mu(x)$.
\end{proof}

\begin{prop} \label{prop:CMD for tri cube} 
Let $(X,G,\mu)$ be
a distal $(2k+2)$-cube uniquely ergodic system. Then the following maps are CMD maps:
\begin{enumerate}
   \item 
$(C_G^{k+1}(X),\mu_{C_G^{k+1}(X)})\xrightarrow{\mathfrak{p}_{0}:c\rightarrow c_{\vec{0}}}(X,\mu)$,
\item 
for any $v\in \{0,1\}^{k+1}$, $(T^{k+1}(X),\mu_{T^{k+1}(X)})\xrightarrow{\tilde{\pi}_v:\mathbf{t}\rightarrow(\omega(\mathbf{t}),\psi_{v}(\mathbf{t}))} (C_G^{k+1}(X)\times_{X}^vC_G^{k+1}(X),(\tilde{\pi}_v)_* \mu_{T^{k+1}(X)})$,
\item 
for any $v\in \{0,1\}^{k+1}$, $ (C_G^{k+1}(X)\times_{X}^vC_G^{k+1}(X),(\tilde{\pi}_v)_* \mu_{T^{k+1}(X)})\xrightarrow{\pi_{L,v}:(c_1,c_2)\mapsto c_{1}} (C_G^{k+1}(X),\mu_{C_G^{k+1}(X)})$.
\end{enumerate}

\end{prop}

\begin{proof}
from Lemma \ref{uniquely ergodic for T(X)}, the maps 1 and 2 are factor maps (in particular surjective) from the uniquely ergodic systems $(C_G^{k+1}(X),\mathcal{HK}^{k+1}(G),\mu_{C_G^{k+1}(X)})$ and $(T^{k+1}(X),T^{k+1}(G),\mu_{T^{k+1}(X)})$ respectively. By Proposition \ref{prop:distality}(2) these systems are distal. The system $(C_G^{k+1}(X)\times_{X}^vC_G^{k+1}(X),\mathcal{HK}^{k+1}(G)\times_{G}^v\mathcal{HK}^{k+1}(G), (\tilde{\pi}_v)_* \mu_{T^{k+1}(X)})$, being the image in factor map 2 of a uniquely ergodic distal system is uniquely ergodic and distal by Proposition \ref{prop:distality}(2) and Theorem \ref{thm:factor of u.e.}. We can now apply Proposition \ref{CMD for distal} to conclude that maps 1,2,3 are CMD maps.

\end{proof}

\begin{defn}\label{definition of eta and Txx}
Let $(X,G,\mu)$ be
a distal $(2k+2)$-cube uniquely ergodic system. Denote the continuous equivariant measure section w.r.t.\   $\omega:(T^{k+1}(X),T^{k+1}(G))\rightarrow (C_G^{k+1}(X),\mathcal{HK}^{k+1}(G))$ by $\{\eta^{c}\}_{c\in C_G^{k+1}(X)}$. Denote the continuous equivariant measure section w.r.t.\  $\pi_{T}:(T^{k+1}(X),T^{k+1}(G))\rightarrow (X,G):\mathbf{t}\rightarrow\mathbf{t}_{\vec{1}}$ by 
 $\{\mu_{T^{k+1}(X)}^{x}\}_{x\in X}$.
\end{defn}
\begin{cor}\label{cor:The CSM } Let $(X,G,\mu)$ be
a distal $(2k+2)$-cube uniquely ergodic system. Then 
 $\mu_{T^{k+1}(X)}^{x}=\int_{C_G^{k+1}(X)}\eta^{c}(\mathbf{t})d\mu_{C_G^{k+1}(X)}^{x}(c)$ for all $x\in X$.
\end{cor}

\begin{proof}
Notice $\omega=\pi_L\circ \tilde{\pi}$ and  $\pi_{T}=\mathfrak{p}_{0}\circ\omega$ and the involved maps are factor maps between uniquely ergodic distal systems. Therefore the involved maps are measure-preserving. Now apply Lemma \ref{lem:CSM circ} and Proposition \ref{prop:CMD for tri cube}.
\end{proof}

\begin{lem}\label{measure-preserving}
Let $(X,G,\mu)$ be a distal $(2k+2)$-cube uniquely ergodic system. Then $\psi_v,\omega:T^{k+1}(X)\rightarrow C_G^{k+1}(X)$ are measure-preserving for $v\in \{0,1\}^{k+1}$. Moreover, $\omega_*\mu_{T^{k+1}(X)}^x=\mu_{C_G^{k+1}(X)}^x,\  (\psi_{\vec{0}})_*\mu_{T^{k+1}(X)}^x=\mu_{C_G^{k+1}(X)}^x,\ (\psi_{v})_*\eta^c=\mu_{C_G^{k+1}(X)}^{c_v}$  
for all $x\in X$. If $v\neq \vec{0}$,  $(\psi_{v})_{*}\mu_{T^{k+1}(X)}^{x}=\mu_{C_G^{k+1}(X)}$ for all $x\in X$. In addition for all $x\in X$ and $g\in G$, $(g^{[k+1]})_*\mu_{C_G^{k+1}(X)}^x=\mu_{C_G^{k+1}(X)}^{gx}$.
\end{lem}
\begin{proof}

As $(X,G,\mu)$ is a distal $(2k+2)$-cube uniquely ergodic system, in particular a $(k+1)$-cube uniquely ergodic system, $(C_G^{k+1}(X),\mathcal{HK}^{k+1}(G))$ is uniquely ergodic. As $\psi_v(T^{k+1}(G))=\mathcal{HK}^{k+1}(G)$ for $v\in \{0,1\}^{k+1}$ and $\omega(T^{k+1}(G))=\mathcal{HK}^{k+1}(G)$ (see Lemma \ref{factor map lemmma}),  $(\psi_v)_*\mu_{T^{k+1}(X)}$ and $\omega_*\mu_{T^{k+1}(X)}$ are $\mathcal{HK}^{k+1}(G)$-invariant measures on $(C_G^{k+1}(X),\mathcal{HK}^{k+1}(G))$, we have:
\begin{equation}\label{equation for measures}(\psi_v)_*\mu_{T^{k+1}(X)}=\omega_*\mu_{T^{k+1}(X)}=\mu_{C_G^{k+1}(X)}.
\end{equation}
Notice that $ \mathfrak{p}_0\circ \psi_{\vec{0}}=\pi_T$. By Equation \eqref{equation for measures}, $$\mu_{C_G^{k+1}(X)}=(\psi_{\vec{0}})_*\mu_{T^{k+1}(X)}=\int_{X}(\psi_{\vec{0}})_*\mu_{T^{k+1}(X)}^xd\mu
(x).$$
As $\psi_{\vec{0}}(T^{k+1}(X))=C_G^{k+1}(X)$, $\{(\psi_{\vec{0}})_*\mu_{T^{k+1}(X)}^x\}_{x\in X}$ is a measure disintegration w.r.t.\  $\mathfrak{p}_0$.
Notice that $\psi_{\vec{0}}$ is a continuous map. Therefore for any continuous function $F:C_G^{k+1}(X)\rightarrow \mathbb{C}$, $F\circ \psi_{\vec{0}}$ is a continuous function on $T^{k+1}(X)$. As $\{\mu^x_{T^{k+1}(X)}\}_{x\in X}$ is a continuous system of measures, 
$$x\mapsto \int_{T^{k+1}(X)}F\circ \psi_{\vec{0}}d\mu^x_{T^{k+1}(X)}$$
is a continuous map. Therefore $(\psi_{\vec{0}})_*\mu_{T^{k+1}(X)}^x$ is a continuous system of measures w.r.t.\  $\mathfrak{p}_0:C_G^{k+1}(X)\rightarrow X$. As $\mu_{C_G^{k+1}(X)}^x$ is a continuous system of measures for $\mathfrak{p}_0$ and $\mu$ has full support, it holds that $(\psi_{\vec{0}})_*\mu_{T^{k+1}(X)}^x=\mu_{C_G^{k+1}(X)}^x$ (see Remark \ref{uniqueness of CSM}). Similarly, as $\mathfrak{p}_0\circ \omega=\pi_T$ and $\mathfrak{p}_0,\  \omega$ are continuous maps, one has that $\omega_*\mu_{T^{k+1}(X)}^x=\mu_{C_G^{k+1}(X)}^x$. Notice that for $c\in C_G^{k+1}(X)$, $\mathfrak{p}_0(\psi_v(\mathbf{t}))=c_v$ for any $\mathbf{t}\in \omega^{-1}(c)$. Therefore $ (\psi_{v})_*\eta^c=\mu_{C_G^{k+1}(X)}^{c_v}$.

Let $T^{k+1}_{\id}(G)\subset T^{k+1}(G)$ be the subgroup such that
$$T^{k+1}_{\id}(G)=\{\mathbf{t}\in T^{k+1}(G):\mathbf{t}_{\vec{1}}=\id\}.$$
Notice that $\mu_{T^{k+1}(X)}$ is $T^{k+1}(G)$-invariant, therefore  
for every $t\in T^{k+1}_{\id}(G)$, $t_{*}\mu_{T^{k+1}(X)}^{x}$ is a measure disintegration w.r.t.\  $\pi_T$ and therefore $t_{*}\mu_{T^{k+1}(X)}^{x}=\mu_{T^{k+1}(X)}^x$ for a.s. $x$. As $t$ is a continuous function, $t_{*}\mu_{T^{k+1}(X)}^{x}$ is also a continuous system of measures. As $\mu$ has full support on $X$,  $t_{*}\mu_{T^{k+1}(X)}^{x}=\mu_{T^{k+1}(X)}^x$ for any $x$. For $v\neq \vec{0}$, as $\psi_{v}(T^{k+1}_{\id}(G))=\mathcal{HK}^{k+1}(G)$, one has that $(\psi_{v})_{*}\mu_{T^{k+1}(X)}^{x}$ is a $\mathcal{HK}^{k+1}(G)$-invariant
measure on $C_G^{k+1}(X)$. Notice that $(C_G^{k+1}(X),\mathcal{HK}^{k+1}(G))$
is uniquely ergodic, therefore $(\psi_{v})_{*}\mu_{T^{k+1}(X)}^{x}=\mu_{C_G^{k+1}(X)}$.

As $g^{[k+1]}$ is a continuous function, we conclude similarly as above that  for all $x\in X$ and $g\in G$, $(g^{[k+1]})_*\mu_{C_G^{k+1}(X)}^x=\mu_{C_G^{k+1}(X)}^{gx}$.
\end{proof}

\subsection{Alternating sum formula for nilcycles on tricubes.\label{subsec:measure on tricube}}

\begin{lem}
\label{lem: tricube measure nilcycle} Let $(X,G,\mu)$ be a distal $(2k+2)$-cube uniquely ergodic system and $\rho\colon C_G^{k+1}(X)\to A$ a nilcycle.
Then for $\mu$-a.e. $x$:
\begin{equation}
\sum_{v\in\{0,1\}^{k+1}}(-1)^{|v|}\rho(\psi_{v}(\mathbf{t}))=\rho(\omega(\mathbf{t})),\ \mu_{T^{k+1}(X)}^{x}-\mathrm{a.e.}\mathbf{t};\label{eq: 1-tri-cube nilcycle-1}
\end{equation}
and for $\mu_{C_G^{k+1}(X)}$-a.e. $c$:
\begin{equation}
\sum_{v\in\{0,1\}^{k+1}}(-1)^{|v|}\rho(\psi_{v}(\mathbf{t}))=\rho(c),\ \eta^{c}-\mathrm{a.e.}\mathbf{t}.\label{eq:2-tri-cube nilcycle-1}
\end{equation}
\end{lem}

\begin{proof}
Let $\mathbf{t}\in T^{k+1}(X)$ be a tricube. Let $\alpha=\{v,v'\}$ be an edge in $\{0,1\}^{k+1}$, i.e.\  there exists $i_{1}\in\{1,\ldots,k+1\}$ such that
$v_{i_{1}}\neq v'_{i_{1}},v'(i)=v'(i)$ for $i\neq i_{1}$. By the definition
of $T^{k+1}(X)$, $\psi_{v}(\mathbf{t}),\psi_{v'}(\mathbf{t})\in C_G^{k+1}(X)$.
We also notice that $\psi_{v}(\mathbf{t})\text{ and }\psi_{v'}(\mathbf{t})$
have a common $(k+1)$-face, which means that there exist two cube isomorphisms
$\sigma_{1},\sigma_{2}:C_G^{k+1}(X)\rightarrow C_G^{k+1}(X)$ such that
$(\sigma_{1}(\psi_{v}(\mathbf{t})),\sigma_{2}(\psi_{v'}(\mathbf{t})))\in\mathcal{P}^{k+1}(X)$.
Let $\tau_{\alpha}:T^{k+1}(X)\rightarrow\mathcal{P}^{k+1}(X)$ be
the map defined by $\tau_{\alpha}(\mathbf{t})=(\sigma_{1}(\psi_{v}(\mathbf{t})),\sigma_{2}(\psi_{v'}(\mathbf{t})))$. By Lemma \ref{measure-preserving}, as $X$ is a distal $(2k+2)$-cube uniquely ergodic system, $(\psi_{v})_{*}\mu_{T^{k+1}(X)}=\omega_{*}\mu_{T^{k+1}(X)}=\mu_{C_G^{k+1}(X)}$. Similarly, as $X$ is a distal $(2k+2)$-cube uniquely ergodic system, $(\tau_{\alpha})_{*}\mu_{T^{k+1}(X)}=\mu_{\mathcal{P}^{k+1}(X)}$.

Following the proof of \cite[Lemma 3.3.31]{candela2016alg_notes}, we see that if the 'glueing property' (i.e.\  $\rho(c_1\|c_2)=\rho(c_1)+\rho(c_2)$) of $\rho$ holds for \textit{every} $(c_1,c_2)\in \mathcal{P}^{k+1}(X)$, then the outer cube $\omega(\mathbf{t})$ may be expressed as a 'glueing sum' of the cubes $\psi_v(\mathbf{t})$. Thus under this assumption, (\ref{eq: 1-tri-cube nilcycle-1})
and (\ref{eq:2-tri-cube nilcycle-1}) hold for every $c$ and every $x$. Now let $D\subset T^{k+1}(X)$ be the set of measure $1$ which
corresponds to the intersection of pullbacks by $\tau_{\alpha}$
of full measure sets which have properties 1 and 2 of nilcycles
for all edges $\alpha$. Using the set $D$, \eqref{eq: 1-tri-cube nilcycle-1}
and (\ref{eq:2-tri-cube nilcycle-1}) follow from a similar argument to the one in the proof of \cite[Lemma 3.3.31]{candela2016alg_notes} (however  in this case equality holds only \textit{almost} everywhere), where one also uses Lemma \ref{lem:CSM circ} which implies $\mu_{T^{k+1}(X)}^{x}=\int_{\mathfrak{p}_{0}^{-1}(x)}\eta^{c}d\mu_{C_G^{k+1}(X)}^{x}$. 
\end{proof}
\subsection{Function bundles.}
\label{subsec:Function-bundles}\label{subsec:Iterated-Haar-measure.}
\begin{defn}
\label{def:Let--be function bundle}Let $\pi:X\rightarrow Y$ be a
continuous and surjective map between two compact metric spaces. Let $\mu_{\pi}^{\bullet}=\{\mu_{\pi}^{y}\}_{y\in Y}$
be a continuous system of measures. Let $A$ be a compact space. The \strong{function bundle} of $\pi:X\rightarrow Y$ is the space $\mathcal{L}(X\xrightarrow{\pi}Y,A)=\stackrel{\circ}{\bigcup}_{y\in Y}\mathcal{B}(\pi^{-1}(y),A,\mu_{\pi}^{y})$,
where $\mathcal{B}(B,S,\nu_{B})$ is the quotient of the set of Borel
measurable functions $f:B\rightarrow S$ by the equivalence relation
$\sim$ defined by $f\sim g\Longleftrightarrow\nu_{B}(b\in B:f(b)\neq g(b))=0$.
In this article, the equivalence classes of $f$ will be denoted by
$[f]$. Let $\hat{\pi}:\mathcal{L}(X\xrightarrow{\pi}Y,A)\rightarrow Y$ be the map 
defined by $\hat{\pi}(f)=y$ if $f\in\mathcal{B}(\pi^{-1}(y),A,\mu_{\pi}^{y})$. According to \cite[Proposition 2.3.3]{candela2016cpt_notes}, the map $\hat{\pi}$ is continuous.

The topology of $\mathcal{L}(X\xrightarrow{\pi}Y,A)$ is the coarsest
topology making the functions 

\[
\phi_{F_{1},F_{2}}:f\rightarrow\int_{\pi^{-1}(\hat{\pi}(f))}F_{1}(f(v))F_{2}(v)d\mu_{\pi}^{\hat{\pi}(f)}(v),
\]
continuous for every pair of continuous functions $F_{1}:A\rightarrow\mathbb{C}$
and $F_{2}:X\rightarrow\mathbb{C}$. In particular, if $A$ is a compact
abelian group, we may assume that $F_{1}$ is a character of $A$,
i.e.\  $F_{1}:A\rightarrow S^{1}\text{ and }F_{1}(a+b)=F_{1}(a)F_{1}(b)$.\footnote{Let $\hat{A}$ be the group of  characters of $A$, i.e.\  the dual
group of $A$. By Pontryagin duality theorem, $A=\hat{\hat{A}}$.
Thus for $a\neq b\in A$, $\hat{\hat{a}}\neq \hat{\hat{b}}$,
i.e.\  there exists a character $\chi\in\hat{A}$ such that $a(\chi)\neq b(\chi)$.
As $a(\chi)=\chi(a)$, one has that $\hat{A}$ separates $A$. By Stone-Weierstrass
theorem, the subalgebra generated by $\hat{A}$ is dense in $C(A)$.}
\end{defn}
\section{Proof of the main theorem.\label{sec:The-topological-nilspace}}

In this section, we prove Theorem \ref{thm:main theorem}. An overview of the steps of the proof is given by the captions of the various subsections of this section.
\subsection{Definitions}\label{subsec:Definitions}
Recall that in Theorem \ref{thm:main theorem} we are given a distal $(2k+2)$-cube uniquely ergodic system $(X,G,\mu)$. Using Lemma \ref{lem: tricube measure nilcycle}, we define a full
measure set $X'$ such that for all $x\in X'$, for $\mu_{T^{k+1}(X)}^{x}$-a.e.
$\mathbf{t}\in T^{k+1}(X)$,
\begin{equation}\label{def:X'}
\sum_{\nu\in\{0,1\}^{k+1}}(-1)^{|\nu|}\rho(\psi_{\nu}(\mathbf{t}))=\rho(\omega(\mathbf{t})).
\end{equation}

\noindent
Since the acting group $G$ is countable, we can assume that $X'$
is $G$-invariant. As $\mu$ is the unique ergodic measure of a
distal system, $\mu$ has full support. Thus we know that $$X=\overline{X'}.$$ 
\noindent
 For $v\in \{0,1\}^{n}$, define $\mathfrak{p}_{v,n}:C_G^{n}(X)\rightarrow X:\ \mathbf{x}\rightarrow\mathbf{x}_{v}$ as the
projection from $C_G^{n}(X)$ to the $v$-coordinate. We also write $\mathfrak{p}_{v}=\mathfrak{p}_{v,n}$ if no confusion arises, as well as $\mathfrak{p}_{0}$ instead of $\mathfrak{p}_{\vec{0}}$. In particular, for $v=\vec{0}$, $n=k+1$, define 
$$M'=\{-\rho_{x}+a:\ x\in X',a\in A\}\subset\mathcal{L}(C_G^{k+1}(X)\xrightarrow{\mathfrak{p}_{0}}X,\ A),\footnote{The reason why we use $-\rho_x+a$ and not $\rho_x+a$ is elucidated by the proof of Lemma \ref{fact:well defined and continuous}.}$$ where $\rho_{x}:\mathfrak{p}_0^{-1}(x)\rightarrow A:c\rightarrow \rho(c)$ is an element in $\mathcal{B}(\mathfrak{p}_0^{-1}(x),A,\mu_{C_G^{k+1}(X)}^x)$. Define $$M=\overline{M'}.$$
\noindent
Notice that $\mathfrak{p}_0$ also induces a natural map:
$\hat{\mathfrak{p}}_{0}:\mathcal{L}(C_G^{k+1}(X)\xrightarrow{\mathfrak{p}_{0}}X,A)\rightarrow X$ (see Definition \ref{def:Let--be function bundle}).

\subsection{The space $M$ is compact.}\label{subsec:The-compactness-of}
In this section we establish that $M$ is compact. Our approach is heavily influenced by \cite{CS12} and we use extensively \cite{candela2016cpt_notes} and \cite{candela2016alg_notes}. See Remark \ref{rem for candela}.

\begin{defn}
\label{def:the definition of E} Define $\mathcal{E}:\mathcal{L}(C_G^{k+1}(X)\xrightarrow{\mathfrak{p}_{0}}X,\ A)\rightarrow\mathcal{L}(C_G^{k+1}(X)\times_{X}C_G^{k+1}(X)\xrightarrow{\pi'}X,\ A)$ by
$\mathcal{E}(f)(c_{0},c_{1})=f(c_{0})-f(c_{1})$.
\end{defn}

\begin{lem}
\label{lem:E-is-continuous.}$\mathcal{E}$ is continuous.
\end{lem}

\begin{proof}
Let $\{g_{n}\}_{n\in \N}$ converge to $g$ in $\mathcal{L}(C_G^{k+1}(X)\xrightarrow{\mathfrak{p}_{0}}X,\ A)$,
i.e., for any character $\chi\in\hat{A}$ and any continuous function
$f:\ C_G^{k+1}(X)\rightarrow\mathbb{C}$,
\begin{equation}
\lim_{n\rightarrow\infty}\int_{\mathfrak{p}_{0}^{-1}(x_{n})}\chi( g_{n}(c))f(c)d\mu_{C_G^{k+1}(X)}^{x_{n}}(c)=\int_{\mathfrak{p}_{0}^{-1}(x)}\chi(g(c))f(c)d\mu_{C_G^{k+1}(X)}^{x}(c).\label{eq:*1}
\end{equation}
The continuous functions of the form $F=F_{1}(c_{0})F_{2}(c_{1})$
are dense in $C(C_G^{k+1}(X)\times_{X}C_G^{k+1}(X))$ in the uniform norm by the Stone-Weierstrass theorem. Thus to establish the continuity of $\mathcal{E}$, we just need to show that
for any continuous function $F=F_{1}(c_{0})F_{2}(c_{1}):\ C_G^{k+1}(X)\times_X C_G^{k+1}(X)\rightarrow\mathbb{C}$,

\begin{equation}
\begin{array}{c}
\underset{n\rightarrow\infty}{\lim}\int_{\mathfrak{p}_{0}^{-1}(x_{n})\times\mathfrak{p}_{0}^{-1}(x_{n})}\chi(g_{n}(c_{0})-g_{n}(c_{1}))F(c_{0},c_{1})d(\mu_{C_G^{k+1}(X)}^{x_{n}}\times\mu_{C_G^{k+1}(X)}^{x_{n}})(c_0,c_1)\\
=\int_{\mathfrak{p}_{0}^{-1}(x)\times\mathfrak{p}_{0}^{-1}(x)}\chi(g(c_{0})-g(c_{1}))F(c_{0},c_{1})d(\mu_{C_G^{k+1}(X)}^{x}\times\mu_{C_G^{k+1}(X)}^{x})(c_0,c_1).
\end{array}\label{continuous 1}
\end{equation}
Since the right-hand side of \eqref{continuous 1} equals the following product of integrals
\begin{equation}\label{eq:product}
\begin{array}{l}
\int_{\mathfrak{p}_{0}^{-1}(x_{n})}\overline{\chi(g_{n}(c_{1}))}F_{2}(c_{1})d\mu_{C_G^{k+1}(X)}^{x_{n}}(c_1)\int_{\mathfrak{p}_{0}^{-1}(x_{n})}\chi(g_{n}(c_{0}))F_{1}(c_{0})d\mu_{C_G^{k+1}(X)}^{x_{n}}(c_0),
\end{array}
\end{equation}
we choose $f=F_{1}$ in Equation \eqref{eq:*1} to show that the latter integral in the product of integrals \eqref{eq:product} converges to
\[
\begin{array}{l}
\int_{\mathfrak{p}_{0}^{-1}(x)}\chi(g(c_0))F_{1}(c_{0})d\mu_{C_G^{k+1}(X)}^{x}(c_{0}),\end{array}
\]
and similarly for the first integral in \eqref{eq:product}. This establishes (\ref{continuous 1}).
\end{proof}
The following lemma is used to prove Lemma \ref{lem:cov is cont}.
\begin{lem}
\label{lem: separate measurable set} Let $\pi:(X,\lambda_X)\rightarrow (Y,\lambda_Y)$ be
a CMD map w.r.t.\  the continuous system of measures $\{\gamma_{y}\}_{y\in Y}$.
Let $g:\ X\rightarrow A$ be a measurable map, where $A$ is a compact
abelian group. If there exists a measurable set $V_{0}\subset Y$
such that $g|_{\pi^{-1}(y)}$ is not a.e.-$\gamma_{y}$ constant for
any $y\in V_{0}$ and $\lambda_{Y}(V_{0})>0$, then there exist a
character $\chi$ of $A$, $b\in\mathbb{R}$ and measurable sets $U_{2},U_{3}\subset X$
with $\lambda_{X}(U_{2}),\lambda_{X}(U_{3})>0$ such that:
\begin{enumerate}
\item
 $\pi(U_{2})=\pi(U_{3})$;
\item
$0<\gamma_{y}(U_{3})\leq\gamma_{y}(U_{2})\text{ for any }y\in\pi(U_{2})$;
\item $\chi(g(U_{2}))\subset e([b,b+1/4]),\ \chi(g(U_{3}))\subset e([b+1/2,b+3/4])$.
\end{enumerate}
Here $e(x)=e^{2\pi i x}$ and $e([c,d])=\{e(x):x\in [c,d]\}$.
\end{lem}
\begin{proof}
Let $\{\chi_{n}\}_{n\in\mathbb{N}}$ be the set of characters of $A$.
Fix $y\in V_{0}$. Since $\{\chi_{n}\}_{n\in\mathbb{N}}$ separate
$A$, we can find $\chi_{n_{y}}$ such that $\chi_{n_{y}}(g|_{\pi^{-1}(y)})$
is not constant, i.e.\ 
\[
\gamma_{y}\times \gamma_{y}(\{(x,x')\in\pi^{-1}(y)\times\pi^{-1}(y):\ \chi_{n_{y}}(g(x))\overline{\chi_{n_{y}}(g(x'))}\in e((0,1))\})>0,
\]
w.l.o.g.,
$$
\gamma_{y}\times\gamma_{y}\{(x,x')\in\pi^{-1}(y)\times\pi^{-1}(y):\ \chi_{n_{y}}(g(x))\overline{\chi_{n_{y}}(g(x'))}\in e((0,\frac{1}{2}])\}>0.
$$
Since $(0,\frac{1}{2}]=\bigcup_{n\geq 6,30\leq k\leq2^{n-1}-1.}[\frac{k}{2^{n}},\frac{k+1}{2^{n}}]$,
there exists $[\delta_{1},\delta_{1}+\epsilon_{1}]$ such that $0<\delta_{1}\leq\frac{1}{2},\ 0<\epsilon_{1}\leq\frac{\delta_{1}}{30}$
and
\[
\gamma_{y}\times\gamma_{y}\{(a,a')\in\pi^{-1}(y)\times\pi^{-1}(y):\ \chi_{n_{y}}(g(x))\overline{\chi_{n_{y}}(g(x')})\in e([\delta_{1},\delta_{1}+\epsilon_{1}])\}>0.
\]
Since
\[
\begin{array}{l}
\{(b,b')\in[0,1]\times[0,1]:b-b'\in[\delta_{1},\delta_{1}+\epsilon_{1}]\}\\
\subset\bigcup_{1\leq k\leq2^{n}-1,2^{n}\geq2/\epsilon_{1}}[\frac{k}{2^{n}}+\delta_{1},\frac{k+1}{2^{n}}+\delta_{1}+\epsilon_{1}]\times[\frac{k}{2^{n}},\frac{k+1}{2^{n}}],
\end{array}
\]
there exists $[\delta_{2},\delta_{2}+\epsilon_{2}]$ such that $\epsilon_{2}<\frac{1}{2}\epsilon_{1}$
and
\[
\begin{array}{l}
\gamma_{y}\{x\in\pi^{-1}(y):\ \chi_{n_{y}}(g(x))\in e([\delta_{2},\delta_{2}+\epsilon_{2}])\}>0\text{ and }\\
\gamma_{y}\{x\in\pi^{-1}(y):\ \chi_{n_{y}}(g(x))\in e([\delta_{2}+\delta_{1},\delta_{2}+\epsilon_{2}+\delta_{1}+\epsilon_{1}])\}>0.
\end{array}
\]

If $\delta_{1}\in[\frac{1}{3},\frac{1}{2}]$, since $\epsilon_{1}\leq\frac{\delta_{1}}{30}<\frac{1}{60}$,
$\epsilon_{2}<\frac{1}{2}\epsilon_{1}$, we can define $b_{y}=\delta_{2}+\epsilon_{1}+\epsilon_{2}-\frac{1}{5}$. Then one has that  $[\delta_{2},\delta_{2}+\epsilon_{2}]\subset[b_{y},b_{y}+\frac{1}{5}]$
and $[\delta_{1}+\delta_{2},\delta_{1}+\delta_{2}+\epsilon_{1}+\epsilon_{2}]\subset[b_{y}+\frac{2}{4},b_{y}+\frac{7}{10}]$.

If $\delta_{1}<\frac{1}{3}$, let $M$ be a positive integer such
that $M\delta_{1}\in[\frac{1}{3},\frac{2}{3})$. Let
\[
b_{y}=\begin{cases}
M(\delta_{2}+\epsilon_{1}+\epsilon_{2})-\frac{1}{5} & \text{if }M\delta_{1}\leq\frac{1}{2}\\
M\delta_{2} & \text{if }M\delta_{1}>\frac{1}{2}
\end{cases}.
\]
Since $M\epsilon_{1}\leq\frac{M\delta_{1}}{30}<\frac{1}{45}$, $M\epsilon_{2}<\frac{1}{2}M\epsilon_{1}$, one
has $[M\delta_{2},M(\delta_{2}+\epsilon_{2})]\subset[b_{y},b_{y}+\frac{1}{5}]$
and $[M(\delta_{1}+\delta_{2}),M(\delta_{1}+\delta_{2}+\epsilon_{1}+\epsilon_{2})]\subset[b_{y}+\frac{2}{4},b_{y}+\frac{7}{10}]$.
In this case, we replace $\chi_{n_{y}}$ by $\chi_{n_{y}}^{M}$. Note
that $\chi_{n_{y}}^{M}$ is still a character. 

Summarizing the two cases, one has that for each $y\in V_0$ there exist a character $\chi_{n_{y}}$ and $b_{y}\in[0,1)$
such that
\[
\begin{array}{l}
\gamma_{y}\{x\in\pi^{-1}(y):\ \chi_{n_{y}}(g(x))\in e([b_{y},b_{y}+\frac{1}{5}])\}>0,\\
\gamma_{y}\{x\in\pi^{-1}(y):\ \chi_{n_{y}}(g(x))\in e([b_{y}+\frac{1}{2},b_{y}+\frac{7}{10}])\}>0.
\end{array}
\]
Notice that for each $b_y$, we can find $0\leq j_y\leq19$ such that $b_{y}\in[\frac{j_y}{20},\frac{j_y+1}{20})$.
Thus for any $b_{y}\in V_{0}$, there exist $0\leq j_y\leq19$ and a
character $\chi_{n_{y}}$ such that
\begin{equation}
\begin{array}{l}
\gamma_{y}\{x\in\pi^{-1}(y):\ \chi_{n_{y}}(g(x))\in e([\frac{j_y}{20},\frac{j_y}{20}+\frac{1}{4}])\}>0,\\
\gamma_{y}\{x\in\pi^{-1}(y):\ \chi_{n_{y}}(g(x))\in e([\frac{j_y}{20}+\frac{1}{2},\frac{j_y}{20}+\frac{3}{4}])\}>0.
\end{array}\label{eq:1/4 to 3/4}
\end{equation}
For any $n,j$, define:
\[
\hat{V}_{n,j}=\left\{ y\in V_{0}:\ \begin{array}{c}
\gamma_{y}\{x\in\pi^{-1}(y):\ \chi_{n}(g(x))\in e([\frac{j}{20},\frac{j}{20}+\frac{1}{4}])\}>0\text{ and }\\
\gamma_{y}\{x\in\pi^{-1}(y):\ \chi_{n}(g(x))\in e([\frac{j}{20}+\frac{1}{2},\frac{j}{20}+\frac{3}{4}])\}>0.
\end{array}\right\}.
\]
The set $\hat{V}_{n,j}$ is measurable since $\chi_{n},g$ are measurable
functions and for any measurable set $U\subset A$, $y\rightarrow\gamma_{y}(U)$
is a measurable function from $Y$ to $[0,1]$(see Theorem \ref{thm:(measure-disintegration-theorem)}).
By Equation (\ref{eq:1/4 to 3/4}), $V_{0}\subset \bigcup_{n\in\mathbb{N},0\leq j\leq19}\hat{V}_{n,j}$.
Since $\lambda_{Y}(V_{0})>0$, there exists $n,j$ such that $\hat{V}_{n,j}$
has positive $\lambda_{X}$-measure. We replace $V_{0}$ by $\hat{V}_{n,j}$. Let $\chi=\chi_{n}$. Choosing $b=\frac{j}{20}$, for any $y\in V_{0}=\hat{V}_{n,j}$,
\[
\begin{array}{l}
\gamma_{y}\{x\in\pi^{-1}(y):\ \chi(g(x))\in e([b,b+\frac{1}{4}])\}>0,\\
\gamma_{y}\{x\in\pi^{-1}(y):\ \chi(g(x))\in e([b+\frac{2}{4},b+\frac{3}{4}])\}>0.
\end{array}
\]
Let $U_{2}'=\{x\in\pi^{-1}(V_{0}):\ \chi(g(x))\in e([b,b+\frac{1}{4}])\}$,
$U_{3}'=\{x\in\pi^{-1}(V_{0}):\ \chi(g(x))\in e([b+\frac{1}{2},b+\frac{3}{4}])\}$.
one has
\[
\begin{array}{l}
\bullet\ \pi(U_{2})=\pi(U_{3})=V_{0};\\
\bullet\ \gamma_{y}(U_{2}),\gamma_{y}(U_{3})>0\text{ for any }y\in\pi(U_{2});\\
\bullet\ \chi(g(U_{2}))\subset e([b,b+1/4]),\ \chi(g(U_{3}))\subset e([b+1/2,b+3/4]).
\end{array}
\]
We notice that we can replace $b$ by $b+\frac{1}{2}$ to exchange
$U_{2}'$ and $U_{3}'$. Thus we can assume that $V_{1}=\{y\in V_{0}:\ \gamma_{y}(U_{2}')\geq\gamma_{y}(U_{3}')>0.\}$
has positive $\lambda_{y}$-measure. Let $U_{2}=\pi^{-1}(V_{1})\cap U_{2}'$
and $U_{3}=\pi^{-1}(V_{1})\cap U_{3}'$, then one has properties $(1),(2),(3)$ in the statement of the Lemma.
\end{proof}
\begin{lem}
\label{lem:cov is cont} Let $X,Y,Z$ be compact spaces and $A$ a compact abelian group. Let $\mu_{\phi}^{\bullet}$ be a
continuous system of measures for $\phi:X\rightarrow Z$, $\nu_{\psi}^{\bullet}$ a continuous system of measures for
$\psi:Y\rightarrow Z$ and $\gamma_{\pi}^{\bullet}$ a continuous system of measures for $\pi:X\rightarrow Y$
such that $\phi=\psi\circ\pi$ and $\int_{y\in\psi^{-1}(z)}\gamma_{\pi}^{y}d\nu_{\psi}^{z}=\mu_{\phi}^{z}$.
Define
\[
\mathcal{{L}}_{\pi}(X\xrightarrow{\phi}Z,A)=\{[f\circ\pi]|\,f\in\mathcal{{L}}(Y\xrightarrow{\psi}Z,A)\}\subset\mathcal{{L}}(X\xrightarrow{\phi}Z,A),
\]
then $\mathcal{{L}}_{\pi}(X\xrightarrow{\phi}Z,A)$ is closed in $\mathcal{{L}}(X\xrightarrow{\phi}Z,A)$
and $i:\mathcal{{L}}_{\pi}(X\xrightarrow{\phi}Z,A)\rightarrow\mathcal{{L}}(Y\xrightarrow{\psi}Z,A)\}$
given by $i([f\circ\pi])=f$ is well-defined and continuous.
\end{lem}

\begin{proof}
First we note $\mathcal{{L}}_{\pi}(X\xrightarrow{\phi}Z,A)$ is well
defined as the composition of two Borel functions is Borel. To show
the map $i$ is well-defined, we need to show that $[f_{1}\circ\pi]=[f_{2}\circ\pi]$
implies $[f_{1}]=[f_{2}]$. By Definition \ref{def:Let--be function bundle}, one may associate to the maps $\phi:X\rightarrow Z$ and
$\psi:Y\rightarrow Z$, the maps   $\hat{\phi}:\mathcal{L}(X\xrightarrow{\phi}Z,A)\rightarrow Z$  and  $\hat{\psi}:\mathcal{L}(Y\xrightarrow{\psi}Z,A)\rightarrow Z$ respectively. Since $\hat{\psi}(f)=\hat{\phi}(f\circ\pi)$, if $\hat{\psi}(f_{1})\neq\hat{\psi}(f_{2})$,
then $[f_{1}\circ\pi]\neq[f_{2}\circ\pi]$. Thus we assume that $\hat{\psi}(f_{1})=\hat{\psi}(f_{2})=z$
and $[f]_{1}\neq[f_{2}]$. Thus for a compatible metric $d_A$ for $A$:
\[
\int_{\psi^{-1}(z)}d_A(f_{1}(y),f_{2}(y))d\nu_{\psi}^{z}(y)>0.
\]
Since $\int_{y\in\psi^{-1}(z)}\gamma_{\pi}^{y}d\nu_{\psi}^{z}=\mu_{\phi}^{z}$,
\[
\begin{array}{l}
\int_{\phi^{-1}(z)}d_A\left(f_{1}(\pi (x)),f_{2}(\pi(x))\right)d\mu_{\phi}^{c}(x)\\
=\int_{y\in\psi^{-1}(z)}\int_{\pi^{-1}(y)}d_A\left(f_{1}(y),f_{2}(y)\right) d\gamma_{\pi}^{y}(x)d\nu_{\psi}^{z}(y)
=\int_{\psi^{-1}(z)}d_A(f_{1}(y),f_{2}(y))d\nu_{\psi}^{z}(y)>0,
\end{array}
\]
which implies that $[f_{1}\circ\pi]\neq[f_{2}\circ\pi]$. Let us now show that $i$ is continuous. Assume $g_{i}=f_{i}\circ\pi\in\mathcal{{B}}_{z_{i}}(X,Z,A)\rightarrow g=f\circ\pi\in\mathcal{{B}}_{z}(X,Z,A)$,
where $z_{i},z\in Z$ with $z_{i}\rightarrow z$ and $f_{i},f\in\mathcal{{L}}(Y\xrightarrow{\psi}Z,A)$.
Given $\chi:A\rightarrow\mathbb{C}$ a character and $F:Y\rightarrow\mathbb{C}$
a continuous function, one has by the fact that $\int_{\pi^{-1}(y)}f\circ\pi d\gamma_{\pi}^{y}=f(y)$
and $\int_{y\in\psi^{-1}(z)}\gamma_{\pi}^{y}d\nu_{\psi}^{z}=\mu_{\phi}^{z}$ that

\[
\int_{y\in\psi^{-1}(z_{i})}\chi(f_{i}(y))F(y)~d\nu_{\psi}^{z_{i}}(y)\stackrel{y=\pi(x)}{=}\int_{x\in\phi^{-1}(z_{i})}\chi(g_{i}(x))F(\pi(x))~d\mu_{\phi}^{z_{i}}(x).
\]
Notice that $g_{i}\rightarrow g$. Thus the right-hand side of the last equation converges as $i\rightarrow\infty$
to $\int_{\phi^{-1}(z)}\chi(g(x))F(\pi(x))~d\mu_{\phi}^{z}(x)$. As this integral is equal to

\[
\begin{array}{l}
\int_{x\in\phi^{-1}(z)}\chi(f(\pi(x)))F(\pi(x))~d\mu_{\phi}^{z}(x)
\stackrel{y=\pi(x)}{=}\int_{y\in\psi^{-1}(z)}\chi(f(y))F(y)~d\nu_{\psi}^{z}(y),
\end{array}
\]
it follows that $f_{i}\rightarrow f$. Now we show that $\mathcal{{L}}_{\pi}(X\xrightarrow{\phi}Z,A)$
is closed in $\mathcal{{L}}(X\xrightarrow{\phi}Z,A)$. Let $f_{i}\circ\pi\in\mathcal{{L}}_{\pi}(X\xrightarrow{\phi}Z,A)$
such that $g_{i}:= f_{i}\circ\pi\rightarrow g\in\mathcal{{L}}(X\xrightarrow{\phi}Z,A)$.
We claim that there exists $f\in\mathcal{{L}}(Y\xrightarrow{\psi}Z,A)\}$
such that $g=f\circ\pi$, which implies $g\in\mathcal{{L}}_{\pi}(X\xrightarrow{\phi}Z,A)$,
which means that $\mathcal{{L}}_{\pi}(X\xrightarrow{\phi}Z,A)$ is
closed. Let us denote $z=\hat{\phi}(g)\in Z$. If the claim is wrong, there exists
a positive measure set $V_{0}\subset\psi^{-1}(z)$ such that for any
$y\in V_{0}$, $g|_{\pi^{-1}(y)}$ is not a constant function. From
Lemma \ref{lem: separate measurable set}, there exists a character
$\chi_{1}$, $\beta\in[0,1)$ and measurable sets $U_{2},U_{3}\subset\psi^{-1}(z)$
with $\mu_{\phi}^{z}(U_{2}),\mu_{\phi}^{z}(U_{3})>0$ such that

\begin{equation}\begin{array}{l}\label{eq:< measure}
\bullet\ \pi(U_{2})=\pi(U_{3});\\
\bullet\ 0<\gamma_{\pi}^y(U_{3})\leq\gamma_{\pi}^{y}(U_{2})\text{ for any }b\in\pi(U_{2});\\
\bullet\ \chi_{1}(g(U_{2}))\subset e([\beta,\beta+1/4]),\ \chi_{1}(g(U_{3}))\subset e([\beta+1/2,\beta+3/4]).
\end{array}
\end{equation}
As $0<\gamma_{\pi}^y(U_{3})$ for $y\in V_0$ and $V_0$ has positive measure, $\mu_{\phi}^{z}(U_{3})>0$. Let $\epsilon_2>0$ such that 
\begin{equation}\label{ep2measure U3}
\epsilon_{2}<\mu_{\phi}^{z}(U_{3})/2.
\end{equation}
Let $F_{2},F_{3}$ be continuous functions such that
\[
\|F_{2}-1_{U_{2}}\|_{L^{1}(\phi^{-1}(z))}<\epsilon_{2}\text{ and }\|F_{3}-1_{U_{3}}\|_{L^{1}(\phi^{-1}(z))}<\epsilon_{2}.
\]

Consider the continuous function\[
G:S^{1}\rightarrow S^{1}:G(e(\alpha))=\begin{cases}
1 & \alpha\in[\beta,\beta+\frac{1}{4}]\\
-1 & \alpha\in[\beta+\frac{2}{4},\beta+\frac{3}{4}]\\
e(2(\alpha-\beta)-\frac{1}{2}) & \alpha\in[\beta+\frac{1}{4},\beta+\frac{2}{4}]\\
e(2(\alpha-\beta)) & \alpha\in[\beta+\frac{3}{4},\beta+1]
\end{cases}
\]
and let $\chi_{0}=G\circ\chi_{1}$. By the
definition of $G$, $|\chi_0(a)|=1$ for any $a\in A$, $\chi_{0}(g(x))=1$ for any $x\in U_{2}$, $\chi_{0}(g(x))=-1$
for any $x\in U_{3}$. Let us estimate
\begin{equation}
\begin{array}{l}
M_1=|\int_{\phi^{-1}(z)}(\chi_{0}(g(x))(F_{2}(x)-F_{3}(x)))d\mu_{\phi}^{z}(x)|\\
\geq|\int_{\phi^{-1}(z)}(\chi_{0}(g(x))1_{U_{2}}(x)-\chi_{0}(g(x))1_{U_{3}}(x))d\mu_{\phi}^{z}(x)|-2\epsilon_{2}\\
=\mu_{\phi}^{z}(U_{2})+\mu_{\phi}^{z}(U_{3})-2\epsilon_{2}.
\end{array}\label{eq:F to U}
\end{equation}
Here we used strongly the properties of $U_{2}$ and $U_{3}$. As
$g_{i}=f_i\circ\pi \rightarrow g$, we deduce from the equality in \eqref{eq:F to U} that 

\begin{equation}\label{F to U 3}\begin{array}{ll}M_1=\underset{i\rightarrow\infty}{\lim}|\int_{\phi^{-1}(z_{i})}\chi_{0}(f_{i}(\pi(x)))(F_{2}(x)-F_{3}(x))d\mu_{\phi}^{z_{i}}(x)|\\
=\underset{i\rightarrow\infty}{\lim}|\int_{\psi^{-1}(z_{i})}\int_{\pi^{-1}(y)}\chi_{0}(f_{i}(\pi(x)))(F_{2}(x)-F_{3}(x))d\gamma^y_{\phi}(x)d\nu_{\psi}^{z_i}(y)|\\
\leq\underset{i\rightarrow\infty}{\lim}\int_{\psi^{-1}(z_{i})}|\chi_{0}(f_{i}(y))\int_{\pi^{-1}(y)}(F_{2}(x)-F_{3}(x))d\gamma^y_{\phi}(x)|d\nu_{\psi}^{z_i}(y)\\
= \underset{i\rightarrow\infty}{\lim}\int_{\psi^{-1}(z_{i})}|\int_{\pi^{-1}(y)}(F_{2}(x)-F_{3}(x))d\gamma_{\pi}^{y}(x)|d\nu_{\psi}^{z_{i}}(y).
\end{array}
\end{equation}
Now a standard estimate from above will lead to a contradiction. As $\gamma_{\pi}^{y}$ is continuous system of measures,  we know that 
\begin{equation}
y \rightarrow \int_{\pi^{-1}(y)}(F_{2}(x)-F_{3}(x)))d\gamma_{\pi}^{y}(x)\text{ is continuous}.\label{eq: continuous int}
\end{equation}
As $\{\nu_{\psi}^{z}\}$ is a continuous system of measures, by Equation \eqref{F to U 3} and \eqref{eq: continuous int}, one has that

\begin{equation}\label{eq:F to U 2}
\begin{array}{l}
M_1\leq \int_{\psi^{-1}(z)}|\int_{\pi^{-1}(y)}(F_{2}(x)-F_{3}(x)))d\rho_{\pi}^{y}(x)|d\nu_{\psi}^{z}(y)\\
\leq\int_{\psi^{-1}(z)}|\int_{\pi^{-1}(y)}(1_{U_{2}}(x)-1_{U_{3}}(x))d\gamma_{\pi}^{y}(x)|d\nu_{\psi}^{z}(y)+2\epsilon_{2}\\
\stackrel{(\ref{eq:< measure})}{=}\int_{\psi^{-1}(z)}(\gamma_{\pi}^{y}(U_{2})-\gamma_{\pi}^{y}(U_{3}))d\nu_{\psi}^{z}(y)+2\epsilon_{2}
=\mu_{\phi}^{z}(U_{2})-\mu_{\phi}^{z}(U_{3})+2\epsilon_{2}.
\end{array}
\end{equation}
By Equation \eqref{eq:F to U} and \eqref{eq:F to U 2}, $$\mu_{\phi}^{z}(U_{2})+\mu_{\phi}^{z}(U_{3})-2\epsilon_{2}\leq\mu_{\phi}^{z}(U_{2})-\mu_{\phi}^{z}(U_{3})+2\epsilon_{2}. $$ 
Therefore
$\mu_{\phi}^{z}(U_{3})\leq 2\epsilon_{2}$. Recall from Equation \eqref{ep2measure U3} that $\epsilon_{2}<\frac{\mu_{\phi}^{z}(U_{3})}{2}$. We thus have a contradiction.
\end{proof}
\begin{lem}
\label{lem:continuous full support into closed} Let $f:X\rightarrow Y$
be a continuous function between topological spaces and suppose $\mu$
is a measure of full support on $X$. Let $K\subset Y$ be closed. If there exists a Borel set of full $\mu$ measure $X_0$ such that $f(X_{0})\subset K$, then $f(X)\subset K$.
\end{lem}

\begin{proof}
Let $x_{0}\in X$. By assumption for all $n$, $\mu(B_{\frac{1}{n}}(x_{0}))>0$.
Thus $B_{\frac{1}{n}}(x_{0})\cap X_{0}\neq\emptyset$. Therefore we
may find $y_{n}\rightarrow x$ so that $f(y_{n})\in K$. As $K$ is
closed and $f$ continuous we conclude $f(x_{0})\in K$.
\end{proof}
\begin{rem}\label{rem for candela}Following \cite{CS12}, \cite[Definition 3.3.24]{candela2016alg_notes} defines $M=\bigcup_{x\in X}\rho_{x}+A$,
where we define $M=\overline{\bigcup_{x\in X'}-\rho_{x}+A}$. The
closure operation introduces additional complications. In particular, in the following
lemma, our $g(x)=\mathcal{E}(\rho_{x})$ has to be continuously extended
to $X$, for which Lemma \ref{lem:E-is-continuous.}-\ref{lem:continuous full support into closed}
are dedicated. In \cite{candela2016cpt_notes} this difficulty does
not arise and one may define directly $g(x)=\mathcal{E}(\rho_{x})$
for $x\in X$ (in \cite{candela2016cpt_notes} this function is called
$g'$). The last steps in establishing the compactness are not identical
but similar in spirit to the treatment in \cite{candela2016cpt_notes}.
\end{rem}
\begin{lem}
\label{lem:The-map-continuous} There exists a continuous function $D:X\rightarrow\mathcal{{L}}(C_G^{k+1}(X)\times_{X}C_G^{k+1}(X)\xrightarrow{\pi'}X,A)$
such that the map $g:\ X\rightarrow\mathcal{E}(M):\ g(x)=\rho_{x}(\cdot)-\rho_{x}(\cdot\cdot)$
satisfies $g(x)=D(x)$ for all $x\in X'$.
\end{lem}

\begin{proof}
In this proof, let $\mu_{T}$ denote $\mu_{T^{k+1}(X)}$. For each $v\in\{0,1\}_{*}^{k+1}=\{0,1\}^{k+1}\setminus\{\vec{0}\}$,
let

\[
g'_{v:}\ X\rightarrow\mathcal{L}(T^{k+1}(X)\xrightarrow{\pi_{T}}X,A),\ x\mapsto(g'_{v}(x):\ \mathbf{t}\rightarrow\rho(\psi_{v}(\mathbf{t}))).
\]
We now show that $g'_{v}$ is continuous by showing that for any continuous
functions $F_{1}:\ A\rightarrow\mathbb{C}$, $F_{2}:\ T^{k+1}(X)\rightarrow\mathbb{C}$,

\begin{equation}
q:\ x\rightarrow\int_{\pi_{T}^{-1}(x)}F_{1}(g'_{v}(x)(\mathbf{t}))F_{2}(\mathbf{t})d\mu_{T^{k+1}(X)}^{x}
\end{equation}
is continuous. Indeed, by Lusin's Theorem \cite[Theorem 2.24]{rudin2006real}, for any $\epsilon>0$, there exists a
continuous function $F_{3}:\ C_G^{k+1}(X)\rightarrow\mathbb{C}$ such
that $\|F_{1}\circ\rho-F_{3}\|_{L_{1}(\mu_{C_G^{k+1}(X)})}<\epsilon$.
Let $q':\ x\rightarrow\int_{\pi_{T}^{-1}(x)}F_{3}(\psi_{v}(\mathbf{t}))F_{2}(\mathbf{t})d\mu_{T^{k+1}(X)}^{x}(\mathbf{t})$.
By Lemma \ref{measure-preserving}, one has that $(\psi_v)_*\mu_{T^{k+1}(X)}^x=\mu_{C_G^{k+1}(X)}$ for all $x\in X$. Notably $\mu_{C_G^{k+1}(X)}$ does not depend on $x\in X$. Thus we will be able to perform the following $\|\cdot\|_{\infty}$ estimate:
\[
\begin{array}{l}
\|q'(x)-q(x)\|_{\infty}=\|\int_{\pi_{T}^{-1}(x)}\left(F_{1}(g'_{v}(x)(\mathbf{t}))-F_{3}(\psi_{v}(\mathbf{t})\right)F_{2}(\mathbf{t})d\mu_{T^{k+1}(X)}^{x}(\mathbf{t})\|_{\infty}\\
\leq\|F_{2}\|_{\infty}\|\int_{\pi_{T}^{-1}(x)}\left(F_{1}(\rho(\psi_{v}(\mathbf{t})))-F_{3}(\psi_{v}(\mathbf{t})\right)d\mu_{T^{k+1}(X)}^{x}(\mathbf{t})\|_{\infty}\\
=\|F_{2}\|_{\infty}\|F_{1}\circ\rho-F_{3}\|_{L_{1}(\mu_{C_G^{k+1}(X)})}<\epsilon.
\end{array}
\]
Thus $q$ is a uniform limit of continuous functions, so $q$ is continuous. We conclude that the function $g':\ x\rightarrow\sum_{v\in\{0,1\}_{*}^{k+1}}(-1)^{|v|}g'_{v}(x)$
is continuous. By Equation \eqref{def:X'}, for $x\in X'$, for $\mu_{T^{k+1}(X)}^{x}$-a.e.
$\mathbf{t}\in T^{k+1}(X)$,
\begin{equation}
g'(x)(\mathbf{t})=\rho_{x}(\omega(\mathbf{t}))-\rho_{x}(\psi_{\vec{0}}(\mathbf{t})).\label{eq:D continuous -1}
\end{equation}
Recall the continuous function $\tilde{\pi}:T^{k+1}(X)\rightarrow C_G^{k+1}(X)\times_{X}C_G^{k+1}(X)$, given by $\tilde{\pi}(\mathbf{t})=(\psi_{\vec{0}}(\mathbf{t}),\omega(\mathbf{t}))$. Now apply Lemma \ref{lem:cov is cont} equating $X,Y,Z,\pi,\phi,\psi$ of this lemma in the following way: $X=T^{k+1}(X),Y=C_G^{k+1}(X)\times_{X}C_G^{k+1}(X),Z=X,$ $\pi=\tilde{\pi},\phi=\pi_T,\psi=\pi'$. By Proposition \ref{prop:CMD for tri cube}, the conditions required by Lemma \ref{lem:cov is cont} on maps and  measure disintegrations are satisfied. Thus we conclude that
\[
\begin{array}{l}
K=\mathcal{{L}}_{\tilde{\pi}}(T^{k+1}(X)\xrightarrow{\pi_{T}}X,A)
=\{[f\circ\tilde{\pi}]|\,f\in\mathcal{{L}}(C_G^{k+1}(X)\times_{X}C_G^{k+1}(X)\xrightarrow{\pi'}X,A)\}
\end{array}
\]
is closed and
\[
i:\mathcal{{L}}_{\tilde{\pi}}(T^{k+1}(X)),X,A)\rightarrow\mathcal{{L}}(C_G^{k+1}(X)\times_{X}C_G^{k+1}(X)\xrightarrow{\pi'}X,A)
\]
given by $i([f\circ\tilde{\pi}])=f$ is well defined and continuous. Notice that for $x\in X'$, $[g'(x)]=[g(x)\circ\tilde{\pi}]$. From Equation
\eqref{eq:D continuous -1}, for all $x\in X'$, $g'(x)\in K$. From
Lemma \ref{lem:continuous full support into closed}, $g'(x)\in K$
for all $x\in X$.

Define $D:X\rightarrow\mathcal{{L}}(C_G^{k+1}(X)\times_{X}C_G^{k+1}(X)\xrightarrow{\pi'}X,A)$
by $D=i\circ g'$. Notice that $D$ is continuous and $i\circ g'=g$
for $x\in X'$.
\end{proof}
\begin{lem}
\label{rem: convergent subsequence}  For any sequence $h_{n}\in\mathcal{L}(C_G^{k+1}(X)\xrightarrow{\mathfrak{p}_{0}}X,A)$
such that $\mathcal{E}(h_{n})$ converges, there exists a convergent
subsequence $\{h_{m_{n}}\}_{n\in\mathbb{N}}$.
\end{lem}
\begin{proof}
This is proven as part of the proof of \cite[Proposition 2.3.13]{candela2016cpt_notes}.
\end{proof}
Recall that $M'=\bigcup_{x\in X'}(-\rho_{x}+A)$. Now we prove the
main theorem in this subsection:
\begin{thm}
\label{thm:M-is-compact.} $M=\overline{M'}$ is compact.
\end{thm}

\begin{proof}
From Lemma \ref{lem:The-map-continuous}, as $X$ is compact,
\[
D(X)\subset\{[f\circ\tilde{\pi}]|\,f\in\mathcal{{L}}(C_G^{k+1}(X)\times_{X}C_G^{k+1}(X),X,A)\}
\]
is compact. From Lemma \ref{lem:The-map-continuous}, there exists a full measure
subset $X'$ of $X$ such that $D(x)=\rho_{x}(\cdot)-\rho_{x}(\cdot\cdot)$
for $x\in X'$. As $\mathcal{E}(M')=-g(X')=-D(X')\text{ and }\overline{D(X')}\subset D(X)$
is compact, it holds that $\overline{\mathcal{E}(M')}$ is compact. For any $h\in M$, choose $\{h_{n}\}_{n\in\mathbb{N}}\subset M'$
such that ${\displaystyle \lim_{n\rightarrow\infty}h_{n}}=h$. From
Lemma \ref{lem:E-is-continuous.}, ${\displaystyle \lim_{n\rightarrow\infty}\mathcal{E}(h_{n}})=\mathcal{E}(h)$.
Thus $\mathcal{E}(M)\subset\overline{\mathcal{E}(M')}$. For any $\{g_{n}\}_{n\in\mathbb{N}}\subset M$, as $\overline{\mathcal{E}(M')}$
is compact and $\{\mathcal{E}(g_{n})\}_{n\in\mathbb{N}}\subset\mathcal{E}(M)\subset\overline{\mathcal{E}(M')}$,
there exists a subsequence $\{g_{m_{1,n}}\}_{n\in\mathbb{N}}$ of
$\{g_{n}\}_{n\in\mathbb{N}}$ such that $\mathcal{E}(g_{m_{1,n}})$
is convergent. From Lemma \ref{rem: convergent subsequence}, there
exists a subsequence $\{g_{m_{2,n}}\}_{n\in\mathbb{N}}$ of $\{g_{m_{1,n}}\}_{n\in\mathbb{N}}$
such that $\{g_{m_{2,n}}\}_{n\in\mathbb{N}}$ is convergent. As any
sequence in $M$ has a convergent subsequence, one has that $M=\overline{M'}$
is compact.
\end{proof}

\subsection{The extension $M\rightarrow X$ is a topological model for Theorem
\ref{thm:main theorem}.}\label{subsec:model}

\subsubsection{The action of $G$ is well defined and continuous.}\label{sec: continuous}
Recall from Subsection \ref{subsec:Dynamical-background.} that to the abelian group extension $Y=(X\times A,\mu\times m_{\haar(A)},G)\rightarrow (X,\mu,G)$ in Theorem \ref{thm:main theorem}, is associated a cocycle which we will denote by $\beta:G\times X\rightarrow A$. We define a Borel action $G\times M'\rightarrow M'$ of $G$ on $M'=\hat{\mathfrak{p}}_0^{-1}(X')$  by
\begin{equation}\label{definition of G on M'}
 g(-\rho_{x}+a)=-\rho_{gx}+\beta(g,x)+a.
\end{equation}
An important intermediate goal is to show this action extends to a continuous action $(M,G)$.
\begin{lem}
\label{fact:well defined and continuous} For any $g\in G$, $x_n\in X'$, $\{g(-\rho_{x_{n}}+a_{n})\}_{n\in\mathbb{N}}$
converges if $\{-\rho_{x_{n}}+a_{n}\}_{n\in\mathbb{N}}$
converges.
\end{lem}

\begin{proof}
For any continuous function $F:C_G^{k+1}(X)\rightarrow\mathbb{C}$,
$\|F\|_{\infty}\neq0$, character $\chi:A\rightarrow\mathbb{C}$, $g\in G$, $x\in X'$ and $a\in A$, consider:
$$\phi_{F,\chi,g,x,a}\stackrel{\triangle}{=}\int_{\mathfrak{p}_{0}^{-1}(gx)}F(c)\chi(-\rho_{gx}(c)+\beta(g,x)+a)d\mu_{C_G^{k+1}(X)}^{gx}(c).$$
Assume that $\{-\rho_{x_{n}}+a_{n}\}_{n\in\mathbb{N}}$ converges.  We will show that $\phi_{F,\chi,g,x_n,a_n}$ is a Cauchy sequence and thus is a converging sequence. We start by rewriting $\phi=\phi_{F,\chi,g,x,a}$ using the change of variable $c=g^{[k+1]}c'$ (which implies $x=c'_{\vec{0}}$) and Lemma \ref{measure-preserving}, defining $F'=F\circ g^{[k+1]}$ and using the abbreviation $\rho=\rho_x$,
$$\phi=\int_{\mathfrak{p}_{0}^{-1}(x)}F'(c')\chi(-\rho(g^{[k+1]}c')+\beta(g,c'_{\vec{0}})+a)d\mu_{C_G^{k+1}(X)}^{x}(c').$$
Using Equation
\eqref{eq:the action on M} on Page \pageref{eq:the action on M}, we have:
\begin{equation*}\begin{array}{ll}\phi=\\ \int_{\mathfrak{p}_{0}^{-1}(x)}F'(c)\chi(-\rho(c)+a{\displaystyle-\sum_{v\in\{0,1\}_{*}^{k+1}}(-1)^{|v|}\beta(g,c_{v}))}d\mu_{C_G^{k+1}(X)}^{x}.
\end{array}
\end{equation*}
By  Lemma \ref{measure-preserving}, we may use the change of variable $c=\omega(\bf{t})$ (which implies by Subsection \ref{subsec:The-tri-cube(3-cube)}, $c_{v}=\omega({\bf{t}})_{v}={\bf{t}}_{\Omega(v)}={\bf{t}}_{\psi_v(\vec{0})}=\psi_v({\bf{t}} )_{\vec{0}}$) to rewrite
$$\phi=\int_{\pi_{T}^{-1}(x)}F'(\omega({\bf{t}}))\chi(-\rho(\omega({\bf{t}}))+a{\displaystyle -\sum_{v\in\{0,1\}_{*}^{k+1}}(-1)^{|v|}\beta(g,\psi_v({\bf{t}} )_{\vec{0}}))}d\mu_{T^{k+1}(X)}^{x}({\bf{t}}).$$
As $x\in X'$, by Equation \eqref{def:X'} this in turn equals
\begin{equation*}
\begin{array}{l}
\phi=\int_{\pi_{T}^{-1}(x)}F'(\omega({\bf{t}}))\chi\big(-{\displaystyle \sum_{\nu\in\{0,1\}^{k+1}}(}-1)^{|\nu|}\rho(\psi_{\nu}(\mathbf{t}))+a\\
\ \ \ \ \ \ \ \ -{\displaystyle \sum_{v\in\{0,1\}_{*}^{k+1}}(}-1)^{|v|}\beta(g,\psi_v({\bf{t}} )_{\vec{0}})\big)d\mu_{T^{k+1}(X)}^{x}({\bf{t}}).
\end{array}\label{eq:continuous 1-2}
\end{equation*}
By the multiplicativity of characters, we have
\begin{equation}
\begin{array}{l}
\phi= \\
\int_{\pi_{T}^{-1}(x)}F'(\omega({\bf{t}})\chi(-\rho(\psi_{\vec{0}}(\mathbf{t}))+a)
{\displaystyle \prod_{\nu\in\{0,1\}_{*}^{k+1}}}\chi\left((-1)^{|\nu|+1}(\rho(\psi_{\nu}(\mathbf{t}))+\beta(g,\psi_{\nu}(\mathbf{t})_{\vec{0}})\right)d\mu_{T^{k+1}(X)}^{x}({\bf{t}}).
\end{array}\label{eq:continuous 1-2'}
\end{equation}
Let us write this expression as $\int_{\pi_{T}^{-1}(x)}A_{\vec{0}}({\bf{t}})\prod_{v\in\{0,1\}_{*}^{k+1}}A_{v}({\bf{t}})d\mu_{T^{k+1}(X)}^{x}({\bf{t}})$. Note $\|A_v\|_{\infty}\leq 1$ for $v\in\{0,1\}_{*}^{k+1}$. Let $A'_{v}:T^{k+1}\rightarrow \mathbb{C}$, be functions such that $\|A'_v\|_{\infty}\leq 1$. By \cite[Lemma 3.3]{GWSY2018}  
\begin{equation}
\begin{array}{l}{\displaystyle
|\int_{\pi_{T}^{-1}(x)}A_{\vec{0}}\big(\prod_{v\in\{0,1\}_{*}^{k+1}}A_{v}-\prod_{v\in\{0,1\}_{*}^{k+1}}A'_{v}\big)|\leq \|A_{\vec{0}}\|_{\infty}\sum_{v\in\{0,1\}_{*}^{k+1}}\|A_{v}-A'_{v}\|_{1}.
}
\end{array}\label{eq:AA'}
\end{equation}
We can thus approximate $\phi$ in the following way. Fix $\epsilon>0$. By Lusin's theorem \cite[Theorem 2.24]{rudin2006real}, there exist continuous functions $H_{v,\epsilon}:C_G^{k+1}(X)\rightarrow\mathbb{C}$ with $\|H_{v,\epsilon}\|_{\infty}\leq 1$
such that
\[
\|H_{v,\epsilon}(c)-\chi\left((-1)^{|\nu|+1}(\rho(c)+\beta(g,c_v))\right)\|_{L^{1}(\mu_{C_G^{k+1}(X)})}<\frac{\epsilon}{2^{k+1}\|F'\|_{\infty}}.
\]
By Lemma \ref{measure-preserving}, $(\psi_{\nu})_{*}\mu_{T^{k+1}(X)}^{x}=\mu_{C_G^{k+1}(X)}$, and therefore:
\begin{equation}
\begin{array}{l}
\|H_{v,\epsilon}(\psi_{\nu}(\mathbf{t}))-\chi\left((-1)^{|\nu|+1}(\rho(\psi_{\nu}(\mathbf{t}))+\beta(g,\psi_{\nu}(\mathbf{t})_{\vec{0}})\right)\|_{L^{1}(\mu_{T^{k+1}(X)}^{x})}<\frac{\epsilon}{2^{k+1}\|F'\|_{\infty}}.
\end{array}\label{eq:H}
\end{equation}
Denote:
$$\phi_{\epsilon,x}=\int_{T^{k+1}(X)}F'(\omega(\mathbf{t})){\displaystyle \chi(\rho(-\psi_{\vec{0}}(\mathbf{t}))+a))\prod_{\nu\in\{0,1\}_{*}^{k+1}}}H_{v,\epsilon}(\psi_{\nu}(\mathbf{t}))d\mu_{T^{k+1}(X)}^{x}(\mathbf{t}).$$
Therefore by Equations \eqref{eq:continuous 1-2'},\eqref{eq:AA'} and \eqref{eq:H} for all $x\in X'$,
\begin{equation}
\begin{array}{l}
|\phi-\phi_{\epsilon,x}|<\epsilon.
\end{array}\label{eq:continuous 1-4}
\end{equation}
We may now complete the proof. Recall that by assumption $\{-\rho_{x_{n}}+a_{n}\}_{n\in\mathbb{N}}$ converges in $\mathcal{L}(C_G^{k+1}(X)\xrightarrow{\mathfrak{p}_{0}}X,\ A)$. By Lemma \ref{measure-preserving}, for any $x\in X$, $(\psi_{\vec{0}})_{*}\mu_{T^{k+1}(X)}^{x}=\mu_{C_G^{k+1}(X)}^{x}$, so it follows that $\{-\rho_{x_{n}}(\psi_{\vec{0}}(\mathbf{t}))+a_{n}\}_{n\in\mathbb{N}}$ converges
in $\mathcal{L}(T^{k+1}(X)\stackrel{\pi_{T}}{\rightarrow}X,A)$. As $F'(\omega(\mathbf{t})){\displaystyle \prod_{\nu\in\{0,1\}_{*}^{k+1}}}H_{v}(\psi_{\nu}(\mathbf{t}))$ is a continuous function on $T^{k+1}(X)$, $\phi_{\epsilon,x_n}$ converges. Denote the limit by $C_{\epsilon}$. Denote $\phi_n=\phi_{F,\chi,g,x_n,a}$. By Equation  \eqref{eq:continuous 1-4}, $|\phi_n-\phi_{\epsilon,x_n}|<\epsilon$ for all $n$. Thus for any $\epsilon$, there exists $N_{\epsilon}$ such
that if $n>N_{\epsilon}$, $|\phi_n-C_{\epsilon}|<2\epsilon$. Thus $\phi_n$ is a Cauchy sequence as desired.
\end{proof}
Suppose $f\in M$ and fix $g\in G$. Assume $-\rho_{x_n}+a_n\rightarrow f$ for $-\rho_{x_n}+a_n\in M'$ and define
\begin{equation}\label{definition of G action1}
   gf:= \lim_{n\rightarrow\infty}g(-\rho_{x_{n}}+a_{n})\,\,(f\in M, \, g\in G).
\end{equation}
This definition does not depend on the choice of sequence. Indeed let $\{-\rho_{x_{n}}+a_{n}\}_{n\in\mathbb{N}},\{-\rho_{x_{n}'}+a_{n}'\}_{n\in\mathbb{N}}$
be two convergent sequences such that

\[
f=\lim_{n\rightarrow\infty}-\rho_{x_{n}}+a_{n}=\lim_{n\rightarrow\infty}-\rho_{x_{n}'}+a_{n}'.
\]
By combining the two sequences into one converging sequence, Lemma
\ref{fact:well defined and continuous} shows that the two limits $
\lim_{n\rightarrow\infty}g(-\rho_{x_{n}}+a_{n}),\ \lim_{n\rightarrow\infty}g(-\rho_{x_{n}'}+a_{n}')
$
exist and are equal. Thus  Equation \eqref{definition of G action1} gives rise to a paring $G\times M\rightarrow M$.

\begin{lem}\label{lem:well define and continuous}
The above pairing $G\times M\rightarrow M$ defines a continuous action $(M,G)$.
\end{lem}
\begin{proof}
The fact that the paring $G\times M\rightarrow M$ satisfies $\id x=x$ and $g(hx)=(gh)x$ for all $x\in M$ and $g,h\in G$ is straight-forward. 
Now we prove the continuity of $G\times M\rightarrow M$. Recall that $G$ is countable and equipped with the discrete topology. Thus it is enough to prove that each $g\in G$ defines a  continuous function $g:M\rightarrow M$. Suppose $f_{m}\rightarrow f$
in $M$ and $\underset{n\rightarrow\infty}{\lim}-\rho_{x_{m,n}}+a_{m,n}=f_{m}$
for $m\in\mathbb{N}$. Let $d$ be a compatible metric for $M$. There exists $\{N_{m}\}_{m\in\mathbb{N}}$
such that
\begin{equation}
d(-\rho_{x_{m,n}}+a_{m,n},f_{m})<\frac{1}{2^{m}},\ d(g(-\rho_{x_{m,n}}+a_{m,n}),gf_{m})<\frac{1}{2^{m}}\label{eq:limit of t}
\end{equation}
for any $n\geq N_{m}$ and $m\in\mathbb{N}$. Let $N_{\epsilon}$
be an integer such that $d(f_{m},f)<\epsilon$ for any $m>N_{\epsilon}$.
As for any $m,m'>N_{\epsilon}$,
\[
\begin{array}{l}
d(-\rho_{x_{m,N_{m}}}+a_{m,N_{m}},-\rho_{x_{m',N_{m'}}}+a_{m',N_{m'}})\\
\leq d(-\rho_{x_{m,N_{m}}}+a_{m,N_{m}},f_{m})+d(-\rho_{x_{m',N_{m'}}}+a_{m',N_{m'}},f_{m'})+d(f_{m},f_{m'})\\
<2\epsilon+\frac{1}{2^{m}}+\frac{1}{2^{m'}},
\end{array}
\]
thus $\{-\rho_{x_{m,N_{m}}}+a_{m,N_{m}}\}_{m\in\mathbb{N}}$ is a Cauchy
sequence. By Equation \eqref{eq:limit of t}, one has that
$\underset{m\rightarrow\infty}{\lim}-\rho_{x_{m,N_{m}}}+a_{m,N_{m}}=f$.
By the definition of $gf$, $\underset{m\rightarrow\infty}{\lim}g(-\rho_{x_{m,N_{m}}}+a_{m,N_{m}})=gf$. Equation \eqref{eq:limit of t} shows that $\underset{m\rightarrow\infty}{\lim}d(g(-\rho_{x_{m,N_{m}}}+a_{m,N_{m}}),gf_{m})=0$.
This implies that
\[
\lim_{m\rightarrow\infty}gf_{m}=gf,
\]
as desired.
\end{proof}
\subsubsection{The extension $M\rightarrow X$ is a topological group extension.}
\begin{lem}
\label{lem:subsequence converge to f} Let $f,g\in M$ such that $\hat{\mathfrak{p}}_{0}(f)=\hat{\mathfrak{p}}_{0}(g)=x$.
Then there exists $a\in A$ such that $g=f+a.$
\end{lem}

\begin{proof}

Let $\{-\rho_{x_{n}}+a_{n}\}_{n\in\mathbb{N}},\{-\rho_{x_{n}'}+a_{n}'\}_{n\in\mathbb{N}}$ be two sequences in $M'$
such that $
\lim_{n\rightarrow\infty}-\rho_{x_{n}}+a_{n}=f$ and $\lim_{n\rightarrow\infty}-\rho_{x_{n}'}+a_{n}'=g.
$
From Lemma \ref{lem:The-map-continuous},
$
\lim_{n\rightarrow\infty}\rho_{x_{n}}(\cdot)-\rho_{x_{n}}(\cdot\cdot)=\lim_{n\rightarrow\infty}\rho_{x_{n}'}(\cdot)-\rho_{x_{n}'}(\cdot\cdot).
$
From Lemma \ref{lem:E-is-continuous.},
\[
\lim_{n\rightarrow\infty}(-\rho_{x_{n}}(\cdot)+a_{n})-(-\rho_{x_{n}}(\cdot\cdot)+a_{n})=\lim_{n\rightarrow\infty}(-\rho_{x_{n}'}(\cdot)+a_{n}')-(-\rho_{x_{n}'}(\cdot\cdot)+a_{n}'),
\]
therefore $f(\cdot)-f(\cdot\cdot)=g(\cdot)-g(\cdot\cdot)$ in $\mathcal{{L}}(C_G^{k+1}(X)\times_{X}C_G^{k+1}(X)\xrightarrow{\pi'}X,A)$.
That is, for $\mu_{C_G^{k+1}(X)}^{x}\times\mu_{C_G^{k+1}(X)}^{x}$-a.e.
$(c_{1},c_{2})$, $f(c_{1})-f(c_{2})=g(c_{1})-g(c_{2})$. By the definition
of function bundles,  there exists $c_{2}$ such that for $\mu_{C_G^{k+1}(X)}^{x}$-a.e.
$c_{1}$,
\[
g(c_{1})-f(c_{1})=g(c_{2})-f(c_{2}).
\]
Let $a=f(c_{2})-g(c_{2})$, therefore $g=f+a$.
\end{proof}
\begin{prop}\label{prop:M->X is group ext}
The extension $\hat{\mathfrak{p}}_0: (M,G)\rightarrow (X,G)$ is a topological group extension by the group $A$.
\end{prop}
\begin{proof}
Recall the definition of topological group extension in Subsection \ref{subsec:Dynamical-background.}. By the definition
of function bundle, $\hat{\mathfrak{p}}_{0}:M\rightarrow X$ is continuous
(see Section \ref{subsec:Function-bundles}) and thus it is easy to see that $\hat{\mathfrak{p}}_{0}:(M,G)\rightarrow (X,G)$ is a factor map. By
Lemma \ref{lem:subsequence converge to f} for all $x\in X$, $\hat{\mathfrak{p}}_{0}^{-1}(x)=f+A$ for some $f\in M$. To see that the action of $A$ on $M:$ $A\times M\rightarrow M:\ (a,f)\rightarrow f+a$
is continuous and free, recall that $f_{n}\rightarrow f$ in $M$ if:
\[
\begin{array}{l}
\underset{n\rightarrow\infty}{\lim}\int_{\hat{\mathfrak{p}}_{0}^{-1}(\hat{\mathfrak{p}}_{0}(f_{n}))}F_{1}(f_{n}(v))F_{2}(v)d\mu_{C_G^{k+1}(X)}^{\hat{\mathfrak{p}}_{0}(f_{n})}(v)\\
=\int_{\hat{\mathfrak{p}}_{0}^{-1}(\hat{\mathfrak{p}}_{0}(f))}F_{1}(f(v))F_{2}(v)d\mu_{C_G^{k+1}(X)}^{\hat{\mathfrak{p}}_{0}(f)}(v),
\end{array}
\]
for every character $F_{1}:A\rightarrow S^{1}$ and continuous function
$F_{2}:X\rightarrow\mathbb{C}$. As $F_{1}$ is a character, it holds
that
\[
\begin{array}{l}
\int_{\mathfrak{p}_{0}^{-1}(\hat{\mathfrak{p}}_{0}(f))}F_{1}(f(v)+a)F_{2}(v)d\mu_{C_G^{k+1}(X)}^{\hat{\mathfrak{p}}_{0}(f)}(v)\\
=F_{1}(a)\int_{\mathfrak{p}_{0}^{-1}(\hat{\mathfrak{p}}_{0}(f))}F_{1}(f(v))F_{2}(v)d\mu_{C_G^{k+1}(X)}^{\hat{\mathfrak{p}}_{0}(f)}(v),
\end{array}
\]
which gives the continuity of the map: $(a,f)\rightarrow f+a$. The freeness of the action is trivial. Finally by the definition of the action of $G$ on $M'$ (Equation \eqref{definition of G on M'}), for any $m\in M'$, $gam=agm$ for any $a\in A$, $g\in G$. As $M=\overline{M'}$, one has that the  actions $A$ and $G$ commute.
\end{proof}
\subsubsection{The extension $M\rightarrow X$ is a strictly ergodic model.}

\begin{lem}\label{abelian group extension and uniquely ergodic}
Assume that $\pi:(Y,G)\rightarrow(X,G)$ is a topological group extension by a compact abelian group $A$, such that $(X,G)$ is uniquely ergodic with invariant measure $\mu$. Let $\nu$ be the associated measure on $Y$ w.r.t.\  the Haar measure of $A$, i.e.,  $\nu=\int_X\delta_x \times m_{\haar(A)} d\mu(x)$. If $(Y,\nu, G)$ is ergodic, then $(Y,G)$ is uniquely ergodic.
\end{lem}
\begin{proof}

Let $\nu'$ be an ergodic measure of $(Y,G)$. One has that $a_*\nu'$ is ergodic for any $a\in A$. Let $\hat{\nu}=\int_Aa_*\nu'dm_{\haar(A)}(a)$. By considering the measure disintegration of $\hat{\nu}$ above $X$, it is easy to see that $\hat{\nu}$ is the associated measure on $Y$ w.r.t.\  the Haar measure of $A$, i.e.\  $\hat{\nu}=\nu$. As $\nu$ is ergodic, it holds that $\nu'=\nu$, i.e.\  $(Y,G)$ is uniquely ergodic.

\end{proof}

\begin{thm}\label{strictly ergodic model}
The extension $(M,G)\xrightarrow{\hat{\mathfrak{p}}_0}(X,G)$ is a strictly ergodic distal model for $(Y,\mu\times m_{\haar(A)},G)\xrightarrow{\pi}(X,\mu,G)$.
\end{thm}

\begin{proof}
 Note that $X'\times A$ is a $G$-invariant subset of $Y$ of full $\mu\times m_{\haar(A)}$ measure. The set $M'$ is a copy of $X'\times A$ in a natural way, $M'\leftrightarrow X'\times A:-\rho_x+a\leftrightarrow (x,a)$. Moreover the action of $G$ on $M'$ defined by \eqref{definition of G on M'} agrees with the action of $G$ on $X'\times A\subset Y$, and the action of $A$ on $X\times A$ agrees with the action of $A$ on $M'$. Thus the measure $\mu\times m_{\haar(A)}$ is well defined on $M$ and is $G$-invariant and moreover $(M,G)\xrightarrow{\hat{\mathfrak{p}}_0}(X,G)$ is a topological model for $(Y,\mu\times m_{\haar(A)},G)\xrightarrow{\pi}(X,\mu,G)$. By assumption $\mu\times m_{\haar(A)}$ is ergodic. Thus by Lemma \ref{abelian group extension and uniquely ergodic}, $(M,G)$ is uniquely ergodic. By Proposition \ref{prop:M->X is group ext}, $\hat{\mathfrak{p}}_0:(M,G)\rightarrow (X,G)$ is a topological group extension. As a group extension of a distal system is distal (\cite[Chapter V, Proposition 4.5]{dV93}), $(M,G)$ is distal. By Remark \ref{rem:strictly ergodic distal}, a distal uniquely ergodic system is strictly ergodic.
\end{proof}

\subsubsection{The extension $M\rightarrow X$ is a fibration of order at most $k$.}\label{subsec:extension of order k}

\begin{prop}
\label{thm:k+1 uniqueness} Let $\mathbf{f},\mathbf{g}\in C_G^{k+1}(M)$
such that $\mathbf{f}_{v}=\mathbf{g}_{v}$ for $v\in\{0,1\}_{*}^{k+1}$
and $\hat{\mathfrak{p}}_{0}(\mathbf{f}_{w})=\hat{\mathfrak{p}}_{0}(\mathbf{g}_{w})$
for $w\in\{0,1\}^{k+1}$. Then $\mathbf{f}_{0}=\mathbf{g}_{0}$.
\end{prop}
Before proving Proposition \ref{thm:k+1 uniqueness}, we need some
preparation. From Lemma \ref{lem: tricube measure nilcycle} for any $x\in X'$, $\mu_{C_G^{k+1}(X)}^{x}$-a.e. $c\in C_G^{k+1}(X)$,
\begin{equation}
\int_{\omega^{-1}(c)}d_{A}\left(\sum_{\nu\in\{0,1\}^{k+1}}(-1)^{|\nu|}\rho(\psi_{\nu}(\mathbf{t})),\rho(\omega(\mathbf{t}))\right)d\eta^{c}(\mathbf{t})=0.\label{eq: a.e. Ck+1-1}
\end{equation}
Notice that in Equation \eqref{eq: a.e. Ck+1-1}, $\omega(\mathbf{t})=c$
for $\mathbf{t}\in \omega^{-1}(c)$. There exists a full $\mu_{C_G^{k+1}(X)}$-measure
set $V\subset C_G^{k+1}(X)$, such that for $c\in V$, Equation (\ref{eq: a.e. Ck+1-1})
holds. Also notice that $\mathfrak{p}_v:C_G^{k+1}(X)\rightarrow X$ are measure-preserving for all $v\in \{0,1\}^{k+1}$. Therefore by replacing $V$ by $V\cap\bigcap_{v\in \{0,1\}^{k+1}}(\mathfrak{p}_v)^{-1}(X')$ we may assume that for every $c\in V$ and $v\in \{0,1\}^{k+1}$, it holds that $c_v\in X'$.

\begin{defn}\label{definition of Q(M)}

From Lemma \ref{lem:the action on M}, there exists $U\subset V$  such that $\mu^{[k+1]}(U)=1$ and for any $c\in U$, $\rho(\mathbf{g}c)=\rho(c)+\sum_{v\in\{0,1\}^{k+1}}(-1)^{|v|}\beta(\mathbf{g}_{v},c_{v})$ for any $\mathbf{g}\in \mathcal{HK}^{k+1}(G)$. In particular, we can assume that $U$ is $\mathcal{HK}^{k+1}(G)$-invariant and $\sigma$-invariant for any cube isomorphism $\sigma$. Define 
\begin{equation}\label{definiton of Q(MU)}
\begin{array}{ll}
 Q^{k+1}_U(M)=  & \{\mathbf{f}:\{0,1\}^{k+1}\rightarrow M^{[k+1]}:\nu\mapsto-\rho_{c_{\nu}}+\mathbf{a}_{\nu}|\\\, & c\in U
\text{ and }\sum_{\nu\in\{0,1\}^{k+1}}(-1)^{|\nu|}\mathbf{a}_{\nu}=\rho(c)\}.
\end{array}
\end{equation}
\end{defn}

\begin{lem}\label{uniqueness for Q} Let $\mathbf{f},\mathbf{g}\in \overline{Q^{k+1}_U(M)}$
such that $\mathbf{f}_{v}=\mathbf{g}_{v}$ for $v\in\{0,1\}_{*}^{k+1}$
and $\hat{\mathfrak{p}}_{0}(\mathbf{f}_{\vec{0}})=\hat{\mathfrak{p}}_{0}(\mathbf{g}_{\vec{0}})$. Then $\mathbf{f}_{\vec{0}}=\mathbf{g}_{\vec{0}}$.
\end{lem}
\begin{proof}
We suppose
\[
\lim_{n\rightarrow\infty}(-\rho_{b(n,\mathbf{f})_{v}}^{(n)}+a_{v,\mathbf{f}}^{(n)})_{v\in\{0,1\}^{k+1}}=\mathbf{f},\ \lim_{n\rightarrow\infty}(-\rho_{b(n,\mathbf{g})_{v}}^{(n)}+a_{v,\mathbf{g}}^{(n)})_{v\in\{0,1\}^{k+1}}=\mathbf{g},
\]
such that $(-\rho_{b(n,\mathbf{f})_{v}}^{(n)}+a_{v,\mathbf{f}}^{(n)})_{v\in\{0,1\}^{k+1}},(-\rho_{b(n,\mathbf{g})_{v}}^{(n)}+a_{v,\mathbf{g}}^{(n)})_{v\in\{0,1\}^{k+1}}\in  Q^{k+1}_U(M)$. Therefore $\lim_{n\rightarrow \infty} b(n,\mathbf{f})=\lim_{n\rightarrow \infty} b(n,\mathbf{g})\in C_G^{k+1}(X)$.
By Definition \ref{definition of Q(M)}, one has that
\begin{equation}
\begin{array}{l}
\rho(b(n,\mathbf{f}))=\sum_{\nu\in\{0,1\}^{k+1}}(-1)^{|\nu|}a_{v,\mathbf{f}}^{(n)};\\
\text{ }\rho(b(n,\mathbf{g}))=\sum_{\nu\in\{0,1\}^{k+1}}(-1)^{|\nu|}a_{v,\mathbf{g}}^{(n)}.
\end{array}\label{eq:uniqueness-2}
\end{equation}
To prove the lemma, we need
to show that
\begin{equation}
\mathbf{f}_{\vec{0}}=\lim_{n\rightarrow\infty}-\rho_{b(n,\mathbf{f})_{\vec{0}}}^{(n)}+a_{\vec{0},\mathbf{f}}^{(n)}=\lim_{n\rightarrow\infty}-\rho_{b(n,\mathbf{g})_{\vec{0}}}^{(n)}+a_{\vec{0},\mathbf{g}}^{(n)}=\mathbf{g}_{\vec{0}}.\label{eq:uniqueness}
\end{equation}
By the construction of $U$, as $U\subset V$, Equation \eqref{eq: a.e. Ck+1-1} holds for every $b\in U$:

\[
\rho(b)=\sum_{\nu\in\{0,1\}^{k+1}}(-1)^{|\nu|}\rho(\psi_{\nu}(\mathbf{t}))\text{ for }\eta^{b}\text{-a.e. }\mathbf{t}.
\]
By Equation \eqref{eq:uniqueness-2} , for $\bullet=\mathbf{f}$ or
$\mathbf{g}$,
\begin{equation}
\sum_{\nu\in\{0,1\}^{k+1}}(-1)^{|\nu|}\rho(\psi_{\nu}(\mathbf{t}))=\sum_{\nu\in\{0,1\}^{k+1}}(-1)^{|\nu|}a_{v,\bullet}^{(n)}\text{ for }\eta^{b(v,n,\bullet)}(\mathbf{t})\text{-a.e. }\mathbf{t}.\label{eq:the completion equation}
\end{equation}
By assumption,
\[
\lim_{n\rightarrow\infty}-\rho_{b(n,\mathbf{f})_{v}}^{(n)}+a_{v,\mathbf{f}}^{(n)}=\mathbf{f}_{v}=\mathbf{g}_{v}=\lim_{n\rightarrow\infty}-\rho_{b(n,\mathbf{g})_{v}}^{(n)}+a_{v,\mathbf{g}}^{(n)},
\]
for $v\in\{0,1\}_{*}^{k+1}$. Thus for any character $\chi$ and any
continuous function $F$,
\[
\begin{array}{l}
\ \ \underset{n\rightarrow\infty}{\lim}(\int_{\mathfrak{p}_{0}^{-1}(b(n,\mathbf{f})_{v})}\chi(-\rho_{b(n,\mathbf{f})_{v}}(c)+a_{v,\mathbf{f}}^{(n)})Fd\mu_{C_G^{k+1}(X)}^{b(n,\mathbf{f})_{v}}(c)\\
=\underset{n\rightarrow\infty}{\lim}\int_{\mathfrak{p}_{0}^{-1}(b(n,\mathbf{g})_{v})}\chi(-\rho_{b(n,\mathbf{g})_{v}}(c)+a_{v,\mathbf{g}}^{(n)})Fd\mu_{C_G^{k+1}(X)}^{b(n,\mathbf{g})_{v}}(c)).
\end{array}
\]
From Lemma \ref{measure-preserving}, $(\psi_{v})_{*}\eta^{b(n,\bullet)}=\mu_{C_G^{k+1}(X)}^{b(n,\bullet)_{v}}$, therefore for $v\in\{0,1\}_{*}^{k+1}$,
\begin{equation}
\begin{array}{l}
\ \ \underset{n\rightarrow\infty}{\lim}\int_{\omega^{-1}(b(n,\mathbf{f}))}\chi(-\rho(\psi_{v}(\mathbf{t}))+a_{v,\mathbf{f}}^{(n)})F\circ\psi_vd\eta^{b(n,\mathbf{f})}(\mathbf{t})\\
=\underset{n\rightarrow\infty}{\lim}\int_{\omega^{-1}(b(n,\mathbf{g}))}\chi(-\rho(\psi_{v}(\mathbf{t}))+a_{v,\mathbf{g}}^{(n)})F\circ\psi_v d\eta^{b(n,\mathbf{g})}(\mathbf{t}).
\end{array}\label{eq:uniqueness-3}
\end{equation}
We claim that for any character $\chi$ and continuous function $H:T^{k+1}(X)\rightarrow \mathbb{C}$, the following limits exist and are equal:
\begin{equation}
\begin{array}{ll}
 \lim_{n\rightarrow\infty}\int_{\omega^{-1}(b(n,\mathbf{f}))}\chi(-\rho(\psi_{v}(\mathbf{t}))+a_{v,\mathbf{f}}^{(n)})H(\mathbf{t})d\eta^{b(n,\mathbf{f})}(\mathbf{t})\\
  =\lim_{n\rightarrow\infty}\int_{\omega^{-1}(b(n,\mathbf{g}))}\chi(-\rho(\psi_{v}(\mathbf{t}))+a_{v,\mathbf{g}}^{(n)})H(\mathbf{t})d\eta^{b(n,\mathbf{g})}(\mathbf{t}).
 \end{array}
\label{eq:uniqueness-for tricube}
\end{equation}
To prove the claim, notice that by Proposition \ref{prop:CMD for tri cube}, $\tilde{\pi}_v=(\omega,\psi_v):T^{k+1}(X)\rightarrow C_G^{k+1}(X)\times_{X}^vC_G^{k+1}(X)$, $\pi_{L,v}:C_G^{k+1}(X)\times_{X}^vC_G^{k+1}(X)\rightarrow C_G^{k+1}(X)$ and $\omega=\pi_{L,v}\circ \tilde{\pi}_v$ are CMD factor maps. Therefore there exist continuous systems of measures $\gamma_1^{(c_1,c_2)}$ w.r.t.\ $\tilde{\pi}_v$ and $\gamma_2^{c_1}$ w.r.t.\ $\pi_{L,v}$. By Lemma \ref{lem:CSM circ}, $\eta^{c_1}=\int \gamma_1^{(c_1,c_2)}d\gamma_2^{c_1}$. For $\bullet=\mathbf{f}$ or $\mathbf{g}$, one has that
\begin{equation}\label{H->F_H}\begin{array}{l}\int_{\omega^{-1}(b(n,\bullet))}\chi(-\rho(\psi_{v}(\mathbf{t}))+a_{v,\bullet}^{(n)})H(\mathbf{t})d\eta^{b(n,\bullet)}(\mathbf{t})\\ 
=\int_{\pi_L^{-1}(b(n,\bullet))} \int_{\tilde{\pi}^{-1}(b(n,\bullet),c)}\chi(-\rho(c)+a_{v,\bullet}^{(n)}) H(\mathbf{t}) d\gamma_1^{(b(n,\bullet),c)}(\mathbf{t})d\gamma_2^{b(n,\bullet)}(c)\\
=\int_{\pi_L^{-1}(b(n,\bullet))}\chi(-\rho(c)+a_{v,\bullet}^{(n)})\left(\int_{\tilde{\pi}^{-1}(b(n,\bullet),c)}H(\mathbf{t})d\gamma_1^{(b(n,\bullet),c)}(\mathbf{t})\right)d\gamma_2^{b(n,\bullet)}(c)\\
=\int_{\pi_L^{-1}(b(n,\bullet))}\chi(-\rho(c)+a_{v,\bullet}^{(n)}) F_H(b(n,\bullet),c)d\gamma_2^{b(n,\bullet)}(c)\\
\stackrel{c=\psi_{v}(\mathbf{t})}{=}\int_{\omega^{-1}(b(n,\bullet))}\chi(-\rho(\psi_{v}(\mathbf{t}))+a_{v,\bullet}^{(n)})F_H(b(n,\bullet),\psi_{v}(\mathbf{t}))d\eta^{b(n,\bullet)}(\mathbf{t}),
\end{array}\end{equation}
where $F_H(c',c)=\int_{\tilde{\pi}^{-1}(c',c)}H(\mathbf{t})d\gamma_1^{(c',c)}(\mathbf{t})$ is a continuous function for $C_G^{k+1}(X)\times_X^v C_G^{k+1}(X)\rightarrow \mathbb{C}$ as $\{\gamma_1^{(c',c)}\}$ is a continuous system of measures. For any continuous functions $F',F:C_G^{k+1}(X)\rightarrow \mathbb{C}$, one uses the crucial fact that $\lim_{n\rightarrow \infty}b(n,\mathbf{f})=\lim_{n\rightarrow \infty}b(n,\mathbf{g})$ to conclude $\lim_{n\rightarrow \infty}F'(b(n,\mathbf{f}))=\lim_{n\rightarrow \infty}F'(b(n,\mathbf{g}))$. Notice the fact that for any (bounded) convergent sequences $e_n,f_n\in \mathbb{C}$, $\lim_{n\rightarrow\infty}e_nf_n=(\lim_{n\rightarrow\infty}e_n)(\lim_{n\rightarrow\infty}f_n)$. Therefore, by Equation \eqref{eq:uniqueness-3} we have the following equality:
$$\begin{array}{ll}\underset{n\rightarrow\infty}{\lim}F'(b(n,\mathbf{f}))\int_{\omega^{-1}(b(n,\mathbf{f})))}\chi(-\rho(\psi_{v}(\mathbf{t}))+a_{v,\mathbf{f}}^{(n)})F(\psi_{v}(\mathbf{t}))d\eta^{b(n,\mathbf{f})}(\mathbf{t})\\
=\underset{n\rightarrow\infty}{\lim}F'(b(n,\mathbf{g}))\int_{\omega^{-1}(b(n,\mathbf{g})))}\chi(-\rho(\psi_{v}(\mathbf{t}))+a_{v,\mathbf{g}}^{(n)})F(\psi_{v}(\mathbf{t}))d\eta^{b(n,\mathbf{g})}(\mathbf{t}).
\end{array}$$
As the continuous function $F_H(c',c)$ can be uniformly approximated by a finite sum of continuous functions of the form $F'(c')F(c)$, one has that 
$$\begin{array}{ll}\underset{n\rightarrow\infty}{\lim}\int_{\omega^{-1}(b(n,\mathbf{f})))}\chi(-\rho(\psi_{v}(\mathbf{t}))+a_{v,\mathbf{f}}^{(n)})F_H(b(n,\mathbf{f}),\psi_{v}(\mathbf{t}))d\eta^{b(n,\mathbf{f})}(\mathbf{t})\\
=\underset{n\rightarrow\infty}{\lim}\int_{\omega^{-1}(b(n,\mathbf{g})))}\chi(-\rho(\psi_{v}(\mathbf{t}))+a_{v,\mathbf{g}}^{(n)})F_H(b(n,\mathbf{g}),\psi_{v}(\mathbf{t}))d\eta^{b(n,\mathbf{g})}(\mathbf{t}).
\end{array}$$
By Equation \eqref{H->F_H}, Equation \eqref{eq:uniqueness-for tricube} holds for $v\in \{0,1\}^{k+1}_*$.

 As for arbitrary two convergent sequences
$\{h_{1}^{(n)}\}_{n\in\mathbb{N}},\{h_{2}^{(n)}\}_{n\in\mathbb{N}}$
in $\mathcal{L}(T^{k+1}(X)\stackrel{\omega}{\rightarrow}C_G^{k+1}(X),A)$,
$\{h_{1}^{(n)}+h_{2}^{(n)}\}_{n\in\mathbb{N}}$ converges (see \cite[Lemma 2.3.9, Lemma 2.3.15]{candela2016cpt_notes}
for more details), we have
\[
\begin{array}{l}
\ \ \underset{n\rightarrow\infty}{\lim}(\int_{\omega^{-1}(b(n,\mathbf{f}))}\chi(\sum_{\nu\in\{0,1\}_{*}^{k+1}}(-1)^{|v|}(-\rho(\psi_{v}(\mathbf{t}))+a_{v,\mathbf{f}}^{(n)}))Fd\eta^{b(n,\mathbf{f})}(\mathbf{t})\\
=\underset{n\rightarrow\infty}{\lim}\int_{\omega^{-1}(b(n,\mathbf{g}))}\chi(\sum_{\nu\in\{0,1\}_{*}^{k+1}}(-1)^{|v|}(-\rho(\psi_{v}(\mathbf{t}))+a_{v,\mathbf{g}}^{(n)}))Fd\eta^{b(n,\mathbf{g})}(\mathbf{t})).
\end{array}
\]
By Equation \eqref{eq:the completion equation} and Equation \eqref{eq:uniqueness-3},
one has that
\[
\begin{array}{l}
\ \ \underset{n\rightarrow\infty}{\lim}\int_{\omega^{-1}(b(n,\mathbf{f}))}\chi(-(\rho(\psi_{\vec{0}}(\mathbf{t}))-a_{\vec{0},\mathbf{f}}^{(n)}))Fd\eta^{b(n,\mathbf{g})}(\mathbf{t})\\
=\underset{n\rightarrow\infty}{\lim}\int_{\omega^{-1}(b(n,\mathbf{g}))}\chi(-((\rho(\psi_{\vec{0}}(\mathbf{t}))-a_{\vec{0},\mathbf{g}}^{(n)}))Fd\eta^{b(n,\mathbf{g})}(\mathbf{t})).
\end{array}
\]
From Lemma \ref{measure-preserving},  $\psi_{\vec{0}}:\omega^{-1}(b(n,\bullet))\rightarrow\mathfrak{p}_{0}^{-1}(b(n,\bullet)_{\vec{0}})$
is measure-preserving. It holds that
\[
\begin{array}{l}
\ \ \underset{n\rightarrow\infty}{\lim}\int_{\mathfrak{p}_{0}^{-1}(b(n,\mathbf{f})_{\vec{0}})}\chi(-(\rho_{b(n,\mathbf{f})_{\vec{0}}}(c)+a_{\vec{0},\mathbf{f}}^{(n)}))Fd\mu_{C_G^{k+1}(X)}^{b(n,\mathbf{f})_{\vec{0}}}(c)\\
=\underset{n\rightarrow\infty}{\lim}\int_{\mathfrak{p}_{0}^{-1}(b(n,\mathbf{g})_{\vec{0}})}\chi(-(\rho_{b(n,\mathbf{g})_{\vec{0}}}(c)+a_{\vec{0},\mathbf{g}}^{(n)}))Fd\mu_{C_G^{k+1}(X)}^{b(n,\mathbf{g})_{\vec{0}}}(c)).
\end{array}
\]
Thus Equation (\ref{eq:uniqueness}) holds.
\end{proof}
\begin{lem}
\label{lem:extension of order k-1} For any $\mathbf{g}\in\mathcal{HK}^{k+1}(G)$,
$\mathbf{g}(Q^{k+1}_U(M))\subset Q^{k+1}_U(M)$. 
\end{lem}

\begin{proof}
According to the definition of  $Q^{k+1}_U(M)$ in Equation \eqref{definiton of Q(MU)}, $Q^{k+1}_U(M)$ consists of $\mathbf{f}:\{0,1\}^{k+1}\rightarrow M^{[k+1]}:\nu\mapsto-\rho_{c_{\nu}}+\mathbf{a}_{\nu}$ such that $c\in U$ and $\sum_{\nu\in\{0,1\}^{k+1}}(-1)^{|\nu|}\mathbf{a}_{\nu}=\rho(c)$.
For any $\mathbf{f}=(-\rho_{c_{v}}+\mathbf{a}_{v})_{v\in\{0,1\}^{k+1}}\in Q^{k+1}_U(M)$,
by Lemma \ref{lem:the action on M},
\[
\rho(\mathbf{g}c)=\rho(c)+\sum_{v\in\{0,1\}^{k+1}}(-1)^{|v|}\beta(\mathbf{g}_{v},c_{v})=\sum_{v\in\{0,1\}^{k+1}}(-1)^{|v|}(\mathbf{a}_{v}+\beta(\mathbf{g},c_{v})).
\]
By the definition of the action of $G$ (see Section \ref{sec: continuous}),
\[
\mathbf{g}\mathbf{f}=(-\rho_{\mathbf{g}_{v}c_{v}}+\mathbf{a}_{v}+\beta(\mathbf{g}_{v},c_{v}))_{v\in\{0,1\}^{k+1}},
\]
for any $\mathbf{g}\in\mathcal{HK}^{k+1}(G)$, therefore $\mathbf{g}\mathbf{f}\in Q^{k+1}_U(M)$.
\end{proof}

Now we can prove Proposition \ref{thm:k+1 uniqueness}.

\begin{proof}[Proof of Proposition \ref{thm:k+1 uniqueness}]

We claim that $C_G^{k+1}(M)\subset\overline{Q^{k+1}_U(M)}$. Then from Lemma \ref{uniqueness for Q}, the desired result follows. To prove the claim, notice that for any cube isomorphism $\sigma$, $\sigma(U)=U$ and for any $\mathbf{a}\in A^{[k+1]}$,
$$\sum_{v\in\{0,1\}^{k+1}}(-1)^{|v|}\mathbf{a}_{v}=\sum_{v\in\{0,1\}^{k+1}}\sgn(\sigma)(-1)^{|v|}\mathbf{a}_{\sigma(v)}.$$
Therefore $\sigma(Q^{k+1}_U(M))=Q^{k+1}_U(M).$ As $\sigma$ is a continuous map, one has that $\sigma(\overline{Q^{k+1}_U(M)})=\overline{Q^{k+1}_U(M)}$.
Suppose $x\in X'$ then $\hat{\mathfrak{p}}_0^{-1}(x)=-\rho_x+A$. As $(C_G^{k+1}(M),\mathcal{HK}^{k+1}(G))$ is minimal,
\[
C_G^{k+1}(M)=\overline{\{\mathbf{g}{-\rho_{x}}^{[k+1]}:\ \mathbf{g}\in\mathcal{HK}^{k+1}(G)\}},
\]
where ${-\rho_{x}}^{[k+1]}=(-\rho_{x},\ldots,-\rho_{x})$. As $(\hat{\mathfrak{p}}_0)^{[k+1]}(\overline{Q^{k+1}_U(M)})=C_G^{k+1}(X)$, there exists $\mathbf{a}\in A^{[k+1]}$ such that $$(-\rho_{x}+\mathbf{a}_v)_{v\in\{0,1\}^{k+1}}\in \overline{Q^{k+1}_U(M)}.$$
Notice that for any $\mathbf{f}\in \sigma(Q^{k+1}_U(M))$ and $\mathbf{b}\in A^{[k+1]}$, if $\sum_{v\in\{0,1\}^{k+1}}(-1)^{|v|}\mathbf{b}_{v}=0$, then 
\begin{equation}\label{invariant for L}
(\mathbf{f}_v+\mathbf{b}_{v})_{v\in\{0,1\}^{k+1}}\in \overline{Q^{k+1}_U(M)}.
\end{equation}
Therefore we can assume that $(-\rho_x,\ldots,-\rho_x,-\rho_x+s)\in \overline{Q^{k+1}_U(M)}$ for some $s\in A$. Let $\mathbf{b}=(0,\ldots,0,-s,-s)$, i.e.\  $\mathbf{b}_{(1,\ldots,1,0)}=-s,$ $\mathbf{b}_{(1,\ldots,1,1)}=-s$, $\mathbf{b}_v=0$ for other $v$. Then $\sum_{v\in\{0,1\}^{k+1}}(-1)^{|v|}\mathbf{b}_{v}=0$ and 
$$(-\rho_x,\ldots,-\rho_x,-\rho_x-s,-\rho_x)=(-\rho_x,\ldots,-\rho_x,-\rho_x+s)+\mathbf{b}\in \overline{Q^{k+1}_U(M)}.$$
As there exists a cube isomorphism $\sigma$ such that $\sigma(-\rho_x,\ldots,-\rho_x,-\rho_x-s,-\rho_x)=(-\rho_x,\ldots,-\rho_x,-\rho_x-s)$, one has that $(-\rho_x,\ldots,-\rho_x,-\rho_x-s)\in \overline{Q^{k+1}_U(M)}$. From Lemma \ref{uniqueness for Q}, $s=0$. Therefore $(-\rho_x,\ldots,-\rho_x)\in \overline{Q^{k+1}_U(M)}$. From Lemma \ref{lem:extension of order k-1} and the definition of $C_G^{k+1}(M)$, $C_G^{k+1}(M)\subset \overline{Q^{k+1}_U(M)}$.
\end{proof}

\begin{thm}\label{thm:extension of order k} The map
$\mathfrak{p}_0:M\rightarrow X$ is a fibration of order at most $k$. In particular $\nrp_G^{[k]}(M\rightarrow X)=\triangle.$
\end{thm}

\begin{proof}
By Theorem \ref{thm:distal->fibrant}, as $(M,G)$ is distal and minimal, $(M,C_G^{\bullet}(M))$
is a fibrant cubespace. Moreover $(C_G^{n}(M),\mathcal{HK}^{n}(G))$ is minimal for any $n\in \mathbb{N}$, being the closure of a unique orbit in a distal system. By \cite[Theorem 1.29]{GL2019} a factor map between minimal distal systems is a fibration. Thus
 $\hat{\mathfrak{p}}_{0}:(M,C_G^{\bullet}(M))\rightarrow(X,C_G^{\bullet}(X))$
is a fibration.

Let $x,y\in M$. Recall from Definition \ref{NRP(M-G)} that $x\sim_{k}y$ if and only if
$c^{k+1}_y(x)\in C_G^{k+1}(M)$. Obviously $c^{k+1}_x(x)\in C_G^{k+1}(M)$. Note $c^{k+1}_y(x)_{\nu}=c^{k+1}_y(x)_{\nu}=x$ for $\nu\in\{0,1\}_{*}^{k+1}$. Thus if  $\hat{\mathfrak{p}}_0(x)=\hat{\mathfrak{p}}_0(y)$ and $x\sim_{k}y$
then $\hat{\mathfrak{p}}_0^{[k+1]}(c^{k+1}_y(x))=\hat{\mathfrak{p}}_0^{[k+1]}(c^{k+1}_x(x))\in C_G^{k+1}(X)$.
From Proposition \ref{thm:k+1 uniqueness}, $c^{k+1}_y(x)_{\vec{0}}=y=c^{k+1}_x(x)_{\vec{0}}=x$. As
$x=y$ we conclude $\nrp_G^{[k]}(M\rightarrow X)=\triangle.$
\end{proof}
\begin{rem}

Recall Definition \ref{k ergodic fibration}  of a $k$-ergodic fibration. It is easy to prove that if  the group $L$  from Definition \ref{def:nilcycle}, satisfies
$$ L=\{\mathbf{a}\in A^{[k+1]}:\theta_{k+1} (\mathbf{a})=0\},$$
then the extension $M\rightarrow X$ from the previous theorem is a $k$-ergodic fibration.
\end{rem}

\section{A new proof of the Host-Kra structure theorem.}\label{new proof of the Host-Kra structure theorem}

In this section, we assume that $G$ is a finitely generated abelian group (e.g., $G=\Z^k$, $k\in \N$) and that $(X,\mathcal{B},\mu, G)$ is an ergodic m.p.s. All cited lemmas generalize effortlessly to this setting. We need some facts about nilsystems:
\begin{prop}
 (\cite[Theorem B]{Leibman2005}; proof of \cite[Theorem 2.19]{Leibman2005} )\label{Nilsystem fact}
Let $(X,G)$ be a nilsystem.
\begin{enumerate}
       \item  Every subsystem of $(X,G)$ is a nilsystem. 
              
       \item  The nilsystem $(X,G)$ is uniquely ergodic iff it is minimal iff it is ergodic w.r.t.\  its  Haar measure (see Section \ref{subsec:Nilsystem-and-a system at most d}).
\end{enumerate}

\end{prop}

For the next proposition, recall the definition of $Z_{k}(X)$ in Subsection \ref{subsec:The Host-Kra factors}.

\begin{prop}(\cite[Chapter 18, Theorem 6]{HK2018})
\label{abelian group extension} Let $(X,\mu,G)$ be an ergodic system.
Then $Z_{k}=Z_{k}(X)$ is an abelian group extension of $Z_{k-1}=Z_{k-1}(X)$,
i.e., $Z_{k}=Z_{k-1}\times A$, where $A$ is a compact abelian group.
\end{prop}

\subsection{Overview of the proof.}\label{subsec:overview_proof}

Theorem \ref{thm:main theorem} allows us to present a new approach to the Host-Kra structure
theorem \cite[Theorem 10.1]{HK05} in the generality of finitely generated abelian groups. In this section, we will prove the following theorem.
\begin{thm}\label{thm:order k isomorphic as m.p.s. to nilsystem}
Let $G$ be a finitely generated abelian group (e.g., $G=\Z^k$, $k\in \N$). An ergodic system $(Y,\nu,G)$ of order $k$ is isomorphic
to an inverse limit of minimal nilsystems of degree at most $k$  as m.p.s.
\end{thm}

\begin{proof}[Proof of Theorem \ref{thm:order k isomorphic as m.p.s. to nilsystem}]

We will prove the statement by induction on $k$. For $k=1$, the statement follows as $Y=Z_1(Y)$ is a Kronecker system (\cite[Proposition 8, Chapter 9]{HK2018}).

Denote $X=Z_{k-1}(Y)$, then $Z_{k-1}(X)=X$ (\cite[Chapter 9, Corollary 9]{HK2018}). By the inductive assumption $X$ has a topological model which is an inverse limit of minimal  nilsystems of degree at most $(k-1)$. Abusing notation we still denote by $(X,G)$ the topological model of  $(X,\mu,G)$ such that  $(X,G)=\underleftarrow{\lim}(X_{m}=H_m/\Gamma_m,G)$, where $(X_{m},G)$ are minimal nilsystems of degree at most $(k-1)$. It is not hard to represent the nilsystems $(X_{m},G)$ as finite towers of distal extensions. Thus $(X_{m},G)$ are distal and so is their inverse limit $(X,G)$.  

For any $\ell\in \mathbb{N}$, it is easy to see that $(C_G^{\ell}(X),\mathcal{HK}^{\ell}(G))=\underleftarrow{\lim}(C_G^{\ell}(X_m),\mathcal{HK}^{\ell}(G))$. Indeed, note that $C_G^{\ell}(X_m)\subset H_m^{[\ell]}/\Gamma^{[\ell]}$ are closed and $\mathcal{HK}^{\ell}(G)$-invariant subsets and therefore by Proposition \ref{Nilsystem fact}-1, $(C_G^{\ell}(X_{m}),\mathcal{HK}^{\ell}(G))$ are minimal\footnote{Note that $C_G^\ell(X_m)$ is the orbit closure of one element and therefore by distality is minimal.} nilsystems of degree at most $(k-1)$. As $\mathcal{HK}^{\ell}(G)$ is a finitely generated abelian group, from  Proposition \ref{Nilsystem fact}-2, $(C_G^{\ell}(X_{m}),\mathcal{HK}^{\ell}(G))$ are uniquely ergodic. As an inverse limit of uniquely ergodic systems is uniquely ergodic (\cite[Proposition 3.5]{Lehrer87}),  $(C_G^{\ell}(X),\mathcal{HK}^{\ell}(G))$ is uniquely ergodic. Setting $\ell=2k+2$ we have that $(X,G)$ is a $(2k+2)$-cube uniquely ergodic system.

We denote $X=Z_{k-1}(Y)$ and $\mu=\nu_{k-1}$. Let $\pi_k :Y \rightarrow Y$ be the (measurable) factor map induced by this identification. Using Proposition \ref{abelian group extension}, we write
\begin{equation}\label{eq:ext}
Y =(X\times A,\nu=\mu\times m_{\haar(A)},G)\stackrel{\pi_k}{\rightarrow}(X,\mu,G),
\end{equation}
where A is a compact abelian group. 
 In Section \ref{subsec:Construction-of-a nilcycle from the HK extension} we show that \emph{there exists a nilcycle of degree $k$ for the extension \eqref{eq:ext}}. We now have the prerequisites to use Theorem \ref{thm:main theorem}. We conclude that $\pi_k$ has a strictly
ergodic topological model $r:(M,G)\rightarrow (X,G)$ which is a fibration
of order at most $k$. Let $(x,y)\in \nrp_G^{[k]}(M)$.  By Theorem \ref{thm:dynamical sturcture strong ver}, $\nrp_G^{[k+1]}(X)=\triangle$. As $(r\times~r)(\nrp_G^{[k]}(M))\subset \nrp_G^{[k]}(X)$, one has that $r(x)=r(y)$. As $(M,G)$ is an extension
of order $k$ of $(X,G)$, $x=y$. Thus $\nrp_G^{[k]}(M)=\triangle$. By Theorem \ref{thm:dynamical sturcture strong ver}, $(M,G)$ is an inverse limit of
nilsystems of degree at most $k$.
 \end{proof}

\subsection{Construction of a nilcycle for Host-Kra factors.\label{subsec:Construction-of-a nilcycle from the HK extension}}

In this section we construct a nilcycle $\rho:C_G^{k+1}(X)\rightarrow A$ for the extension \eqref{eq:ext}. This is made possible by analyzing the abelian group extension $\pi_k^{[k+1]}:(Y^{[k+1]},\nu^{[k+1]},G^{[k+1]})\rightarrow (X^{[k+1]},\mu^{[k+1]},G^{[k+1]})$ induced from the abelian group extension $\pi_{k}:(Y,\nu, G)\rightarrow (X,\mu,G)$. Identify $(Y^{[k+1]},\nu^{[k+1]})=(X^{[k+1]}\times A^{[k+1]},\nu^{[k+1]})$.

\begin{lem}
\label{lem:mu k+1 inv under g edge-1}\cite[Chapater 9, Proposition 3]{HK2018} The measure $\nu^{[k+1]}$
is invariant under $a^{\alpha}$ for all edges $\alpha$ and all $a\in A$.
\end{lem}

Define
\begin{equation}\label{def:tilde_L}
\tilde{L}=\{\mathbf{u}\in A^{[k+1]}:\ \theta_{k+1}(\mathbf{u})=\sum_{v\in\{0,1\}^{k+1}}(-1)^{|v|}\mathbf{u}_{v}=0\}.
\end{equation}
\noindent
This group will play an important role in the analysis of $\nu^{[k+1]}$.

\begin{lem}
\label{lem:muk*-is-invariant} The measure $\nu^{[k+1]}$ is invariant under
$\tilde{L}$ and the measure $\nu^{[k+1]*}$ is invariant under $A^{[k+1]*}$ for
any $k\geq1$.
\end{lem}

\begin{proof}
We claim that for any $\mathbf{u}\in \tilde{L}$, there exist finite collections of edges
$\{\alpha_{s}\}_{s}$ and group elements $\{g_s\}_s\subset A$ such that $\mathbf{u}=\sum_{s}g_{s}^{\alpha_{s}}$. From Lemma \ref{lem:mu k+1 inv under g edge-1}, $\nu^{[k+1]}$
is invariant under $g^{\alpha}$. If the claim holds, then $\nu^{[k+1]}$
is invariant under $\mathbf{u}$.

We proceed to prove the claim. For $k=1$, let $\alpha_{1}=\{(0,0),(0,1)\},\alpha_{2}=\{(0,1),(1,1)\},\alpha_{3}=\{(1,1),(1,0)\}$.
For any $\mathbf{u}\in L$, let $g_{1}=\mathbf{u}(0,0),\ g_{2}=\mathbf{u}(0,1)-g_{1},\ g_{3}=\mathbf{u}(1,1)-g_{2}$,
then $\mathbf{u}=g_{1}^{\alpha_{1}}+g_{2}^{\alpha_{2}}+g_{3}^{\alpha_{3}}$.

Inductively, we assume that the claim for $k$ holds. Now we prove
the claim for $k+1$. For $\mathbf{u}\in \tilde{L}$, $\sum_{v\in[k+1]}(-1)^{|v|}\mathbf{u}(v)=0$.
Let $\alpha_{0}=\{(\vec{0},0),(\vec{0},1)\}$ $g_{0}=\sum_{v\in[k]\times\{0\}}(-1)^{|v|}\mathbf{u}(v)$,
then $\mathbf{u}'=\mathbf{u}-g_{0}^{\alpha_{0}}\in \tilde{L}$ satisfies

\[
\sum_{v\in[k]\times\{0\}}(-1)^{|v|}\mathbf{u}(v)=0\text{ and }\sum_{v\in[k]\times\{1\}}(-1)^{|v|}\mathbf{u}(v)=0.
\]
By the inductive assumption, we know that there exist a finite collections of edges $\{\alpha_{s}\}_{s}$
such that $\mathbf{u}'=\sum_{s}g_{s}^{\alpha_{s}}$. Then $\mathbf{u}=g_{0}^{\alpha_{0}}+\sum_{s}g_{s}^{\alpha_{s}}$,
which proves the claim.

The measure $\nu^{[k+1]*}$ is induced by the projection $\mathfrak{p}^*:Y^{[k+1]}\rightarrow Y^{[k+1]*}:c\rightarrow C_G^{*}$.
For any $\mathbf{a}^{*}\in A^{[k+1]*}$, let $a_{0}=\sum_{v\in[k+1]*}(-1)^{|v|+1}\mathbf{a}^{*}(v)$,
then $(a_{0},\mathbf{a}^{*})\in \tilde{L}$.

As $\nu^{[k+1]}$ is invariant under $\tilde{L}$, $\nu^{[k+1]}$
is invariant under $(a_{0},\mathbf{a})$. One has that for any $\nu^{[k+1]*}$-measurable
set $B$,
$$
\nu^{[k+1]*}(\mathbf{a}^{*}B)=\nu^{[k+1]}(Y\times\mathbf{a}^{*}B)=\nu^{[k+1]}((a_{0},\mathbf{a}^{*})(Y\times B))=\nu^{[k+1]}(Y\times B)=\nu^{[k+1]*}(B).
$$
\end{proof}

 Let $\{\nu^{[k+1](c)}\}_{c\in X^{[k+1]}}$ be the measure disintegration w.r.t.\  the
projection $p_1=\pi_{k}^{[k+1]}:X^{[k+1]}\times A^{[k+1]}\rightarrow X^{[k+1]}:(c,\mathbf{a})\mapsto c$. The measures $\nu^{[k+1](c)}$ can be seen as measures
of $A^{[k+1]}.$ We will now analyze the measures $\{\nu^{[k+1](c)}\}$. 

\begin{lem}\cite[Chapter 9, Theorem 15(iii)]{HK2018}
\label{prop:There-exists-aisomorphism} Let $(Y,\nu,G)$ be a system of order $k$. Then the map
$p_{2}:Y^{[k+1]}\rightarrow Y^{[k+1]*}:(c,\mathbf{a})\rightarrow(c^{*},\mathbf{a}^{*})$ is a Borel isomorphism.
\end{lem}

\begin{lem}
\label{ergodic L} For $\mu^{[k+1]}$-a.e. $c\in X^{[k+1]}$,
$(p_{1}^{-1}(c),\nu^{[k+1](c)},\tilde{L})$ is ergodic.
\end{lem}

\begin{proof}
From Lemma \ref{prop:There-exists-aisomorphism}, $p_{2}$ is a Borel
isomorphism. As $(C^{k+1}_G(X),\mathcal{HK}^{k+1}(G),\mu^{[k+1]})$ is uniquely ergodic, one has that $p_1$ is measure-preserving. Thus for $\mu^{[k+1]}$-a.e. $c\in X^{[k+1]}$,
$p_{2}:(p_{1}^{-1}(c),\nu^{[k+1](c)})\rightarrow(p_{2}(p_{1}^{-1}(c)),(p_{2})_{*}\nu^{[k+1](c)})$
is a Borel isomorphism. As by Lemma \ref{lem:muk*-is-invariant} $\nu^{[k+1]}$ is $\tilde{L}$-invariant and $p_{1}^{-1}(c)$ is $\tilde{L}$-invariant,
it holds that $\nu^{[k+1](c)}$ is $\tilde{L}$-invariant for $\mu^{[k+1]}$-a.e. $c\in X^{[k+1]}$. Consider the m.p.s.\ $(p_{1}^{-1}(c),\nu^{[k+1](c)},\tilde{L}).$
 Notice that  $\tilde{L}\rightarrow A^{[k+1]*}: \mathbf{a'}\rightarrow \mathbf{a'}^*$ is a group isomorphism. One has that $p_{2}(p_{1}^{-1}(c))=\{c^{*}\}\times A^{[k+1]*}$  and for any $\mathbf{a}\in A^{[k+1]},\mathbf{a'}\in \tilde{L}$,
 $$p_2(\mathbf{a'}(c,\mathbf{a}))=p_2(c,\mathbf{a}+\mathbf{a'})=(c^*,\mathbf{a}^*+\mathbf{a'}^*)=\mathbf{a'}^*(c^*,\mathbf{a}^*).$$
Therefore for $\mu^{[k+1]}$-a.e. $c$,
$$p_2:(p_{1}^{-1}(c),\nu^{[k+1](c)},\tilde{L})\rightarrow (p_{2}(p_{1}^{-1}(c)),(p_{2})_{*}\nu^{[k+1](c)},A^{[k+1]*} )$$
is a Borel isomorphism and $(p_{2})_{*}\nu^{[k+1](c)}$
is $A^{[k+1]*}$-invariant. As $p_{2}(p_{1}^{-1}(c))=\{C_G^{*}\}\times A^{[k+1]*}$, for the Haar measure $m_{\haar(A^{[k+1]*})}$ for $A^{[k+1]*}$, it holds that
$$(p_{2})_*\nu^{[k+1](c)}=\delta_{C_G^*}\times m_{\haar(A^{[k+1]*})}.$$  In particular, $(p_{2}(p_{1}^{-1}(c))=\{c\}\times A^{[k+1]*},(p_{2})_{*}\nu^{[k+1](c)},A^{[k+1]*})  \text{ is ergodic.}$ We conclude that $\mu^{[k+1]}$-a.e. $c\in X^{[k+1]}$, $(p_{1}^{-1}(c),\nu^{[k+1](c)},\tilde{L})$
is ergodic.
\end{proof}

We are almost ready to define the nilcycle for the extension \eqref{eq:ext}. We start by two auxiliary lemmas.
The first lemma is proven in \cite[Lemma 5.6]{HKM10} for $\Z$-actions but easily generalizes to the context of finitely generated abelian actions, in particular $\Z^k$-actions. 
 \begin{lem}\label{fullsupport}\label{muk=mu_C} 
Let $G$ be a finitely generated abelian group. Suppose $(X,G)$ is a t.d.s.\ and $\mu$ is an ergodic measure of $(X,G)$. Then $\mu^{[k]}(C_G^{k}(X))=1$ for $k\geq 1$. In particular, if $(X,G)$ is a $(k+1)$-cube uniquely ergodic system, then $\mu_{C_G^{k+1}(X)}=\mu^{[k+1]}$.
\end{lem}

\begin{lem}
\label{lem:subgroup invariance} Let $\gamma$ be a measure on a metrizable
compact abelian group $A$ which is invariant and ergodic under the
action of a closed subgroup $H\subset A$, then $\gamma=m_{\haar(H)}+a$
for some $a\in A,$ where $m_{\haar(H)}$ is the Haar measure on $H$ and $m_{\haar(H)}+a$ is defined by $(m_{\haar(H)}+a)(B)=m_{\haar(H)}\left((B+a)\cap H\right)$
for any $B\subset A$.
\end{lem}

\begin{proof}
Let $\gamma'$ be the push forward of $\gamma$ by the canonical projection,
i.e.\  for any measurable set $B\subset A/H$, $\gamma'(B)=\gamma(\{a\in A:a+H\in B\})$.
If $\gamma'$ is not an atomic measure, as $A/H$ is a Polish space,
there exists $B_{0}\in\mathcal{B}(A/H)$ such that $0<\gamma'(B_{0})<1$
\cite[page 14]{RR1981}. But $\{a\in A:a+H\in B\}$ is $H$-invariant,
then by ergodicity, $\gamma'(B_{0})=\gamma(\{a\in A:a+H\in B_{0}\})=0$ or
$1$, which is a contradiction. Therefore $\gamma'$ is an atomic measure,
i.e.\  $\supp(\gamma)=-a+H$ for some $a\in A$. Since $m_{\haar(H)}$ is the unique $H$-invariant measure supported
on $H$, $m_{\haar(H)}+a$ is the unique $H$-invariant measure supported
on $-a+H$. Thus $\gamma=m_{\haar(H)}+a$ for some $a\in A.$
\end{proof}

By Lemma \ref{fullsupport}, $\mu^{[k+1]}$ is supported on $C^{k+1}_G(X)$. From Lemma \ref{ergodic L}, there is a Borel subset $V\subset C_G^{k+1}(X)$
such that $\mu^{[k+1]}(V)=1$ and for $c\in V$, $\nu^{[k+1](c)}$
is ergodic under the action of $\tilde{L}$. From Lemma \ref{lem:subgroup invariance}, for $c\in V$,
$\nu^{[k+1](c)}=\delta_{c}\times(m_{\haar(\tilde{L})}+\mathbf{a}_{c})$
(i.e.\  $\nu^{[k+1](c)}(B)=m_{\haar(\tilde{L})}(B+\mathbf{a}_{c})$ ) for
some $\mathbf{a}_{c}\in A^{[k+1]}$, where $\delta_{c}$ is the
Dirac measure. By Theorem \ref{thm:(measure-disintegration-theorem)}, the map $q:X^{[k+1]}\rightarrow \mathcal{M}(Y^{[k+1]}):c\rightarrow \nu^{[k+1](c)}$ is a Borel map.  Define the closed subset $H\subset \mathcal{M}(Y^{[k+1]})$ given by $H:=\{\delta_c\times(m_{\haar(\tilde{L})}+\mathbf{a}): \mathbf{a}\in A^{[k+1]},c\in C_G^{k+1}(X)\}$. Thus the map $q_{|V}:V\rightarrow H$ is Borel. Due the definition of $\tilde{L}$ (see \eqref{def:tilde_L}), the map $\theta_{k+1}:A^{[k+1]}\rightarrow A$ induces a well-defined and continuous map $\tilde{\theta}_{k+1}:H\rightarrow A$ given by $\tilde{\theta}(\delta_c\times(m_{\haar(\tilde{L})}+\mathbf{a})=\theta_{k+1}(\mathbf{a})$.  Let $\rho:C_G^{k+1}(X)\rightarrow A$ be defined by
\[
\rho(c)=\tilde{\theta}_{k+1}(q_{|V}(c))
\]
for $c\in V$ and $0_A$ otherwise. As $\rho$ is the composition of two Borel maps (up to measure zero), one has the following fact:
\begin{fact}
The map $\rho:\ C_G^{k+1}(X)\rightarrow A$ is
Borel.
\end{fact}

Now we will show that $\rho$ is a nilcycle. First we need an auxiliary lemma.

\begin{lem}
\label{lem:key joining} Let $\pi_{L},\pi_{U}:Y^{[k+1]}=Y^{[k]}\times Y^{[k]}\to Y^{[k]}$
be given by $\pi_{L}(d_{1},d_{2})=d_{1}$ and $\pi_{U}(d_{1},d_{2})=d_{2}$.
There exists a joining of $(Y^{[k+1]},\nu^{[k+1]})$ with itself $(Y^{[k+1]}\times Y^{[k+1]},\lambda)$
so that

\begin{enumerate}
\item \label{enu:edge map}For $\pi_{E}:Y^{[k+1]}\times Y^{[k+1]}\to Y^{[k]}\times Y^{[k]}=Y^{[k+1]}$
given by $\pi_{E}(c_{1},c_{2})=(\pi_{L}(c_{1}),\pi_{U}(c_{2}))$ it
holds $\pi_{E*}\lambda=\nu^{[k+1]}$.
\item \label{enu:support} For $S:=\{(c_{1},c_{2})\in(Y^{[k+1]})^{2}|\,\pi_{U}(c_{1})=\pi_{L}(c_{2})\}$, it holds $\lambda(S)=1$.

\item \label{enu:projection to gluespace}For $\pi_{k}^{[k+1]}\times\pi_{k}^{[k+1]}:C_G^{k+1}(X)\times A^{[k+1]}\times C_G^{k+1}(X)\times A^{[k+1]}\to C_G^{k+1}(X)\times C_G^{k+1}(X)$
it holds $(\pi_{k}^{[k+1]}\times\pi_{k}^{[k+1]})_{*}\lambda=\mu_{\mathcal{P}^{k+1}(X)}$.
\end{enumerate}
\end{lem}
\begin{proof}

Similarly to $\pi_{L},\pi_{U}$, we define
$\tilde{\pi}_{L},\tilde{\pi}_{U}:\mathcal{HK}^{k+1}(G)\rightarrow\mathcal{HK}^{k}(G)$. These maps are onto by \cite[Proposition 3.3 and Lemma 4.3]{GGY2018}. By Definition \ref{def: muk definition}, $\nu^{[k+1]}$ is defined for bounded measurable functions $F_{1}, F_{2}$ by
\begin{equation}\label{def of mu k}
\int_{Y^{[k+1]}} F_{1}(c_{1})F_{2}(c_{2})d\nu^{[k+1]}(c_{1},c_{2})=\int_{Y^{[k]}}\mathbb{E}(F_{1}|\mathcal{I}_{G}^{[k]})(c)\mathbb{E}(F_{2}|\mathcal{I}_{G}^{[k]})(c)d\nu^{[k]}(c).
\end{equation}
Notice that for any bounded measurable function $F:Y^{[k]}\rightarrow \mathbb{C},$ $F(\pi_U(c_1,c_2))=F(c_2)$, therefore
\[
\int_{Y^{[k+1]}} F(c_2)d\big ((\pi_{U})_{*}\nu^{[k+1]}\big )(c_2)=\int_{Y^{[k]}} F(c_2)d\nu^{[k+1]}(c_1,c_2)=\int_{Y^{[k]}} \mathbb{E}(F|\mathcal{I}_{G}^{[k]})(c)d\nu^{[k]}(c)=\int_{Y^{[k]}} F(c)d\nu^{[k]}(c).
\]
Thus one has that $(\pi_{U})_{*}\nu^{[k+1]}=\nu^{[k]}$. Therefore $(\pi_{U},\tilde{\pi}_{U}),(\pi_{L},\tilde{\pi}_{L})$
are factor maps. Consider the \emph{conditional product joining}\footnote{i.e.\  for any two functions $H_1,H_2:Y^{[k+1]}\rightarrow \mathbb{C}$, $\int_{Y^{[k+1]}} H_1(c_1)H_2(c_2)\lambda(c_1,c_2)=\int_{Y^{[k]}} \mathbb{E}(H_1|\pi_U(Y^{[k+1]}))(c)\mathbb{E}(H_2|\pi_L(Y^{[k+1]}))(c)d\nu^{[k]}(c)$. See \cite[Definition 9.1]{F77}.} $\lambda$
of two copies of $(Y^{[k+1]},\nu^{[k+1]},\mathcal{HK}^{k+1}(G))$
over $(Y^{[k]},\nu^{[k]},\mathcal{HK}^{k}(G))$ through the
maps $(\pi_{U},\tilde{\pi}_{U})$ and $(\pi_{L},\tilde{\pi}_{L})$.
Thus $\lambda(S)=1$,
which is Property \ref{enu:support}.

As $\mathcal{P}^{k+1}(G)=\{(\mathbf{g}_{1},\mathbf{g}_{2})\in\mathcal{HK}^{k+1}(G)\times\mathcal{HK}^{k+1}(G)|\tilde{\pi}_{U}(\mathbf{g}_{1})=\tilde{\pi}_{L}(\mathbf{g}_{2})\}$,
it holds that $\lambda$ is $\mathcal{P}^{k+1}(G)$-invariant. By Lemma \ref{fullsupport}, $\mu^{[k+1]}(C_G^{k+1}(X))=1$. As
$(\pi_{k}^{[k+1]})_{*}\nu^{[k+1]}=\mu^{[k+1]}$, $$\nu^{[k+1]}(C_G^{k+1}(X)\times A^{[k+1]})=\mu^{[k+1]}(C_G^{k+1}(X))=1.$$  As  $\lambda$ is a joining,  $\lambda((C_G^{k+1}(X)\times A^{[k+1]})^2)=1$.
 Property (\ref{enu:projection to gluespace}) holds as $\mu_{\mathcal{P}^{k+1}(X)}$
is the unique ergodic measure of $(\mathcal{P}^{k+1}(X) ,\mathcal{P}^{k+1}(G))$ by Proposition \ref{prop:glueing pair uniquely ergodic}, since $(X,G)$ is $(k+2)$-cube uniquely ergodic and $(X,G)$ is distal as noted in Subsection \ref{subsec:overview_proof}.

For Property (\ref{enu:edge map}) , let $\mathcal{U}=\{Y^{[k]}\times D'|\ D'\in\mathcal{B}(Y^{[k]})\}=\pi_{U}^{-1}(\mathcal{B}(Y^{[k]}))$ and 
$\mathcal{{L}}=\pi_{L}^{-1}(\mathcal{B}(Y^{[k]}))$ be
 sub $\sigma$-algebras of $\mathcal{B}(Y^{[k+1]})$. Fix $D\in\mathcal{B}(Y^{[k]})$. Let $g=\mathbb{E}(1_{D\times Y^{[k]}}|\mathcal{U})$. By Section \ref{Sec: Con exp}, for any $Y^{[k]}\times D'\in\mathcal{U}$, $1_{Y^{[k]}\times D'}(c_{1},c_{2})=1_{D'}(c_{2})$. Thus as $g$ is $\mathcal{U}$-measurable,
there exists a function $g':Y^{[k]}\rightarrow\mathbb{C}$ such that
for $\nu^{[k+1]}$-a.e. $(c_{1},c_{2})$, $g(c_{1},c_{2})=g'(c_{2})$. Notice that for any $F\in\mathcal{B}(Y^{[k]})$,
$$
\nu^{[k+1]}(D\times F)
=\int_{Y^{[k+1]}}1_{D\times Y^{[k]}}1_{Y^{[k]}\times F}d\nu^{[k+1]}=\int_{Y^{[k+1]}}\mathbb{E}(1_{D\times Y^{[k]}}1_{Y^{[k]}\times F}|\mathcal{U})d\nu^{[k+1]}.
$$
As $Y^{[k]}\times F\in\mathcal{U}$, $1_{Y^{[k]}\times F}(c_{1},c_{2})=1_{F}(c_{2})$ and $(\pi_U)_{*}\nu^{[k+1]}=\nu^{[k]}$, by Equation (\ref{con exp commute}),
\[
\begin{array}{l}
\nu^{[k+1]}(D\times F)
=\int_{Y^{[k+1]}}g1_{Y^{[k]}\times F}d\nu^{[k+1]}\\
=\int_{Y^{[k+1]}}(g'1_{F})(c_{2})d\nu^{[k+1]}(c_{1},c_{2})
=\int g'1_{F}d\nu^{[k]}.
\end{array}
\]
By the definition of $\nu^{[k+1]}$ and Equation \eqref{conditional expection} and \eqref{con exp commute},
\[
\begin{array}{l}
\nu^{[k+1]}(D\times F)
=\int_{Y^{[k]}}\mathbb{E}(1_{D}|\mathcal{I}_{G}^{[k]})\mathbb{E}(1_{F}|\mathcal{I}_{G}^{[k]})d\nu^{[k]}\\
=\int_{Y^{[k]}}\mathbb{E}\left(\mathbb{E}(1_{D}|\mathcal{I}_{G}^{[k]})1_{F}|\mathcal{I}_{G}^{[k]}\right)d\nu^{[k]}
=\int_{Y^{[k]}}\mathbb{E}(1_{D}|\mathcal{I}_{G}^{[k]})1_{F}d\nu^{[k]}.
\end{array}
\]
We conclude
\[
\int_{Y^{[k]}} g'1_{F}d\nu^{[k]}=\int_{Y^{[k]}}\mathbb{E}(1_{D}|\mathcal{I}_{G}^{[k]})1_{F}d\nu^{[k]}.
\]
As $F$ is arbitrary,
this implies
\[
g'(c)=\mathbb{E}(1_{D}|\mathcal{I}_{G}^{[k]})(c)\text{ for }\nu^{[k]}\text{-a.e. }c.
\]
By a similar argument to the one applied to $g'$ above, one may find a measurable function $f':Y^{[k]}\rightarrow\mathbb{C}$ so that $f'\circ \pi_L =\mathbb{E}(1_{Y^{[k]}\times F}|\mathcal{L})$. It holds that $f'(c)=\mathbb{E}(1_{F}|\mathcal{I}_{G}^{[k]})(c)$ for $\nu^{[k]}$-a.e. $c$.
Now we consider
\[
\begin{array}{l}
\pi_{E*}\lambda(D\times F)
=\int1_{\pi_{E}^{-1}(D\times F)}d\lambda
=\int1_{D\times Y^{[k]}}(c_{1})1_{Y^{[k]}\times F}(c_{2})d\lambda(c_{1},c_{2}).
\end{array}
\]
Recall that $\lambda$ is a conditional product, therefore
\[
\begin{array}{l}
\int1_{D\times Y^{[k]}}1_{Y^{[k]}\times F}d\lambda
=\int g'(c)f'(c)d\nu^{[k]}
=\int\mathbb{E}(1_{D}|\mathcal{I}_{G}^{[k]})(c)\mathbb{E}(1_{F}|\mathcal{I}_{G}^{[k]})(c)d\nu^{[k]}(c)\\
=\int1_{D}(c_{1})1_{F}(c_{2})\nu^{[k+1]}(c_{1},c_{2})
=\nu^{[k+1]}(D\times F),
\end{array}
\]
which gives $\pi_{E*}\lambda(D\times F)=\nu^{[k+1]}(D\times F)$.
\end{proof}

\begin{thm}
\label{thm: measure compatible weak cocycle} The map $\rho:C_G^{k+1}(X)\to A$
is a nilcycle.
\end{thm}

\begin{proof}

Let $$D=\{(c,\mathbf{a})\in X^{[k+1]}\times A^{[k+1]}:\rho(c)=\theta_{k+1}(\mathbf{a})=\sum_{v\in\{0,1\}^{k+1}}(-1)^{|v|}\mathbf{a}_{v}\}.$$
We claim that $\nu^{[k+1]}(D)=1$. Indeed, we notice that $\nu^{[k+1]}(D)=\int_{c\in X^{[k+1]}}\nu^{[k+1](c)}(D)d\mu^{[k+1]}$.
Since $\nu^{[k+1](c)}(\tilde{L}+\mathbf{a}_{c})=1$ and $(c,\tilde{L}+\mathbf{a}_{c})\subset D$
for $\mu^{[k+1]}$ a.e. $c\in X^{[k+1]}$, $\nu^{[k+1]}(D)=1.$

Let $\sigma\colon\{0,1\}^{k+1}\to\{0,1\}^{k+1}$ be a discrete cube
isomorphism. The map $\sigma$ acts on $X^{[k+1]}$ and $A^{[k+1]}$ by $\sigma(c)=(c_{\sigma(v)})_{v\in\{0,1\}^{k+1}}$,
$\sigma(\mathbf{a})=(\mathbf{a}_{\sigma(v)})_{v\in\{0,1\}^{k+1}}$
and $\sigma(c,\mathbf{a})=(\sigma(c),\sigma(\mathbf{a}))$. First we show that $\rho\circ\sigma(c)=\sgn(\sigma)\cdot\rho(c)$ $\mu_{C_G^{k+1}(X)}$-a.s. Let $D_{\sigma}=D\cap\sigma^{-1}(D)$. From Lemma \ref{lem:muk is invariant under symmetries}
$D_{\sigma}$ has full $\nu^{[k+1]}$ measure. For $(c,\mathbf{a})\in D_{\sigma}$
one has $\theta_{k+1}(\mathbf{a})=\rho(c)$. Thus
\[
\rho(\sigma(c))\stackrel{(\sigma(c),\sigma(a))\in D}{=}\theta_{k+1}(\sigma(\mathbf{a}))=\sgn(\sigma)\theta_{k+1}(\mathbf{a})=\sgn(\sigma)\cdot\rho(c).
\]
Define $F\subset C_G^{k+1}(X)$ by $F_{\sigma}=\{c\in
C_G^{k+1}(X)|\,\rho\circ\sigma(c)=\sgn(\sigma)\cdot\rho(c)\}$.
By Lemma \ref{muk=mu_C}, $\mu_{C_G^{k+1}(X)}=\mu^{[k+1]}=(\pi_{k}^{[k+1]})_{*}\nu^{[k+1]}$. 
By the previous calculation, $D_{\sigma}\subset\pi_{k}^{-1}(F_{\sigma})$
therefore $\mu_{C_G^{k+1}(X)}(F_{\sigma})=1$ as desired. Let $\mathcal{S}$
be the set of discrete cube isomorphisms of $\{0,1\}^{k+1}$. Let
$F=\bigcap_{\sigma\in\mathcal{S}}F_{\sigma}$. As $\mathcal{S}$ is
finite, for any $c\in F$ and discrete cube isomorphism $\sigma$
one has
\begin{equation}
\rho(\sigma(c))=\sgn(\sigma)\cdot\rho(c)\text{ and }\mu_{C_G^{k+1}(X)}(F)=1.\label{eq:glueing set full measure}
\end{equation}

We now show that $\rho$ satisfies the glueing property, i.e.\  for a.e.
$\mu_{\mathcal{P}^{k+1}(X)}$ $(c_{1},c_{2})$, one has $\rho(c_{1}\bigparallel c_{2})=\rho(c_{1})+\rho(c_{2})$. We will be using freely notation from Lemma \ref{lem:key joining} and define $\tilde{\pi}_L,\tilde{\pi}_U:(Y^{k+1}\times Y^{k+1},\lambda)\rightarrow (Y^{k+1},\mu^{[k+1]})$ by $\tilde{\pi}_L(d_1,d_2)=d_1,\tilde{\pi}_U(d_1,d_2)=d_2$.

Let $D'=\tilde{\pi}_{L}^{-1}(D)\cap\tilde{\pi}_{U}^{-1}(D)\cap\pi_{E}^{-1}(D)\cap S$.
By the properties of $\lambda$, one has $\lambda(D')=1$. Notice that for any $((c_{1},\mathbf{a}_{1}),(c_{2},\mathbf{a}_{2}))\in S$,
 $$ \pi_{E}((c_{1},\mathbf{a}_{1}),(c_{2},\mathbf{a}_{2}))=(c_{1}\bigparallel c_{2},\mathbf{a}_{1}\bigparallel\mathbf{a}_{2}).$$
As $(\pi_E)_*\lambda=\mu^{[k+1]}$, $\rho(c_{1}\bigparallel c_{2})=\theta_{k+1}(\mathbf{a}_{1}\bigparallel\mathbf{a}_{2})$
for $\lambda$-a.e. $((c_{1},\mathbf{a}_{1}),(c_{2},\mathbf{a}_{2})).$

Since the factor map $\pi$ sends glueable pairs to glueable pairs,
$(c_{1},c_{2})\in\mathcal{P}^{k+1}(X)$ for $((c_{1},\mathbf{a}_{1}),(c_{2},\mathbf{a}_{2}))\in D'$.
For any $((c_{1},\mathbf{a}_{1}),(c_{2},\mathbf{a}_{2}))\in\pi_{L}^{-1}(D)\cap\pi_{U}^{-1}(D)$,
$\rho(c_{1})+\rho(c_{2})=\theta_{k+1}(\mathbf{a}_{1})+\theta_{k+1}(\mathbf{a}_{2})$;
For any $((c_{1},\mathbf{a}_{1}),(c_{2},\mathbf{a}_{2}))\in\pi_{E}^{-1}(D)\cap S$,
$\rho(c_{1}\bigparallel c_{2})=\theta_{k+1}(\mathbf{a}_{1}\bigparallel\mathbf{a}_{2})$.
For $((c_{1},\mathbf{a}_{1}),(c_{2},\mathbf{a}_{2}))\in S$, one has
$\pi_{U}((c_{1},\mathbf{a}_{1}))=\pi_{L}((c_{2},\mathbf{a}_{2}))$. Therefore 
$$\sum_{v\in\{0,1\}^{k}}(-1)^{|v|}(\mathbf{a}_{1}){}_{(v,1)}=\sum_{v\in\{0,1\}^{k}}(-1)^{|v|}(\mathbf{a}_{2}){}_{(v,0)}.$$
Notice that for $i=1,2$,
$$\theta_{k+1}(\mathbf{a}_{i})=\sum_{v\in\{0,1\}^{k}}(-1)^{|v|}(\mathbf{a}_{i}){}_{(v,0)}+\sum_{v\in\{0,1\}^{k}}(-1)^{|v|+1}(\mathbf{a}_{i}){}_{(v,1)}.$$
Therefore one has that
\[
\begin{array}{l}
\theta_{k+1}(\mathbf{a}_{1})+\theta_{k+1}(\mathbf{a}_{2})\\
=\sum_{v\in\{0,1\}^{k}}(-1)^{|v|}(\mathbf{a}_{1}){}_{(v,0)}+\sum_{v\in\{0,1\}^{k}}(-1)^{|v|+1}(\mathbf{a}_{2}){}_{(v,1)}\\
=\theta_{k+1}(\mathbf{a}_{1}\bigparallel\mathbf{a}_{2}).
\end{array}
\]
Thus for arbitrary $((c_{1},\mathbf{a}_{1}),(c_{2},\mathbf{a}_{2}))\in D',$
$\rho(c_{1}\bigparallel c_{2})=\rho(c_{1})+\rho(c_{2})$. From Lemma \ref{lem:key joining}, $(\pi_{k}^{[k+1]}\times\pi_{k}^{[k+1]})_{*}\lambda=\mu_{\mathcal{P}^{k+1}(X)}$. Let $U=\{(c_{1}, c_{2})\in \mathcal{P}^{k+1}(X): \rho(c_{1}\bigparallel c_{2})=\rho(c_{1})+\rho(c_{2})\}$. One has that $U$ is $\mu_{\mathcal{P}^{k+1}(X)}$-measurable and $D'\subset (\pi_{k}^{[k+1]}\times\pi_{k}^{[k+1]})^{-1}U$.  Therefore $\mu_{\mathcal{P}^{k+1}(X)}(U)=1$
 and for $\mu_{\mathcal{P}^{k+1}(X)}$-a.e. $(c_1,c_2)$, $\rho(c_{1}\bigparallel c_{2})=\rho(c_{1})+\rho(c_{2})$.
\end{proof}

\specialsectioning

\section{Glossary}

\subsection*{Spaces and sets.}
\begin{lyxlist}{00.00.0000}
\item [{$(X,G)$}] Topological dynamical system.
\item [$G$] Finitely generated abelian group with identity element $\id$.
\item [$A$] Compact metrizable abelian group with identity element $0_A$.

\item [{$Z_k(X)$}] Definition \ref{def:corner}, Page \pageref{def:corner}.
\item [{$[k]=\{0,1\}^{k}$.}] ~
\item [{$X^{[k]}=X^{\{0,1\}^{k}}.$}] ~
\item [$\mathcal{B}^{[k]}$] The Borel $\sigma$-algebra of $X^{[k]}$.
\item [{$[k]^{*}=\{0,1\}_{*}^{k}=\{0,1\}^{k}\setminus\{\vec{0}\}$.}] ~
\item [{$X^{[k]*}=X^{\{0,1\}^{k}\setminus\{\vec{0}\}}$.}] ~
\item [{$\ensuremath{\mathcal{HK}^{n}(G)}$}] The Host-Kra
cube group, Section \ref{subsec:Definition-of-the mu k}, Page \pageref{subsec:Definition-of-the mu k}.
\item [{$\mathcal{F}^{n}(G)$}] The face
group, Section \ref{subsec:Definition-of-the mu k}, Page \pageref{subsec:Definition-of-the mu k}.
\item [$C_G^{k}(X)=\overline{\{\mathbf{g}\mathbf{x}:\mathbf{x}=(x,\ldots,x)\in X^{[k]},\mathbf{g}\in\mathcal{HK}^{k}(G)\}},$] Definition \ref{def: k- equivalent}, Page \pageref{def: k- equivalent}.
\item [{$\text{Hom}(V,X)$}] The set of maps $\alpha:V\rightarrow X$
such that for all $v\in V$, $\alpha|_{\{v':\ v'\subset v\}}$ is
a cube of $X$. Section \ref{subsec:Cube-space,-nilspace},
Page \pageref{subsec:Cube-space,-nilspace}.
\item [{$(X,C^{\bullet}(X))$}] Cubespace, Section \ref{subsec:Cube-space,-nilspace},
Page \pageref{subsec:Cube-space,-nilspace}.
\item [{$C_G^{k+1}(X)\times_{X}C_G^{k+1}(X)=\{(c_{1},c_{2})\in C_G^{k+1}(X)\times C_G^{k+1}(X):(c_{1})_{\vec{0}}=(c_{2})_{\vec{0}}\}$.}] ~
\item [{$C_G^{k+1}(X)\times_{X}^vC_G^{k+1}(X)=\{(c_{1},c_{2})\in C_G^{k+1}(X)\times C_G^{k+1}(X):(c_{1})_{v}=(c_{2})_{\vec{0}}\}$.}] ~
\item [{$X'$}] Subsection \ref{subsec:Definitions}, Page
\pageref{subsec:Definitions}.
\item [{$\mathcal{P}^{k}(X),\mathcal{P}^{k}(G)$}] Space of glueable pairs, Definition
\ref{lem:RG is conditional product}, Page \pageref{lem:RG is conditional product}.
\item [{$T^{n}(X)$}] Tricube, Section \ref{subsec:The-tri-cube(3-cube)},
Page \pageref{subsec:The-tri-cube(3-cube)}.
\item [{$T^n(G)$}] Tricube group, Section \ref{subsec:The-tri-cube(3-cube)},
Page \pageref{subsec:The-tri-cube(3-cube)}.
\item [{$L$}] Definition \ref{def:nilcycle}, Page
\pageref{def:nilcycle}.
\item [{$\tilde{L}=\{\mathbf{a}\in A^{[k+1]}:\sum_{v\in\{0,1\}^{k+1}}(-1)^{|v|}\mathbf{a}_{v}=0\}$.}]~
\item [{$\mathcal{I}_{G}^{[k]}$}] The $\triangle_{G}^{[k]}$-invariant
$\sigma$-algebra.
\item [{$\mathcal{J}_{*}^{[k]}(X)$}] The $\sigma$-algebra of sets invariant
under $\mathcal{F}^{k}(G)$ on $X_{*}^{[k]}$.
\item [{$\mathcal{L}(X\xrightarrow{\pi}Y,A)$}] Function bundle, Section
\ref{subsec:Function-bundles}, Page \pageref{subsec:Function-bundles}.
\item [{$\mathcal{S}_{k}$}] The group of $k$-discrete cube isomorphisms.
\item [{$M'=\{-\rho_{x}+a:\ x\in X',a\in A\}$}]  $\rho_{x}=\rho|_{\mathfrak{p_{0}}^{-1}(x)}$, Subsection \ref{subsec:Definitions}, Page
\pageref{subsec:Definitions}.
\item [{$M=\overline{M'}$.}] Subsection \ref{subsec:Definitions}, Page
\pageref{subsec:Definitions}.
\item [{$\mathcal{M}(X),\,  \mathcal{M}_G(X)$}] The set of ($G$-invariant) Borel probability measures of
$X$, Section \ref{subsec:Dynamical-background.}, Page
\pageref{subsec:Dynamical-background.}.
\item [{$\nrp_G^{[k]}(X),\nrp_G^{[k]}(M\rightarrow X)$}] Definition \ref{NRP(M-G)},
Page \pageref{NRP(M-G)}.
\item [{$B^{*}=\{\mathbf{x}^{*}:\ (x_{0},\mathbf{x}^{*})\in B\text{ for some }x_{0}\in X\}$,}] where $B\subset X^{[k]}$.
\item [$Q^{k+1}_U(M)$] Definition \ref{definition of Q(M)},
Page \pageref{definition of Q(M)}.
\end{lyxlist}

\subsection*{Symbols.}
\begin{lyxlist}{00.00.0000}
\item [{$\overline{\alpha}_{j}=\{v\in\{0,1\}^{k}:v_j=1\}$}] The $j$-th
upper face of $\{0,1\}^{k}$.
\item [{$\underline{\alpha}_{j}=\{v\in\{0,1\}^{k}:v_j=0\}$}] The $j$-th
lower face of $\{0,1\}^{k}$.
\item [{$x^{[k]}=(x,\ldots,x)\in X^{[k]}$}] The diagonal
element in $X^{[k]}$. Similarly one defines ${g}^{[k]}$.
\item [{$\triangle(X)=\triangle=\{(x,x)|\,x\in X\}$.}]~
\item [{$\triangle_{G}^{[k]}=\{{g}^{[k]}:g\in G\}$.}]~
\item [{$t^{\overline{\alpha}_{j}}(v)=\begin{cases}
t & v\in\overline{\alpha}_{j}\\
\id& v\notin\overline{\alpha}_{j}
\end{cases}$.}] ~
\item [{$v_j\text{ for }v\in\{0,1\}^{k}$}] The $j$-th coordinate of
$v$.
\item [{$\mathbf{x}_{v}\text{ for \ensuremath{\mathbf{x}}}\in C^{k}(X)$}] The
$v$-coordinate of $\mathbf{x}\in C^{k}(X)$. Similarly, $\mathbf{t}_{v}$
is the $v$-coordinate of $\mathbf{t}\in T^{k}(X)$.
\item [{$|v|\text{ for }v\in\{0,1\}^{k}$}] $|v|=\#\{j:v_j=1\}$.
\item [{$\bigparallel$}] Glueing; Definition \ref{lem:RG is conditional product},
Page \pageref{lem:RG is conditional product}.
\item [{$[f]$}] The equivalence class in function bundle, Definition \ref{subsec:Function-bundles},
Page \pageref{subsec:Function-bundles}.
\item [{$d_{A}$}] A compatible metric on $A$.
\item [{$\sim_k$}] Definition \ref{def: k- equivalent}, Page
\pageref{def: k- equivalent}.
\item [{$\vec{1}=(1,\ldots,1).$}]

\end{lyxlist}

\subsection*{Maps.}
\begin{lyxlist}{00.00.0000}
\item [{$\sgn(\sigma)$}] Section \ref{subsec:Cube-space,-nilspace}, Page
\pageref{subsec:Cube-space,-nilspace}.
\item [{$\rho\colon C_G^{k+1}(X)\to A$}] Nilcycle, Definition \ref{def:nilcycle}, Page
\pageref{def:nilcycle}.
\item [{$\rho_{x}:\mathfrak{p}_0^{-1}(x)\rightarrow A$}] 
Subsection \ref{subsec:Definitions}, Page \pageref{subsec:Definitions}. Abbreviated as $\rho=\rho_x$ in Lemma \ref{fact:well defined and continuous}, Page \pageref{fact:well defined and continuous}.
\item [{$\beta$}] Cocycle, Section \ref{subsec:Dynamical-background.}, Page \pageref{subsec:Dynamical-background.} and Subsection \ref{sec: continuous}, Page \pageref{sec: continuous}.
\item [{$\pi_{k}:Y\rightarrow X$}] Section \ref{subsec:Construction-of-a nilcycle from the HK extension}, Page \pageref{subsec:Construction-of-a nilcycle from the HK extension}.
\item [{$f^{[k]}:Y^{[k]}\rightarrow X^{[k]}$}] Product map for $f:Y\rightarrow X$.
\item [{$\mathfrak{p}^{*}$}] The projection from $X^{[k]}$ to $X^{[k]*}$.
\item [{$\theta_{n},\theta:A^{[n]}\rightarrow A\text{ for abelian group }A$}] $\theta_{n}(\mathbf{a})=\sum_{v\in\{0,1\}^{n}}(-1)^{|v|}\mathbf{a}_{v}$, Definition \ref{def:nilcycle}, Page \pageref{def:nilcycle}.
Sometimes we use $\theta (\mathbf{a})$ if there no confusion arises .
\item [{$\mathfrak{p}_{0},\mathfrak{p}_{0,n}:C_G^{n}(X)\rightarrow X:\ \mathbf{x}\rightarrow\mathbf{x}_{\vec{0}}$}] The
projection from $C_G^{n}(X)$ to the first coordinate. Subsection \ref{subsec:Definitions}, Page
\pageref{subsec:Definitions}. This projection
also gives a natural map:
\item [{$\hat{\mathfrak{p}}_{0},\hat{\mathfrak{p}}_{0,n}:\mathcal{L}(C_G^{n}(X)\xrightarrow{\mathfrak{p}_{0,n}}X,A)\rightarrow X$},] Section
\ref{subsec:Function-bundles}, Page \pageref{subsec:Function-bundles}.
\item [{$\mathfrak{p}_{v},\mathfrak{p}_{v,n}:C_G^{n}(X)\rightarrow X$}] $\mathfrak{p}_{v,n}(c)=c_{v}$. Subsection \ref{subsec:Definitions}, Page
\pageref{subsec:Definitions}.
\item[$\hat{\pi}:\mathcal{L}(X\xrightarrow{\pi}Y,A)\rightarrow Y$] Definition \ref{def:Let--be function bundle}, Page \pageref{def:Let--be function bundle}.
\item [{$\pi':C_G^{k+1}(X)\times_{X}C_G^{k+1}(X)\rightarrow X$}] $\pi'((c_{1},c_{2}))=\mathfrak{p}_{0}(c_{1})=\mathfrak{p}_{0}(c_{2})$.
\item [{$\tilde{\pi}:T^{k+1}(X)\rightarrow C_G^{k+1}(X)\times_{X}C_G^{k+1}(X)$}] $\tilde{\pi}(\mathbf{t})=(\omega(\mathbf{t}),\psi_{\vec{0}}(\mathbf{t}))$.
\item [$\tilde{\pi}_v,\tilde{\phi}_v:\mathbf{t}\rightarrow(\omega(\mathbf{t}),\psi_{v}(\mathbf{t}))$] Definition \ref{def:definition of the product space}, Page \pageref{def:definition of the product space}.
\item [{$\mathfrak{q}_{v}:C_G^{k+1}(M)\rightarrow M$}] The projection from
$C_G^{k+1}(M)$ to the $v$ coordinate for $v\in\{0,1\}^{k+1}$.
\item [{$\pi_{L},\pi_{U}:X^{\{0,1\}^{k+1}}\rightarrow X^{\{0,1\}^{k}}$}] $\pi_L(c_1,c_2)=c_1$, $\pi_U(c_1,c_2)=c_2$, Definition
\ref{lem:RG is conditional product}, Page \pageref{lem:RG is conditional product}.
\item [$\pi_{L,v}, \phi_{L,v}$] Definition \ref{def:definition of the product space}, Page \pageref{def:definition of the product space}. 
\item [{$\Psi_{v},\Omega$}] Section \ref{subsec:The-tri-cube(3-cube)},
Page \pageref{subsec:The-tri-cube(3-cube)}.
\item [{$\psi_{v},\omega$}] Section \ref{subsec:The-tri-cube(3-cube)},
Page \pageref{subsec:The-tri-cube(3-cube)}.
\item [{$\pi_{T}:T^n(X)\rightarrow X$}] $\pi_{T}(\mathbf{t})=\mathbf{t}_{\vec{1}}$.
\item [{$\phi_{F_{1},F_{2}}:f\rightarrow\int_{\pi^{-1}(\hat{\pi}(f))}F_{1}(f(v))F_{2}(v)d\mu_{\pi}^{\hat{\pi}(f)}(v)$}] Section
\ref{subsec:Function-bundles}, Page \pageref{subsec:Function-bundles}.
\item [$D,i,g,g'$] Lemma \ref{lem:The-map-continuous}, Page \pageref{lem:The-map-continuous}.
\item [$p_1,p_2$] Lemma \ref{ergodic L}, Page \pageref{ergodic L}.
\item [{$e(x)=e^{2\pi i x}.$}] 
\item [{$e([c,d])=\{e(x):x\in [c,d]\}.$}]
\end{lyxlist}

\subsection*{Measures. }
\begin{lyxlist}{00.00.0000}
\item [{$\delta_{z}$}] The delta measure of $z\in Z$ for some space
$Z$.
\item [{$\mu$}] The $G$-unique ergodic measure of $X$.
\item [{$\nu=\mu\times m_{\haar(A)}$}] Introduction, Page \pageref{thm:main theorem}.
\item [{$m_{\haar(A)}$}] The Haar measure of a (compact) group $A$. If $H\subset A$ is a closed subgroup and $a\in A$, then we define the measure $m_{\haar(H)}+a$ by $(m_{\haar(H)}+a)(B)=m_{\haar(H)}((B+a)\cap H)$ for $B\subset A$.
\item [{$\mu_{k}$}] The measure on $Z_{k}$.
\item [{$\mu^{[k]}$}] The $k$-th Host-Kra cube
measure, Page \pageref{subsec:Definition-of-the mu k}.
\item [{$\mu_{C_G^{k}(X)}$}] The unique ergodic
measure of $(C_G^{k}(X),\mathcal{HK}^{k}(G))$, Definition \ref{k-cube uniquely ergodic},
Page \pageref{k-cube uniquely ergodic}.
\item [{$\mu_{\mathcal{P}^{k}(X)}$}] The unique
ergodic measure of $(\mathcal{P}^{k}(X),\mathcal{P}^{k}(G))$, Definition \ref{def:meaure on glueing pair},
Page \pageref{def:meaure on glueing pair}.
\item [{$\{\mu_{C_G^{k}(X)}^{x}\}_{x\in X}$}] The measure disintegration
of $\mu_{C_G^{k}(X)}$ w.r.t.\  $\mathfrak{p}_{0}$ (Definition \ref{CSM for p0}, Page \pageref{lem:CSM circ}).
\item [{$\mu_{C_G^{k}(X)\times_{X}C_G^{k}(X)}=\int_X\mu_{C_G^{k}(X)}^{x}\times\mu_{C_G^{k}(X)}^{x}d\mu(x)$}], Definition
\ref{definition of the measures for T(X) and CXC}, Page \pageref{definition of the measures for T(X) and CXC}.
\item [{$\mu_{T^{k}(X)}$}] The measure on $T^{k}(X)$ given by  Definition \ref{definition of the measures for T(X) and CXC},
Page \pageref{definition of the measures for T(X) and CXC}.
\item [{$\{\mu_{T^{k}(X)}^{x}\}_{x\in X}$}] The measure disintegration
of $\mu_{T^{k}(X)}$ w.r.t.\  $\pi_{T}$, Definition \ref{definition of eta and Txx}, Page \pageref{definition of eta and Txx}.
\item [{$\{\eta^{c}\}$}] The measure disintegration of $\mu_{T^{k}(X)}$
w.r.t.\  $\omega$, Definition \ref{definition of eta and Txx}, Page \pageref{definition of eta and Txx}.
\item [{$\{\mu^{[k+1](c)}\}$}] The measure disintegration of $\mu^{[k+1]}$
w.r.t.\  $\pi_{k}^{[k+1]}$.
\item [{$\{\mu_{C_G^{k+1}(M)}^{f}\}$}] The measure disintegration of $\mu_{C_G^{k+1}(M)}$
w.r.t.\  $\mathfrak{q}_{0}$. ~
\end{lyxlist}

\bibliographystyle{alpha}
\bibliography{references}

\Addresses
\end{document}